\documentclass[11pt]{amsart}

\usepackage{color,enumitem}
\usepackage{amsmath,amssymb,bbm}
\usepackage{amscd,tikz-cd}
\usepackage{mathrsfs}
\usepackage{amsthm}

\theoremstyle{plain}

\newtheorem{theorem}{\bf Theorem}[section]
\newtheorem{lemma}[theorem]{\bf Lemma}
\newtheorem{proposition}[theorem]{\bf Proposition}
\newtheorem{corollary}[theorem]{\bf Corollary}
\newtheorem{fact}[theorem]{\bf Fact}

\newtheorem{thmx}{Theorem}

\theoremstyle{definition}
\newtheorem{definition}[theorem]{\bf Definition}

\newtheorem{remark}[theorem]{\bf Remark}

\newcommand{\disp}{\displaystyle}

\newcommand{\eqa}[1]{
\begin{align*}
#1
\end{align*}}

\newcommand{\cel}[1]{
{\rm{cel}}(#1)
}
\newcommand{\bel}[1]{
{\rm el}_{#1}
}
\newcommand{\rbel}[1]{
{\rm rel}_{#1}
}

\newcommand{\Nat}{\mathbb{N}}
\newcommand{\Rea}{\mathbb{R}}
\newcommand{\Int}{\mathbb{Z}}
\newcommand{\Com}{\mathbb{C}}
\newcommand{\Hil}{\mathcal{H}}
\newcommand{\ri}{{\rm i}}

\newcommand{\Inv}{\rm{Inv}}
\newcommand{\LA}{\mathfrak{g}}
\newcommand{\Uni}{\mathcal{U}}
\newcommand{\U}{\mathcal{U}}
\newcommand{\T}{\mathbb{T}}
\newcommand{\Span}{\mathrm{span}}
\newcommand{\re}{\mathrm{Re}}
\newcommand{\Sl}{\langle}
\newcommand{\Sr}{\rangle}

\newcommand{\SL}{\mathrm{SL}}
\newcommand{\GL}{\mathrm{GL}}
\newcommand{\PM}{\mathcal{P}}
\newcommand{\norm}{\|\cdot\|}
\newcommand{\E}{\mathrm{E}}

\newcommand{\R}{\mathbb{R}}
\newcommand{\C}{\mathbb{C}}
\newcommand{\N}{\mathbb{N}}

\newcommand{\cer}{\mathrm{cer}}

\usepackage[all]{xy}

\title{Large scale geometry of Banach-Lie groups}
\author[H. Ando]{Hiroshi Ando}
\address{Department of Mathematics and Informatics\\
Chiba University\\
1-33 Yayoi-cho, Inage, Chiba\\
263-8522 Japan}
\email{hiroando@math.s.chiba-u.ac.jp}
\author[M. Doucha]{Michal Doucha}
\address{Institute of Mathematics\\
Czech Academy of Sciences\\
\v Zitn\'a 25\\
115 67 Praha 1\\
Czech Republic}
\email{doucha@math.cas.cz}
\author[Y. Matsuzawa]{Yasumichi Matsuzawa} 
\address{Department of Mathematics\\
Faculty of Education\\
Shinshu University\\
6-Ro, Nishi-nagano, Nagano\\
380-8544 Japan}
\email{myasu@shinshu-u.ac.jp}

\thanks{H. Ando is supported by JSPS KAKENHI 16K17608 and 20K03647.  M. Doucha is supported by the GA\v{C}R project 19-05271Y and RVO: 67985840. Y. Matsuzawa
is supported by JSPS KAKENHI 26800055 and 26350231.}
\begin{document}
\begin{abstract}
We initiate the large scale geometric study of Banach-Lie groups, especially of linear Banach-Lie groups. We show that the exponential length, originally introduced by Ringrose for unitary groups of $C^*$-algebras, defines the quasi-isometry type of any connected Banach-Lie group. As an illustrative example, we consider unitary groups of separable abelian unital $C^*$-algebras with spectrum having finitely many components, which we classify up to topological isomorphism and up to quasi-isometry, in order to highlight the difference. The main results then concern the Haagerup property, and Properties (T) and (FH). We present the first non-trivial non-abelian and non-localy compact groups having the Haagerup property, most of them being non-amenable. These are the groups $\U_2(M,\tau)$, where $M$ is a semifinite von Neumann algebra with a normal faithful semifinite trace $\tau$. Finally, we investigate the groups $\E_n(A)$, which are closed subgroups of $\GL(n,A)$ generated by elementary matrices, where $A$ is a unital Banach algebra. We show that for $n\geq 3$, all these groups have Property (T) and they are unbounded, so they have Property (FH) non-trivially. On the other hand, if $A$ is an infinite-dimensional unital $C^*$-algebra, then $\E_2(A)$ does not have the Haagerup property. If $A$ is moreover abelian and separable, then $\SL(2,A)$ does not have the Haagerup property.
\end{abstract}

\maketitle

\tableofcontents

\section{Introduction}
A principal topic of geometric group theory is to study groups geometrically as metric spaces, mainly from the large scale point of view, and to study group actions on geometrically interesting metric spaces, where proper actions, resp. actions with fixed points are of main interest. Although by the Birkhoff-Kakutani theorem, every first-countable Hausdorff topological group admits a compatible left-invariant metric, not every such a metric is geometrically interesting. There are two traditionally well studied classes of groups, finitely generated groups and connected Lie groups, that however admit certain canonical distances that have been the central objects of the geometric investigation of these groups. These are the word metrics on finitely generated groups and the left-invariant Riemannian distances on connected Lie groups. Both classes of distances are considered, for sound reasons, to define the quasi-isometry types of the corresponding groups, and make the notions of proper group actions meaningful.

It has been recently observed that large scale geometry can be well extended to bigger classes of groups. Cornulier and de la Harpe in \cite{CdH16} present a systematic large scale geometric study of locally compact groups, in particular clarify which locally compact groups admit metrics that well-define their quasi-isometry type, resp. their coarse type. More or less at the same time, Rosendal in \cite{Ro18} suggested to study coarse geometry of general topological groups. He discovered that the notion of a \emph{coarse structure}, as defined by Roe (see \cite{Roe03}), can be defined on any topological group. This in particular made possible to investigate not necessarily locally compact groups which have well-defined quasi-isometry type in a similar vein as for locally compact groups.

We remark that many important notions connecting group actions with analysis, such as the Haagerup property (also known as, or equivalent to, a-T-menability), i.e. having a metrically proper continuous action on a Hilbert space by affine isometries, or analogous notions for other, mainly $L^p$, Banach spaces, are now meaningful for all topological groups. The opposite properties, i.e. fixed point properties on various metric spaces (in particular on Hilbert spaces which is for $\sigma$-compact locally compact groups equivalent to the famous Property (T), or other $L^p$-spaces) can be formulated and investigated even without considering a coarse structure on a group, but when it is at our disposal, they can be appreciated somewhat better, as we shall see. \\

The goal of this paper is to initiate the study of these problems on Banach-Lie groups, especially on linear Banach-Lie groups. Considering the prominence of Lie groups in geometric group theory and realizing that they have been the main and most natural source of examples of groups having properties mentioned above, e.g. the Haagerup property, the fixed point properties on $L^p$-spaces, Property (T) (indeed, the first examples of non-locally finite groups with Property (T) were Banach-Lie groups \cite{Sha99}), we believe that Banach-Lie groups should play a similar prominence in the non-locally compact case. The task of this paper is to justify this belief. Mostly we will be concerned with linear Banach-Lie groups and their closed subgroups, i.e. closed subgroups of $\GL(n,A)$, where $A$ is a unital Banach algebra. Although in the general theory of infinite-dimensional Lie groups, this is a very special class, for our purposes it is already quite rich and it connects our research with operator algebras and topological $K$-theory, where these groups naturally appear. We refer to \cite{Rob19}, \cite{Va86}, and references therein for the research on the (mostly normal subgroup) structure of linear Banach-Lie groups.

First we realize that every connected Banach-Lie group has a well-defined quasi-isometry type. We show that there are several candidates for distances realizing the large scale geometry (all of them of course quasi-isometrically the same). The standard distance on Banach-Lie groups, generalizing the Riemanninan distance, is the Finsler distance, and metric geometry of Banach-Lie groups with this distance is an established topic of research (see e.g. \cite{Lar19b} and \cite{ALR10}, and references therein). Indeed, the Finsler distance defines the quasi-isometry type (it is a \emph{maximal metric} in the sense of Rosendal \cite{Ro18}). However, in this paper we will promote another distance that has its origin in $C^*$-algebra theory and was defined by Ringrose in \cite{Ri92} as a \emph{$C^*$-exponential length} for unitary groups of $C^*$-algebras. He also proved that for unitary groups it actually coincides with the Finsler distance. On the other hand, the connection between the exponential length and the large scale geometry of the groups $\mathcal{U}_0(A)$ has not been found until recently. 
In \cite{AM20}, the authors found the first connection that $\mathcal{U}_0(A)$ is bounded if and only if $A$ has finite exponential length. Thus establishing the (in)finiteness of exponential length is of importance for our purpose. After the first version of the paper appeared, C. Rosendal asked us whether there exists a \emph{minimal} metric on every Banach-Lie group (see \cite{Ro18-2}). The answer is positive and the next theorem summarizes our findings in this direction.
\begin{thmx}
Let $G$ be a connected Banach-Lie group with a Banach-Lie algebra $\LA$. Then the \emph{exponential length} of $G$, given by the  following formula
\[\bel{G}(g):=\inf \left \{\sum_{i=1}^n\|X_i\|\,\middle|\, g=\exp (X_1)\cdots \exp (X_n),\,X_i\in \mathfrak{g}\,(1\le i\le n)\right \}\]
is a compatible length function on $G$, where it defines its quasi-isometry type and simultaneously it induces a minimal metric. Thus it defines a global Lipschitz geometric structure on $G$.
\end{thmx}

The advantage of the exponential length to the Finsler distance, in our view, is that it is more convenient in computations, especially for linear Banach-Lie groups, which will be demonstrated several times throughout the paper. Indeed, in some cases the exponential length can be computed explicitly. Since having a well-defined quasi-isometry type does not mean that this type is not trivial, i.e. that the group is not quasi-isometric to a point, we provide some criteria for verifying that a given Banach-Lie group is unbounded. This is an important research topic in $C^*$-algebra theory for unitary groups (see e.g. \cite{Ph94}, \cite{Zhang93}, \cite{Lin14}, and references therein).

These new concepts are then illustrated on our `toy examples', the \emph{unitary groups of unital abelian $C^*$-algebras} (in the sequel, we shall use the shorthand \emph{abelian unitary groups} for them). We explicitly compute their exponential length. For separable algebras, we show that whenever the Gelfand spectrum of the algebra is not totally disconnected (otherwise, the unitary group is bounded, see \cite{AM20}), then the corresponding unitary group is quasi-isometrically universal for separable metric spaces. For separable abelian unital $C^*$-algebras whose Gelfand spectrum has finitely many components, we provide a classification of their unitary groups up to topological isomorphism and up to quasi-isometry, highlighting the difference.

Next we focus on Properties (T) and (FH), and the Haagerup property. We provide, to the best of our knowledge, the first non-trivial examples of non-abelian non-locally compact groups with the Haagerup property, most of them even being non-amenable (we note that for non-locally compact groups, amenability does not imply the Haagerup property, see \cite{Ros17}, where the Haagerup property for non-locally compact groups was considered for the first time).\medskip

\begin{thmx} Let $M$ be a semifinite von Neumann algebra acting on a separable Hilbert space with infinite direct summand and let $\tau$ be a normal faithful semifinite trace $\tau$ on $M$. Let $1\leq p<\infty$ and let $\U_p(M,\tau)$ be the $p$-Schatten unitary group associated to $M$. Then $\U_p(M,\tau)$ is an unbounded Polish group acting properly on $L^p(M,\tau)$.

In particular, $\U_2(M,\tau)$ has the Haagerup property.\\

Moreover, if $M$ is a factor, then $\U_p(M,\tau)$ is a Banach-Lie group if and only if $M$ is of type {\rm I}. 
If $M$ is moreover a type ${\rm{II}}_{\infty}$ factor, then the following three conditions are equivalent. 
\begin{itemize}
    \item[{\rm{(i)}}] $\U_p(M,\tau)$ is amenable for some $1\le p<\infty$.
    \item[{\rm{(ii)}}] $\U_p(M,\tau)$ is amenable for every $1\le p<\infty$.
    \item[{\rm{(iii)}}] $M$ is hyperfinite. 
\end{itemize}
\end{thmx}
\medskip

After the first version of the paper appeared, B. Duchesne suggested to us other natural candidates of non-locally compact groups with the Haagerup property: the isometry groups of the separable infinite-dimensional Hilbert and real hyperbolic spaces (with their point-wise convergency topology or topology of uniform convergence on bounded sets). Let us comment that the former, denoted by $G$, with either topology, can be written as $(\Hil,+)\rtimes O(\Hil)$, where $O(\Hil)$ is the orthogonal group and as a subgroup of $G$ it is equipped either with the SOT or the norm topology, depending on the topology of $G$. It is easy to check that the canonical action of $G$ on $\Hil$ is metrically proper if and only if the subgroup $O(\Hil)$ is bounded. In that case, we moreover get that $G$ is coarsely equivalent to the additive group $(\Hil,+)$. Since it is known that $O(\Hil)$ with both topologies is indeed bounded, we get that $G$ with both topologies is a group with the Haagerup property, however the action is only a mild twist of the translation action of $(\Hil,+)$.

We do not know about the coarse geometry of the isometry group of the infinite-dimensional real hyperbolic space (we refer the reader to \cite{Duch20} and its sequel in preparation where this group is investigated).\medskip

If $A$ is any unital (real or complex) Banach algebra and $n\geq 2$, denote by $\E(n,A)$ the closed subgroup of $\GL(n,A)$, which is the Banach-Lie group of invertible elements of the Banach algebra $M_n(A)$ with the norm topology, generated by the elementary matrices.\medskip

\begin{thmx} Let $A$ be a unital Banach algebra, not necessarily abelian and separable, and let $n\geq 3$. Then $\E_n(A)$ is an unbounded group with Properties (T) and (FH).

Moreover, if $A$ is abelian, then $E_n(A)$ is the connected component of the identity of $\SL(n,A)$, thus $\SL(n,A)$ has Property (T) if and only if $\SL(n,A)/\E_n(A)$ does. In particular, if $A$ is a $C^*$-algebra and $X$ is its Gelfand spectrum, then $\SL(n,A)$ has Property (T) if and only if the group of homotopy classes of maps $[X,\mathrm{SU}(n)]$ does.
\end{thmx}

This provides also a large influx of new examples of non-locally compact groups having Property (FH) `non-trivially'. Indeed, most of the known examples of non-locally compact groups having Property (FH) are bounded, i.e. they only have bounded orbits when acting on any metric space.
\begin{thmx}
On the other hand, if $n=2$ and $A$ is an infinite-dimensional unital $C^*$-algebra, then (contrary to the finite-dimensional case) $\E_2(A)$ does not have the Haagerup property. If $A$ is moreover separable and abelian, then $\SL(2,A)$ does not have the Haagerup property.
\end{thmx}

The paper is organized as follows. In Section~\ref{sect:prelim}, we review the background material on large scale geometry (of groups), on Properties (T), (FH), and the Haagerup property, on Banach-Lie groups, and on von Neumann algebras. The subsections on Property (T) and the Haagerup property contain
several new observations, e.g. equating Property (FH) with a \emph{weak Property (T)} for certain class of topological groups.

In Section~\ref{section:explength}, we introduce the main new notions of the paper, the exponential length on Banach-Lie groups. We present its basic properties, compare it with other distances on Lie groups, and connect it with the $C^*$-exponential length. Generalizing certain results of Phillips (mostly from \cite{Ph95}) on $C^*$-exponential length, we introduce methods for showing that the exponential length is unbounded. These new notions are then illustrated in Section~\ref{section:abelian} on abelian unitary groups, for which we compute the exponential length explicitly. We also classify unitary groups of separable abelian unital $C^*$-algebras whose Gelfand spectrum has finitely many components, up to topological isomorphism and up to quasi-isometry.

Then in Section~\ref{section:Up} we take up the study of the $p$-Schatten unitary groups of a semifinite von Neumann algebras and prove the theorem stated for them above. Finally, in Section~\ref{section:En} we study the groups $\E_n(A)$. We show that they are always unbounded groups. We do not know whether they are in general Banach-Lie, we do know it when $A$ is abelian or when $K_0(A)$ is finitely generated. We then prove the second theorem stated above.

\section{Preliminaries}\label{sect:prelim}
This section mostly reviews some known notions and facts that will be crucial for the rest of the paper, but it also contains several new observations (about Properties (T) and (FH) for Polish groups). We cover here large scale geometry of topological groups, Banach-Lie groups, Properties (T) and (FH), the Haagerup property, and von Neumann algebras and noncommutative integration theory.
\subsection{Large scale geometry of topological groups}
Let us start with some basic notions of large scale geometry. The reader is referred to \cite{Roe03} or \cite{NoYu12} for more information.

Let $(X,d_X)$ and $(Y,d_Y)$ be metric spaces. A map $f:X\rightarrow Y$ is called a \emph{quasi-isometric embedding} if there are constants $K,L>0$ such that for all $x,y\in X$ $$\frac{1}{K}d_X(x,y)-L\leq d_Y(f(x),f(y))\leq Kd_X(x,y)+L.$$
If moreover the image $f[X]\subseteq Y$ is a \emph{net}, i.e. there is $\varepsilon>0$ such that for all $y\in Y$ there is $x\in f[X]$ with $d_Y(x,y)<\varepsilon$, then $f$ is called a \emph{quasi-isometry} and the spaces $X$ and $Y$ are \emph{quasi-isometric}.

A strict weakening of the notion of quasi-isometry is the notion of \emph{coarse equivalence}. A map $f:X\rightarrow Y$ between metric spaces as above is called a \emph{coarse embedding} if there exist moduli $\rho_1:[0,\infty)\rightarrow [0,\infty)$ and $\rho_2:[0,\infty)\rightarrow [0,\infty)$, which are non-decreasing and satisfying $\lim_{t\to\infty} \rho_1(t)=\infty$, and such that for all $x,y\in X$ $$\rho_1(d_X(x,y))\leq d_Y(f(x),f(y))\leq \rho_2(d_X(x,y)).$$ If in addition again the image $f[X]\subseteq Y$ is a net, then $f$ is called a \emph{coarse equivalence} and the spaces $X$ and $Y$ are called \emph{coarse equivalent}.

It is clear that the notions of quasi-isometry and coarse equivalence make sense also for pseudometric spaces (i.e. spaces for which the distance function $d$ may vanish also outside of the diagonal), an observation that will be useful below.\medskip

Next we review some key ideas of the recent monograph \cite{Ro18} of Rosendal on large scale geometry of topological groups. Based on the ideas of Roe (\cite{Roe03}), Rosendal notices that it is possible to define a coarse structure for all topological groups. Indeed, a subset $A\subseteq G$ of a topological group is \emph{coarsely bounded} if it is bounded with respect to every continuous left-invariant pseudometric on $G$. Otherwise, it is \emph{coarsely unbounded}.

In this vein, the most important metrizable groups are those for which there exists a compatible left-invariant metric which metrizes its coarse structure. This is well known to be the case e.g. for countable discrete groups as one can always find a proper left-invariant metric on such groups. Even better, there are groups for which one can well-define their quasi-isometry type by finding a compatible left-invariant metric which is maximal (in a sense defined below). This is well known to be the case for finitely generated groups (with their word metrics) and for connected real Lie groups (with their left-invariant Riemannian metrics). Rosendal's discovery is that there are actually also many non-locally compact groups which admit such a metric. We start with precise definitions.
\begin{definition}{\cite[Definition 2.49]{Ro18}}
Let $G$ be a topological group and let $\PM$ be the set of all continuous left-invariant pseudometrics on $G$. Define a partial order relation $\ll$ on $\PM$ by setting, for $d,p\in\PM$, $d\ll p$ if there exist constants $K,L>0$ such that $d\leq Kp+L$.

A left-invariant pseudometric $d\in\PM$ is \emph{maximal} if it is maximal element in $(\PM,\ll)$. 
\end{definition}
It is clear that if $d,p\in\PM$ are both maximal, then they are quasi-isometric (i.e. the spaces $(G,d)$ and $(G,p)$ are quasi-isometric). Moreover, if $\PM$ admits a maximal element $d\in\PM$, then $d$ defines the quasi-isometry type of $G$ (see \cite[Section 2.7]{Ro18} for details).

We state the following criterion that is useful in verifying that a given metric is maximal.
\begin{proposition}{\cite[Proposition 2.52]{Ro18}}\label{prop:maximalmetric}
Let $G$ be a topological group and $d$ a continuous left-invariant pseudometric on $G$. Then the following are equivalent.
\begin{enumerate}
    \item $d$ is maximal.
    \item $d$ is coarsely proper and large scale geodesic.
\end{enumerate}
\end{proposition}
Now we own the definitions of \emph{coarse properness} and \emph{large scale geodesicity}.
\begin{definition}\label{def:coarselyproper}
A continuous left-invariant pseudometric on a topological group $G$ is \emph{coarsely proper} if it generates the left-coarse structure on $G$ (see \cite[Definition 2.4]{Ro18}). By \cite[Lemma 2.40]{Ro18}, this is equivalent, assuming that $d$ is actually a compatible metric and that $G$ has no proper open subgroups, to the requirement that for all $\Delta,\delta>0$ there is $k\in\Nat$ such that for any $d(g,e)<\Delta$ there are elements $g_0=e,\ldots,g_k=g\in G$ satisfying $d(g_i,g_{i+1})<\delta$, for all $0\leq i<k$
\end{definition}
\begin{definition}\label{def:largescalegeodesic}
$(G,d)$ as above is \emph{large scale geodesic} if there is $K>0$ such that for all $g,h\in G$ there are $g_0=g,\ldots,g_n=h$ satisfying $d(g,h)\leq K\sum_{i=0}^{n-1} d(g_i,g_{i+1})$, and for all $0\leq i<n$, $d(g_i,g_{i+1})\leq K$.
\end{definition}

Let $(M,d)$ be a metric space (or in general, just a pseudometric space). Recall that an \emph{intrinsic pseudometric} $d_I$ is defined by $$d_I(x,y):=\inf\left \{L(\gamma)\mid \gamma:[0,1]\rightarrow M\text{ continuous}, \gamma(0)=x,\gamma(1)=y\right \},$$ for $x,y\in M$, where $$L(\gamma):=\sup\left \{\sum_{i=1}^n d(\gamma(x_i),\gamma(x_{i+1})\,\middle|\, x_1=0<x_2<\ldots<x_{n+1}=1\right \}.$$
Notice that $d_I$ is symmetric and satisfies the triangle inequality, however it can attain infinity. In general, it follows from triangle inequalities that $d\leq d_I$. If we have $d=d_I$, then we say that $M$ is a length space.

The following simple lemma will be very useful.
\begin{lemma}\label{lem:maximallengthmetric}
Let $G$ be a topological group without proper open subgroups and let $d$ be a left-invariant compatible metric such that $(G,d)$ is a length space. Then $d$ is maximal.
\end{lemma}
We note that \cite[Example 2.54]{Ro18} states that every compatible left-invariant geodesic metric is maximal. Lemma~\ref{lem:maximallengthmetric} is just a mild generalization.
\begin{proof}
It is easy to check that $d$ satisfies the condition from Definition~\ref{def:coarselyproper}, for any $0<\delta<\Delta$ with $k=\lfloor \frac{\Delta}{\delta}\rfloor+1$. Since $G$ has no proper open subgroups and $d$ is compatible, we get that $d$ is coarsely proper. It is also easy to see $d$ satisfies Definition~\ref{def:largescalegeodesic} with $K$ arbitrarily close to $1$. It follows from Proposition~\ref{prop:maximalmetric} that $d$ is maximal. 
\end{proof}
In the sequel, when working with topological groups and their subsets, we shall simply say that the groups (and their subsets) are \emph{bounded}, resp. \emph{unbounded}, instead of the more precise coarsely bounded, resp. coarsely unbounded. We will also use in the sequel the basic facts (see \cite{Ro18}) that bounded sets are closed under products and continuous homomorphic images.\medskip

Before proceeding further, let us note that whereas maximality of a metric means maximality on the large scale, Rosendal in \cite{Ro18-2} also defines local minimality of metrics. The precise definition is as follows.
\begin{definition}{\cite[Definition 1]{Ro18-2}}\label{def:minimalmetric}
Let $G$ be a topological group. A left-invariant metric $d$ on $G$ is called \emph{minimal} if it is compatible and for any compatible left-invariant metric $\partial$ on $G$ the identity map $(G,\partial)\rightarrow (G,d)$ is Lipschitz on an neighborhood of the identity.
\end{definition}
The interesting point for us is that if $d$ is a left-invariant metric on $G$ that is both maximal and minimal, then it defines a \emph{global Lipschitz geometric structure} on $G$. We refer to \cite[Section 3]{Ro18-2}.\medskip

Finally we note that (continuous) left-invariant metrics, resp. pseudometrics on groups are in one-to-one correspondence with (continuous) \emph{length functions}, resp. \emph{pseudo-length functions} on groups, sometimes also called norms, resp. pseudonorms. Since the latter are sometimes more convenient to define and work with, we shall work with these two notions interchangeably. Here we recall that for a (topological) group $G$, a length function is a (continuous) function $\ell:G\rightarrow [0,\infty)$ satisfying for all $g,h\in G$, $\ell(g)=0$ if and only if $g=1_G$, $\ell(g)=\ell(g^{-1})$, and $\ell(gh)\leq \ell(g)+\ell(h)$. If the first condition is relaxed to $\ell(1_G)=0$, i.e. we allow $\ell$ to vanish on non-trivial elements of $G$, then $\ell$ is a pseudo-length function.
\subsection{Infinite-dimensional Lie groups}
Let $M$ be a topological space. We say that $M$ is a \emph{Banach manifold} if there exists a (real or complex) Banach space $Z$ such that 
\begin{itemize}
    \item $M$ is locally homeomorphic to $Z$.
    \item If $U,V\subseteq M$ are two open sets, with non-empty intersection, homeomorphic to open sets in $Z$ via the maps $p_U:U\rightarrow Z$, resp. $p_V:V\rightarrow Z$, then the map $p_u\circ p_V^{-1}: p_V[U\cap V]\rightarrow p_U[U\cap V]$ is smooth.
\end{itemize}

\begin{definition}
A Hausdorff topological group $G$ is a \emph{Banach-Lie group} if $G$ is topologically a Banach manifold and the group operations are smooth.
\end{definition}

We briefly review some basic properties of Banach-Lie groups that will be useful for us. We refer to \cite[Section 6]{Up85} and \cite[Section IV]{Neeb04} for details. To each Banach-Lie group $G$, a \emph{Banach-Lie algebra} $\LA$ can be assigned, as in the finite-dimensional case, as the tangent space at $1_G$. Since it can be canonically identified with the Banach space on which $G$ is modelled, $\LA$ is a Banach space.

There exists, as in the finite-dimensional case, the \emph{exponential function} $\exp: \LA\rightarrow G$, whose most important property for us is that it is a local diffeomorphism - in fact, being a local homeomorphism is sufficient for our purposes here. This has the following important consequences:
\begin{itemize}
    \item The connected component of the identity $G_0$ in $G$ is path-connected and open (therefore clopen).
    \item For each $X\in\LA$, $\exp(X)\in G_0$, and for each $g\in G_0$ there are $X_1,\ldots,X_n\in\LA$ such that $g=\prod_{i=1}^n \exp(X_i)$.
    \item $G/G_0$ is discrete.
\end{itemize}
The exponential $\exp(a)$ will be occasionally denoted just by $e^a$, when working with linear groups, whenever convenient.

We present few examples below. All of them, and many others, will appear later in the text, so we postpone the reference or verification that these groups are Banach-Lie to the corresponding sections.\\ \\
{\bf Examples.}
\begin{itemize}
    \item Every Banach space is an abelian Banach-Lie group.
    \item Let $A$ be a unital real or complex Banach algebra and let $\GL(A)$ denote the group of all invertible elements of $A$ equipped with the norm topology. Then $\GL(A)$ is a Banach-Lie group.
    \item Let $A$ be a unital $C^*$-algebra and let $\U(A)$ denote its unitary group with the norm topology. Then $\U(A)$ is a Banach-Lie group. Moreover, if $I$ is a closed two-sided ideal in $A$, then $\U_I(A):=\{1+a\in \U(A)\mid a\in I\}$ is a Banach-Lie group.
    \item Let $1\leq p\leq\infty$, let $\Hil$ be a Hilbert space and let $\|\cdot\|_p$ be the $p$-Schatten norm on $\mathbb{B}(\Hil)$. Then $\U_p(\Hil)=\{1+a\mid \|a\|_p<\infty\}$, equipped with the topology coming from the $p$-norm $\|\cdot\|_p$, is a Banach-Lie group.
    \item Let $H$ be a matrix finite-dimensional Lie group and let $X$ be a compact Hausdorff space. Then the group $C(X,H)$ of all continuous maps with point-wise group operations is a Banach-Lie group.
\end{itemize}

We remark that not all the finite-dimensional Lie theory generalizes - even to Banach-Lie groups. One particular instance, which will be of concern for us, is that not every closed subgroup of a Banach-Lie group is a Banach-Lie group. This will have a consequence that sometimes we will work with certain closed subgroups of Banach-Lie groups, which we do not know whether they are Banach-Lie. Some of our techniques will still apply to them. Also, we shall work with the $p$-unitary groups $\U_p(M,\tau)$, where $M$ is a semifinite von Neumann algebra. In that case, $\U_p(M,\tau)$ usually is not a Banach-Lie group, still some large scale geometric properties can be investigated similarly as for $\U_p(\Hil)$ though. Second we note that there are far more general infinite-dimensional Lie groups besides Banach-Lie groups; we refer the reader to \cite{Neeb05}. It is not clear to us though how to apply the ideas from Section~\ref{section:explength} to these more general classes of Lie groups.

\subsection{Properties (T) and (FH)}\label{subsection:prelim-T}
Here we review the basics of the Kazhdan property (T) and Property (FH). We also offer some new observations which relate Property (T) and Property (FH) for Polish (or more general topological) groups. These will be relevant later in Section~\ref{section:En}. Our main reference is \cite{BdHV} and we refer there for any unexplained notion and for more details.

Recall that a unitary representation $\pi:G\rightarrow \Uni(\Hil)$ of a topological group $G$ \emph{almost has invariant vectors} if for every compact set $Q\subseteq G$ and $\varepsilon>0$ there exists a unit vector $\xi\in\Hil$ satisfying $\max_{g\in Q} \|g\xi-\xi\|<\varepsilon$. We say that $G$ has \emph{Property (T)} (see \cite[Definition 121.3]{BdHV} for more details) if every strongly continuous unitary representation of $G$ almost having invariant vectors has a non-zero invariant vector. We also say that $G$ has \emph{strong Property (T)} if we can replace the compact set $Q$ in the definition for a finite set.

Next, we say that a topological group $G$ has \emph{Property (FH)} if every continuous action of $G$ on any Hilbert space $\Hil$ by affine isometries has a fixed point.

Property (T) and Property (FH) are related by the Delorme-Guichardet theorem (see \cite[Theorem 2.12.4]{BdHV}): If a topological group $G$ has Property (T), then it has Property (FH). The converse is not true in general, we shall later see counterexamples in the class of Polish groups, however, if $G$ is locally compact and $\sigma$-compact (in particular, locally compact and Polish), then Property (FH) implies Property (T).\medskip

Apparently, the first example of a non-locally compact group having Property (T) was provided by Shalom in \cite{Sha99} as the group of continuous homomorphism from the circle group $\T$ into $\SL(n,\Com)$, for $n\geq 3$, which can be identified with $\SL(n,C(\T))$. To the best of our knowledge, it was therefore also the first example of an unbounded non-locally compact group having Property (FH). Since then, the research on Property (T) for non-locally compact Polish groups has undergone a significant development, see e.g. \cite{Bek03}, \cite{Tsa12}, \cite{Pes18}, \cite{Iba19}. Curiously, the same cannot be said about Property (FH). Almost all the new examples of non-locally compact Polish groups with Property (T) are bounded, therefore they have Property (FH) for trivial reasons (although it is not always easy to verify that these groups are bounded).

So as far as we know, the only `non-trivial' examples of non-locally compact and unbounded Polish groups having Property (FH) are the examples by Shalom, whose list has been later much expanded by Cornulier in \cite{Cor06}. Among trivial examples are direct products of bounded non-locally compact groups and non-compact locally compact groups with property (T).\medskip

Our first goal is to find a weakening of Property (T) that is equivalent to Property (FH) for much larger class of topological groups. The new examples of unbounded topological groups having Property (T) and Property (FH) will be provided in Section~\ref{section:En}.

We start with a new definition.
\begin{definition}
Let $G$ be a Polish group and let $\pi:G\rightarrow \Uni(\Hil)$ be a strongly continuous unitary representation. We say that $\pi$ \emph{almost has invariant vectors in the bounded sense} if for every $\varepsilon>0$ and every bounded set $Q\subseteq G$ there exists a unit vector $\xi\in\Hil$ such that $\sup_{g\in Q} \|g\xi-\xi\|<\varepsilon$.

We say that $G$ has \emph{weak Property (T)} if every strongly continuous unitary representation of $G$ almost having invariant vectors in the bounded sense has a non-zero invariant vector.
\end{definition}

\begin{remark}
Notice that if $G$ is a locally compact Polish group, then the weak Property (T) for $G$ is equivalent to the standard Property (T) - as bounded is equivalent to pre-compact in this case.
\end{remark}
The following proposition is verified exactly as in the locally compact case.
\begin{proposition}
Let $G$ be a topological group. Then $G$ has the weak Property (T) if and only if there exists a bounded Kazhdan set $Q\subseteq G$; that is, a bounded set $Q\subseteq G$ and $\varepsilon$ such that if $\pi$ is a strongly continuous unitary representation that has $(Q,\varepsilon)$-invariant vector, then it has a non-zero invariant vector.
\end{proposition}

\begin{lemma}
Let $G$ be a Polish group having weak Property (T). Then
\begin{enumerate}
    \item if $Q\subseteq G$ is a bounded Borel generating set, then it is a Kazhdan set;
    \item if $Q\subseteq G$ is a Kazhdan set with non-empty interior, then it is generating.
\end{enumerate}
\end{lemma}
\begin{proof}
First, if $Q\subseteq G$ is a Borel set generating $G$, then $\bar Q=Q\cup Q^{-1}$ is also Borel and generating, i.e. $G=\bigcup_{n\in\Nat} \bar Q^n$. Since each $\bar Q^n$ has the Baire property, there must be $n\in\Nat$ such that $\bar Q^n$ is non-meager and therefore by the Pettis' theorem, $B=\bar Q^{2n}$ contains an open neighborhood $U$ of the identity which we suppose to be symmetric. Since $G$ has the weak Property (T), let $K$ be a bounded Kazhdan set. Now $\bar QU$ is an open symmetric generating set. Let $d'$ be an arbitrary compatible left-invariant metric on $G$. Since $\bar QU$ is bounded (as $Q$ is bounded), it has a finite diameter in $d'$ and therefore, by \cite[Lemma 2.51]{Ros17}, there exists a compatible left-invariant metric $d$ on $G$ that is quasi-isometric to the word metric $\rho$ with $\bar QU$ as a generating set. Since $K$ is bounded and so $d$ is bounded on $K$, and $d$ is quasi-isometric to $\rho$, there exists $k\in\Nat$ such that  $K\subseteq (\bar QU)^k$, so $K\subseteq \bar Q^{2nk+k}$. By a standard computation using triangle inequalities, one gets that every $\varepsilon/(2nk+k)$-invariant vector for $\bar Q$ is $\varepsilon$-invariant for $K$, so every $\varepsilon/2(2nk+k)$-invariant vector for $Q$ is $\varepsilon$-invariant for $K$, showing that $Q$ is Kazhdan.\medskip

If $Q$ is Kazhdan with non-empty interior, then we can argue as in the locally compact case: if $Q$ is not generating, then the subgroup $H$ generated by $Q$ is a proper open subgroup, therefore clopen, so $G/H$ is discrete. The quasi-regular representation of $G$ on $\ell^2(G/H)$ has $H$-invariant vector, in particular $Q$-invariant vector, however no invariant vector, contradicting that $Q$ is Kazhdan.
\end{proof}

We shall call a topological group \emph{$\sigma$-bounded} if it can be covered by countably many bounded sets with the Baire property. Note that a locally compact group is $\sigma$-bounded if and only if it is $\sigma$-compact. We note that even not all Polish groups are $\sigma$-bounded. Indeed, a simple diagonalization argument shows that the Polish group $\prod_{n\in\Nat} \Int$ is not $\sigma$-bounded. If it were covered by countably many bounded sets $(B_n)_{n=1}^{\infty}$, then using the basic observation that the projection on each coordinate of each $B_n$ is bounded, we could pick elements $(z_n)_{n=1}^{\infty}$ such that $z_n\notin P_n(B_n)$, where $P_n$ is the projectioon on the $n$-th coordinate, and so the element $(z_n)_{n=1}^{\infty}\in \prod_{i\in\Nat} \Int\setminus \bigcup_{n\in\Nat} B_n$. Moreover, a topological group is \emph{locally bounded} (see \cite[Definition 2.22]{Ro18}) if it has a bounded neighborhood of the identity.
\begin{lemma}\label{lem:sigma-bounded-loc-bounded}
A Polish group is $\sigma$-bounded if and only if it is locally bounded. In general, a $\sigma$-bounded Baire group is locally bounded.
\end{lemma}
\begin{proof}
Suppose that a Baire group $G$ is covered by the sequence of bounded sets $(B_n)_{n=1}^{\infty}$ which we may suppose are increasing in inclusion. Therefore there is $n$ so that $B_n$ is non-meager. By the Pettis' theorem $B_nB_n^{-1}$ contains a neighborhhod of the identity $U$ which is therefore also bounded.

The converse for Polish groups is clear since countably many translates of the bounded neighborhood of the identity cover the group.
\end{proof}

\begin{proposition}\label{prop:equiv-FH-wT}
If a Polish group has the weak Property (T), then it has Property (FH). Conversely, if a Polish group is $\sigma$-bounded and has Property (FH), then it has weak Property (T).
\end{proposition}
\begin{proof}
Our proof mimics and slightly modifies the standard proofs of the aforementioned equivalence for locally compact Polish groups (or more generally, $\sigma$-compact locally compact groups).

First we show that weak Property (T) implies Property (FH). Suppose by contradiction that $\alpha: G\curvearrowright \Hil$ is a continuous action of $G$ on a Hilbert space $\Hil$ by affine isometries without a fixed point. Let $b:G\rightarrow \Hil$ be the associated continuous cocycle, which is therefore unbounded on $G$. For every $n\in\Nat$ we define a normalised continuous positive definite function $\phi_n: G\rightarrow \Com$ by setting \[\phi_n(g):=\exp\big(-\tfrac{1}{n}\|b(g)\|^2\big).\] 
That $\phi_n$ is indeed positive definite follows from the Schoenberg's theorem (\cite{Sch38}). Let $\pi_n$ be the unitary representation of $G$ on $\Hil_n$ produced by the GNS-construction from $\phi_n$, with the cyclic unit vector $\xi_n$, and let $\pi_\infty:=\bigoplus_{n\in\Nat} \pi_n:G\rightarrow \Uni(\Hil_\infty)$, where $\Hil_\infty=\bigoplus_{n\in\Nat} \Hil_n$. We claim that the representation $\pi_\infty$ almost has invariant vectors in the bounded sense. Pick bounded $Q\subseteq G$ and $\varepsilon>0$. Since $\delta(g,h):=\|b(g)-b(h)\|$ is a continuous left-invariant pseudometric on $G$, and thus is bounded on $Q$, for any sufficiently large $n\in\Nat$ we have \[\sup_{g\in Q} \|\pi(g)\xi_n-\xi_n\|^2=2-2\re\Big(\exp\big(-\tfrac{1}{n}\|b(g)\|^2\big)\Big)<\varepsilon^2,\] so $\xi_n$ is the desired almost invariant vector. Finally, we claim that $\pi$ has no invariant vector. Assume otherwise that $\chi\in\Hil_\infty$ is an invariant vector.

There must exist $n\in\Nat$ such that the orthogonal projection of $\chi$ onto $\Hil_n$ is a non-zero invariant vector $\chi_n\in\Hil_n$ for $\pi_n$. Since the cocycle $b$ is unbounded, there exists a sequence $(g_i)_i\subseteq G$ with $\|b(g_i)\|\to\infty$, so $(g_i\xi_n)_i$ converges weakly to $0$ in $\Hil_n$. Indeed, since $\Span\{g\xi_n\mid g\in G\}$ is dense this follows from $\Sl g_i\xi_n,g\xi_n\Sr=\Sl g^{-1}g_i\xi_n,\xi_n\Sr=\exp(-\tfrac{1}{n}\|b(g^{-1}g_i\|^2)\to 0$, for any $g\in G$, as $\|b(g^{-1}g_i)\|\to\infty$.

Then we have \[\Sl \xi_n,\chi_n\Sr=\lim_{i\to\infty} \Sl \xi_n, g_i^{-1}\chi_n\Sr=\lim_{i\to\infty}\Sl g_i\xi_n,\chi_n\Sr=0,\]
from which we deduce that for any $g\in G$ \[\Sl g\xi_n,\chi_n\Sr=\Sl \xi_n,g^{-1}\chi_n\Sr=\Sl \xi_n,\chi_n\Sr=0.\] Since $\Span\{g\xi_n\mid g\in G\}$ is dense in $\Hil_n$, we conclude that $\chi_n=0$, a contradiction.
\medskip

Conversely, suppose that $G$ has Property (FH) and it is $\sigma$-bounded, therefore locally bounded. So it is covered by the sequence of bounded open sets $(B_n)_{n=1}^{\infty}$ which we may suppose are increasing in inclusion. We proceed exactly as in \cite[Proposition 2.4.5]{BdHV}, just replacing the compact sets covering the $\sigma$-compact locally compact group by bounded open subsets $(B_n)_{n=1}^{\infty}$ covering $G$. Suppose that there exists $\pi:G\rightarrow\Uni(\Hil)$, a strongly continuous unitary representation almost having invariant vectors in the bounded sense, without invariant vectors. It follows that for every $n\in\Nat$ we can find $(B_n,1/2^n)$-invariant unit vector $\xi_n\in\Hil$. We define a cocycle $b:G\rightarrow \bigoplus_{n\in\Nat} \Hil$ by setting \[b(g):=\bigoplus_{n\in\Nat} n(\pi(g)\xi_n-\xi_n).\]
The same argument as in \cite[Proposition 2.4.5]{BdHV} shows that $b$ is a continuous unbounded cocycle for the unitary representation $\bigoplus_{n\in\Nat} \pi$, which implies that the corresponding affine isometric action on $\bigoplus_{n\in\Nat} \Hil$ does not have a fixed point. This finishes the proof.
\end{proof}

\begin{proposition}
Let $G$ be a $\sigma$-bounded topological group, in particular a locally bounded Polish group, with the weak Property (T). Then $G$ is generated by a bounded neighborhood of the identity and if $G$ is moreover Baire (in particular Polish), it admits a maximal compatible left-invariant metric.
\end{proposition}
\begin{proof}
Let $G$ be a $\sigma$-bounded topological group with the weak Property (T), or equivalently with Property (FH) by Proposition~\ref{prop:equiv-FH-wT}. Let $\mathcal{B}$ be the set of all bounded subsets of $G$ with non-empty interior. For each $B\in\mathcal{B}$, let $H_B$ be the subgroup generated by $B$. Note that it is open, so the quotient $G/H_B$ is discrete and we have available the quasi-regular unitary representation $\pi_B:G\rightarrow \ell_2(G/H_B)$. $\pi_B$ has $(B,\varepsilon)$-invariant vectors for every $\varepsilon>0$ - the element $\delta_{H_B}\in\ell_2(G/H_B)$. On the other hand, unless $H_B=G$, $\pi_B$ has no invariant vectors. It follows that the unitary representation $\pi:=\bigoplus_{B\in\mathcal{B}} \pi_B$ has almost invariant vectors in the bounded sense and, assuming that $G$ is not generated by a bounded neighborhood of the identity, $\pi$ has no invariant vectors since no $\pi_B$ does. This is a contradiction.

It follows that $G$ is generated by a bounded neighborhood of the identity, and if it is Baire it admits a maximal compatible left-invariant metric by \cite[Theorem 2.53]{Ro18}.
\end{proof}

\begin{proposition}
Let $G$ be a Polish group and let $H$ be a closed cobounded subgroup. If $H$ has Property (FH), then $G$ does as well.
\end{proposition}
\begin{proof}
Let $\alpha: G\curvearrowright \Hil$ be a continuous action of $G$ on a Hilbert space by affine isometries. By assumption, $\alpha\upharpoonright H$ has a fixed point $\xi\in\Hil$. We claim that the orbit $G.\xi$ is bounded, which implies that $G$ has a fixed point. Indeed, set $\delta(g,h):=\|\alpha(g)\xi-\alpha(h)\xi\|$. It is a continuous left-invariant pseudometric on $G$. If $G.\xi$ were unbounded, there would exist a sequence $(g_i)_i\subseteq G$ with $\delta(g_i,e)\to\infty$. Since $H$ is cobounded in $G$, there exists a sequence $(h_i)_i\subseteq H$ such that $\delta(g_i,h_i)\leq K$, for some uniform constant $K$. However, $\delta(g_i,h_i)=\delta(g_i,e)\to\infty$, which is a contradiction.
\end{proof}
We were informed by C. Rosendal that the converse does not hold. One can take e.g. $S_\infty$ or $\U(\ell^2)$ with pointwise convergence topologies. These groups both contain $\Int$ as a discrete subgroup and both are bounded, so they have (FH), and $\Int$ is therefore their closed cobounded subgroup.

\subsection{The Haagerup property}
The Haagerup property is in a sense a strong converse to Property (T) and Property (FH). There are several equivalent (under some conditions) definitions and we refer to the monograph \cite{CCJJV} for a detailed treatment. The definition we choose to work with here is commonly called \emph{a-T-menability} and was actually introduced by Gromov (\cite{Gro93}). We say that a topological group $G$ has the \emph{Haagerup property (or is a-T-menable)} if it has a continuous metrically proper action by affine isometries on a Hilbert space. We recall that by a metrically proper action $G\curvearrowright \Hil$ by isometries it is usually meant an action such that for some (or equivalently for any) $\xi\in\Hil$, $\lim_{g\to\infty} \|g\xi\|=\infty$, where $g\to\infty$ is interpreted as leaving every compact subset of $G$. This definition is suited for locally compact groups and indeed it is easy to check that only locally compact groups may be a-T-menable with this definition of $g\to\infty$. A straightforward modification asking $g\to\infty$ to mean that $g$ leaves every bounded set, or equivalently $d(g,e)\to\infty$ for every coarsely proper continuous pseudometric, has been considered by Rosendal in \cite[Definition 5]{Ros17}. Improving on the existing results for locally compact groups, Rosendal provided several equivalent conditions for an amenable Polish group to have the Haagerup property. In particular, all Banach spaces that coarsely embeds into a Hilbert space have the Haagerup property as abelian groups. To the best of our knowledge, these Banach spaces are the only known non-trivial examples of non-locally compact groups having the Haagerup property. Trivial examples are again the bounded groups and one can also trivially obtain unbounded and non-locally compact group with the Haagerup property by taking a direct product of a locally compact group with the Haagerup property with some bounded non-locally compact group.

We will present the first non-trivial examples, some of them even not amenable, in Section~\ref{section:Up}. On the other hand, we shall show that some natural candidates among Banach-Lie groups for the Haagerup property actually fail it, in Section~\ref{section:En}.

The usual characterization of the Haagerup property for locally compact groups carry over to the general setting without essential change in the statement and its proofs. 

\begin{proposition}\label{prop: coarse Haagerup}
Let $G$ be a locally bounded Polish group. Then the following conditions are equivalent.
\begin{list}{}{}
\item[{\rm{(i)}}] $G$ has the Haagerup property. 
\item[{\rm{(ii)}}] $G$ admits a sequence $(\varphi_n)_{n=1}^{\infty}$ of normalized ($\varphi_n(1)=1$) continuous positive definite functions such that $\lim_{n\to \infty}\varphi_n=1$ uniformly on bounded subsets and that for each $n\in \mathbb{N}$, $\varphi_n$ vanishes at infinity\footnote{for every $\varepsilon>0$, there exists a bounded set $B\subset G$ such that $|\varphi|<\varepsilon$ on $G\setminus B$.}. 
\end{list}
\end{proposition}

\begin{lemma}\label{lem: sigmabounded open} Let $G$ be a $\sigma$-bounded Polish group. Then there exists an increasing sequence $B_1\subset B_2\subset \dots \subset G$ of bounded open neighborhoods of the identity in $G$ such that $G=\bigcup_{n=1}^{\infty}B_n$ and that for every bounded subset $B\subset G$, there exists an $n\in \mathbb{N}$ such that $B\subset B_n$ holds. 
\end{lemma}
\begin{proof}
Let $B_1'\subset B_2'\subset \cdots $ be a sequence of bounded subsets of $G$. We may assume $1\in B_1'$. 
Because the $\sigma$-boundedness is equivalent to the local boundedness, there exists an open neighborhood $U$ of 1 in $G$ which is bounded. Let $B_n=UB_1'UB_2'\cdots UB_{2^n}'$. 
Then $B_1\subset B_2\subset $ is an increasing sequence of open bounded neighborhoods of 1 in $G$ with $G=\bigcup_{n=1}^{\infty}B_n$ and $B_n^2\subset B_{n+1}\,(n\in \mathbb{N})$. Then the last claim follows from \cite[Proposition 2.7 (4)]{Ro18}. 
\end{proof}
\begin{proof}[Proof of Proposition \ref{prop: coarse Haagerup}]
(i)$\Rightarrow$(ii) Let $\pi\mid G\to \mathcal{U}(\Hil)$ be a strongly continuous unitary representation of $G$ on a Hilbert space $\Hil$ and $b\colon G\to \Hil$ be an associated proper 1-cocycle. Then for each $n\in \mathbb{N}$, the positive definite functions $\varphi_n(g)=\exp (-\tfrac{1}{n}\|b(g)\|^2)\,(g\in G)$ do the job.\\
Indeed, it is clear that $\varphi_n(1)=1$. Because $b$ is assumed to be a coarse embedding, it is expansive. So whenever a sequence $(g_n)_{n=1}^{\infty}$ in $G$ tends to $\infty$ as $n\to \infty$, $\|b(g_n)\|\stackrel{n\to \infty}{\to}\infty$ holds. This shows that $\varphi_n$ vanishes at infinity for each $n$. Moreover, if $B$ is a bounded subset of $G$, then because $b$ is bornologous, $\sup_{g\in B}\|b(g)\|<\infty$ holds. 
Therefore $\disp \lim_{n\to \infty}\varphi_n=1$ uniformly on $B$.\\
(ii)$\Rightarrow$(i) Let $(\varphi_n)_{n=1}^{\infty}$ be a sequence of continuous positive definite functions on $G$ witnessing the Haagerup property of $G$. By Lemma \ref{lem: sigmabounded open}, there exists an increasing sequence $B_1\subset B_2\subset \cdots$ of open bounded neighborhoods of 1 in $G$ such that any bounded subset of $G$ is contained in one of the $B_n$'s. Because $\disp \lim_{n\to \infty}\varphi_n=1$ uniformly on bounded subsets of $G$, by passing to a subsequence we may assume that for every $k\in \mathbb{N}$ the condition $|1-\varphi_n|<2^{-n}$ is satisfied on $B_k$ for every $n\ge k$. Let $(\pi_n,\Hil_n,\xi_n)$ be the cyclic unitary representation associated with $\varphi_n$. Let $\Hil=\bigoplus_{n=1}^{\infty} \Hil_n,\,\pi=\bigoplus_{n=1}^{\infty}\pi_n$. 
Let $g\in G$. Then $g\in B_k$ for some $k\in \mathbb{N}$. 
Then 
\eqa{
\sum_{n=k}^{\infty}\|\xi_n-\pi_n(g)\xi_n\|^2&=2{\rm{Re}}\,\sum_{n=k}^{\infty}(1-\varphi_n(g))\\
&\le 2\sum_{n=k}^{\infty}2^{-n}<\infty.
}
This implies that $b(g):=(\xi_n-\pi_n(g)\xi_n)_{n=1}^{\infty}$ defines an element in $\Hil$. Moreover, $b\colon G\to \Hil$ is a 1-cocycle of $\pi$. We show that $b$ is continuous. Let $(g_k)_{k=1}^{\infty}$ be a sequence in $G$ converging to $g\in G$. Let $\varepsilon>0$. Take $N\in \mathbb{N}$ so that $\sum_{n\ge N+1}2^{1-n}<\varepsilon$ holds. 
By assumption, there exists $k_0\in \mathbb{N}$ such that $g_k^{-1}g\in B_1\,(k\ge k_0)$ holds. 
It then follows that 
\eqa{
\|b(g_k)-b(g)\|^2&=2{\rm{Re}}\,\sum_{n=1}^{\infty}(1-\varphi_n(g_k^{-1}g))\\
&\le 2\sum_{n=1}^N(1-\varphi_n(g_k^{-1}g))+2\sum_{n=N+1}^{\infty}(1-\varphi_n(g_k^{-1}g))\\
&\le 2\sum_{n=1}^N(1-\varphi_n(g_k^{-1}g))+\varepsilon.
}
Since $\disp \lim_{k\to \infty}g_k^{-1}g=1$, it follows that $\limsup_{k\to \infty}\|b(g_k)-b(g)\|^2\le \varepsilon$. Since $\varepsilon$ is arbitrary, we have $\disp \lim_{k\to \infty}b(g_k)=b(g)$ and the continuity of $b$ follows. Next, we show that $b$ is a coarse embedding. Let $(g_k)_{k=1}^{\infty}$ be a sequence in $G$. Assume that $\disp \lim_{k\to \infty}g_k=\infty$, meaning that for every bounded subset $B\subset G$, there exists $k_0\in \mathbb{N}$ such that $g_k\in G\setminus B\,(k\ge k_0)$ holds.  Then for each $n$, $\varphi_n(g_k)\stackrel{k\to \infty}{\to}0$ holds. Let $N\in \mathbb{N}$. There exists $k_0\in \mathbb{N}$ such that ${\rm{Re}}\,(1-\varphi_n(g_k))\ge \tfrac{1}{2}\,(1\le n\le N)$ holds for every $k\ge k_0$. This implies that $\|b(g_k)\|^2\ge N\,(k\ge k_0)$. Since $N$ is arbitrary, it follows that $\|b(g_k)\|\stackrel{k\to \infty}{\to}\infty$ holds.\\
Conversely, if $g_k\to \infty$ does not hold, then there exists a subsequence $g_{k_1},g_{k_2},\dots$ and a bounded subset $B\subset G$ such that $g_{k_i}\in B\,(k\in \mathbb{N})$ holds. Take $m\in \mathbb{N}$ for which $B\subset B_m$ holds. 
Then for each $i\in \mathbb{N}$, we have 
\begin{equation}
\|b(g_{k_i})\|^2\le 4m+\sum_{n=m+1}2^{1-n}.\label{eq: bki}
\end{equation}
The right hand side of (\ref{eq: bki}) is finite and is independent of $i$. Thus, $\|b(g_{k_i})\|\to \infty$ does not hold. This shows that $b$ is a coarse embedding. 
\end{proof}

\subsection{von Neumann algebras}
In this subsection we recall backgrounds on von Neumann algebras which are necessary for our study. For more on basics of von Neumann algebra theory, we refer the reader to Takesaki's books \cite{Takesakibook1,Takesakibook2}.

First, we recall the non-commutative $L^p$-spaces.  
Let $M$ be a von Neumann algebra on a Hilbert space $\Hil$ and $\tau$ be a normal faithful semifinite trace on $M$. A closed and densely defined operator $A$ on $\Hil$ is said to be {\it affiliated with $M$} if $x'A\subset Ax'$ for every $x'\in M'$. This is equivalent to the condition that $u\in M$ and $1_B(|A|)\in M$ for every Borel set $B\subset \R$, where $A=u|A|$ is the polar decomposition of $A$. In this case, $A$ is moreover called {\it $\tau$-measurable} if $\lim_{r\to \infty}\tau(1_{[r,\infty)}(|A|))=0$ holds. 
\begin{definition}Let $1\le p<\infty$. 
The space $L^p(M,\tau)$ of all closed, densely defined and $\tau$-measurable operators on $\Hil$ which have finite $p$-norms is called the {\it non-commutative $L^p$-space} associated with $(M,\tau)$.  
Here, the $p$-norm of $A$, denoted $\|A\|_p$ is given by 
\[\|A\|_p=\left (\int_0^{\infty}\lambda^p{\rm d}\tau(e(\lambda))\right )^{\frac{1}{p}},\]
where $|A|=\int_0^{\infty}\lambda\,{\rm d}e(\lambda)$ is the spectral resolution of $|A|$. 
\end{definition}
The space $L^2(M,\tau)$ is often identified with the Hilbert space on which the semicyclic (GNS) representation of $\tau$ is defined. We also follow this convention. 
For the backdground on non-commutative integration theory, we refer the reader to Nelson's paper \cite{Nelson74},  \cite{PiXu03} or \cite[Chapter IX.2]{Takesakibook2}.

Next, we recall the Connes' characterization of hyperfiniteness. 
\begin{definition}Assume that $M$ is a factor. $M$ is called {\it hyperfinite} if there exists an increasing sequence $M_1\subset M_2\subset \cdots$ of finite-dimensional von Neuman subalgebras of $M$ whose union is dense in $M$ in the strong operator topology. 
$M$ is said to have {\it Schwartz' property P} if for every $x\in \mathbb{B}(\Hil)$, the following condition holds: 
\[\overline{\rm co}\{uxu^*\mid u\in \mathcal{U}(M)\}\cap M'\neq \emptyset.\]
Here, $\mathcal{U}(M)$ denotes the unitary group of $M$ and $\overline{\rm co}$ in the above expression refers to the closed convex hull with respect to the weak operator topology.
\end{definition}
We will use the following equivalence due to Connes \cite[Theorem 6]{Co76}. 

\begin{theorem}[Connes]\label{thm: Connes} Let $M$ be a factor acting on a separable Hilbert space. Then $M$ is hyperfinite if and only if $M$ has Schwartz' property {\rm P}.  
\end{theorem}
He actually proved that these two properties are equivalent to many other important properties, such as injectivity or semidiscreteness (see \cite{Co76}). 

\section{The exponential length}\label{section:explength}
The goal of this section is to show that on every connected Banach-Lie group there is an explicitly defined compatible length function that is maximal, i.e. it defines a quasi-isometry type of the Banach-Lie group. In fact, it also minimal, so it defines the global Lipschitz geometric structure (\cite{Ro18-2}) on any connected Banach-Lie group. We compare this length function with other length functions and left-invariant metrics traditionally considered on Lie groups, and present some of its basic properties. We also comment on more general Banach-Lie groups that are not necessarily connected.

Then we specialize to unitary groups of $C^*$-algebras, where the idea of this length function originates, and which are a large source of interesting examples. Finally, we introduce some methods for showing that the length function, and thus the group itself, is unbounded.

\subsection{Exponential length and its properties}
\begin{definition}
Let $G$ be a connected Banach-Lie group with a Lie algebra $\mathfrak{g}$. For each $g\in G$, define 
$$\bel{G}(g):=\inf \left \{\sum_{i=1}^n\|X_i\|\,\middle|\, g=\exp (X_1)\cdots \exp (X_n),\,X_i\in \mathfrak{g}\,(1\le i\le n)\right \}.$$
\end{definition}
\begin{proposition}\label{prop:bel}
Let $G$ and $\LA$ be as above.
Then 
\begin{list}{}{}
\item[{\rm{(i)}}] $\bel{G}$ is a continuous pseudo-length function on $G$.
\item[{\rm{(ii)}}] $(G,d_G)$ is a length pseudometric space, where $d_G$ is the induced left-invariant pseudometric given by 
\[d_G(g,h):=\bel{G}(g^{-1}h),\,\,g,h\in G.\]
\item[{\rm{(iii)}}]  $d_G$ is a maximal compatible metric on $G$.
\end{list}
\end{proposition}
\begin{proof}
(i) It is immediate from the definition that for every $g,h\in G$ we have $\bel{G}(g)=\bel{G}(g^{-1})$ and $\bel{G}(gh)\leq \bel{G}(g)+\bel{G}(h)$. So $\bel{G}$ is a pseudo-length function. We check that $\bel{G}$, and thus also $d_G$, is continuous. By the properties of the length function, it suffices to verify that for every sequence $(g_n)_{n=1}^{\infty}\subseteq G$, we have that $g_n\to 1$ implies $\bel{G}(g_n)\to 0$. Pick such a sequence $(g_n)_{n=1}^{\infty}\subseteq G$. There exist an open neighborhood $U$ of $0$ in $\LA$ and an open neighborhood $V$ of $1$ in $G$ such that $\exp:\LA\rightarrow G$ induces a homeomorphism between $U$ and $V$. Without loss of generality, $(g_n)_{n=1}^{\infty}\subseteq V$, therefore there is a sequence $(X_n)_{n=1}^{\infty}\subseteq U$ such that $X_n\to 0$, thus $\|X_n\|\to 0$, and $g_n=\exp(X_n)$, for each $n$. Since by definition of $\bel{G}$, for each $n$, $\bel{G}(g_n)\leq \|X_n\|$, we are done.\medskip

(ii) By the left-invariance, it suffices to show that for every $g\in G$ and every $\varepsilon>0$ there exists a path $\gamma$, connecting $1$ and $g$, of length less than $\bel{G}(g)+\varepsilon$. Find $X_1,\ldots,X_n\in\LA$ such that $g=\exp(X_1)\cdot\ldots\cdot\exp(X_n)$ and $\sum_{i=1}^n \|X_i\|<\bel{G}(g)+\varepsilon$. Set $X_0=0$. Since it is clear that for each $1\leq i\leq n$, the length of the path $\gamma_i:[0,1]\rightarrow G$, defined by $t\to \exp(X_{i-1})\exp(tX_i)$, is bounded by $\|X_i\|$. Concatenating the paths $(\gamma_i)_i$ gives a path of length less than $\bel{G}(g)+\varepsilon$ that connects $1$ and $g$.\medskip

(iii) We shall need the result of Corollary~\ref{cor:belcompatible} proved below that $d_G$ is a compatible metric. Then it follows from (ii) and Lemma~\ref{lem:maximallengthmetric} that $d_G$ is a maximal compatible metric.
\end{proof}

Next we prove that $\bel{G}$ is also minimal (recall Definition~\ref{def:minimalmetric}). We start with a lemma.
\begin{lemma}\label{lem:localityLA}
Let $G$ be a connected Banach-Lie group with a Banach-Lie algebra $\LA$. Let $V$ be a bounded open neighborhood of $0\in\LA$ such that such that $\exp: V\rightarrow G$ induces a diffeomorphism with its image. Then there exists $K\geq 1$ such that for every $X\in U$ \[\bel{G}(\exp(X))\leq \|X\|\leq K\cdot \bel{G}(\exp(X)).\]
\end{lemma}
\begin{proof}
Since for any $X\in\LA$ we have by definition $\bel{G}(\exp(X))\leq \|X\|$, we only care about the other inequality. Let $V$ be as in the statement. Suppose that no such $K\geq 1$, as wanted, exists. Then for every $n\in\Nat$ there is $X_n\in V$ such that $\|X_n\|> n\cdot\bel{G}(\exp(X_n))$. Without loss of generality, we may assume that $\inf_{n\in\Nat} \|X_n\|>0$. Indeed, if for some $X\in V$, $n,m\in\Nat$ we have $\|X\|> n\cdot\bel{G}(\exp(X))$ and $mX\in V$, then also $\|mX\|>n\cdot\bel{G}(\exp(mX))$. This holds since \[\|mX\|>mn\cdot\bel{G}(\exp(X))\geq n\cdot\bel{G}(\exp(X)^m)=n\cdot\bel{G}(\exp(mX)).\] However, then we have $\lim_{n\to\infty} \bel{G}(\exp(X_n))=0$ and so since $\bel{G}$ is compatible, $\exp(X_n)\to 1_G$, thus $X_n\to 0$, a contradiction.
\end{proof}
\begin{proposition}
Let $G$ be a connected Banach-Lie group. Then $\bel{G}$ induces a minimal metric.
\end{proposition}
\begin{proof}
We shall use the characterization of minimality from \cite[Theorem 3 (3)]{Ro18-2}. We need to find an open neighborhood $U$ of $1_G\in G$ and $K\geq 1$ such that for every $g\in G$ and $n\in\Nat$, if $g,g^2,\ldots,g^n\in U$ then $n\cdot \bel{G}(g)\leq K\cdot \bel{G}(g^n)$. Let $V$ and $K\geq 1$ be given by Lemma~\ref{lem:localityLA} and set $U:=\exp[V]\subseteq G$. We claim that $U$ and $K\geq 1$ are as desired. Let $g\in G$ and $n\in\Nat$ be such that $g,\ldots,g^n\in U$. Set $X:=\log(g)$. Then for all $1\leq i\leq n$ we also have $\log(g^i)=iX$. By Lemma~\ref{lem:localityLA} we get \[n\cdot \bel{G}(g)\leq n\|X\|=\|nX\|\leq K\cdot\bel{G}(g^n),\] and we are done.
\end{proof}
Since we have already proved that $\bel{G}$ is maximal, the following is an immediate corollary (see \cite[Section 3]{Ro18-2}).
\begin{theorem}\label{thm:elismaxmin}
Let $G$ be a connected Banach-Lie group. Then $\bel{G}$ is both maximal and minimal and thus defines the global Lipschitz geometric structure of $G$.
\end{theorem}

It is now in order to compare $\bel{G}$ and $d_G$ with other distances considered on Lie and Banach-Lie groups.

We start with the finite-dimensional case. Let $G$ be a connected real (finite-dimensional) Lie group and $\LA$ its Lie algebra. Recall (e.g. from \cite[Section 5.6.4]{DruKap}) that $G$ is equipped with a canonical metric, with which it is geometrically investigated: the left-invariant Riemannian distance. We briefly recall its construction. Choose a Hilbert space (i.e. defined using an inner product) norm $\norm$ on $\LA$. Let $\gamma: I\rightarrow G$ be a smooth (or just piecewise smooth) curve. Define the length $L(\gamma)$ of $\gamma$ as follows: $$L(\gamma):=\int_I \|DL_{\gamma(t)^{-1}}(\dot{\gamma(t))}\| dt.$$ We now define the left-invariant Riemannian distance $d(g,h)$, between any $g,h\in G$, by $$d(g,h):=\inf\{L(\gamma)\mid \gamma: [a,b]\rightarrow G\text{ is piecewise smooth}, \gamma(a)=g,\gamma(b)=h\}.$$

\begin{fact}\label{fact:belandRiemann}
Let $G$ be a connected real Lie group, let $d$ be a left-invariant Riemannian distance on $G$ and let $d_G$ be the distance induced by $\bel{G}$ using the same Hilbert norm on $\LA$. Then there exists $K>0$ such that $1/K d_G-K\leq d\leq d_G$; in particular, $d$ and $d_G$ are quasi-isometric.
\end{fact}
\begin{proof}
Since $d$ is left-invariant, it suffices to show that for some $K>0$, for any $g\in G$ we have $1/K \bel{G}(g)-K\leq d(g,e)\leq \bel{G}(g)$. We fix some $g\in G$.

First we show that $d(g,e)\leq \bel{G}(g)$. Fix some $\varepsilon>0$ and let $X_1,\ldots,X_n\in\LA$ be such that $g=\exp(X_1)\cdot\ldots\cdot\exp(X_n)$ and $\sum_{i=1}^n \|X_i\|<\bel{G}(g)+\varepsilon$. For each $i\leq n$, let $\gamma'_i:[0,1]\rightarrow G$ be the smooth curve defined by $\gamma'_i(t):=\exp(tX_i)$, and let $\gamma_1:=\gamma'_1$ and $\gamma_i:=\big(\prod_{j=1}^{i-1} \exp(X_j)\big)\cdot \gamma'_i$, for $1<i$. Let $\gamma:[0,n]\rightarrow G$ be their natural concatenation which is a piecewise smooth curve satisfying $\gamma(0)=e$ and $\gamma(n)=g$. For each $i<n$ and $i<t<i+1$ we have $DL_{\gamma(t)^{-1}}(\dot \gamma(t))=X_i$, so it follows that $L(\gamma)=\sum_{i=1}^n \|X_i\|$, and therefore $d(g,e)<\bel{G}(g)+\varepsilon$. Since $\varepsilon>0$ was arbitrary, we have proved the first inequality.

For the converse, first we find the desired $K>0$. Since $\exp$ induces a local diffeomorphism between $\LA$ and $G$, let $V$ be an open ball around $0$ in $\LA$ such that $V$ and $U:=\exp[V]\subseteq G$ are diffeomorphic (via $\exp$). Without loss of generality, in order to simplify the notation, we assume that $V$ is the unit ball. Set $$K':=\max \left \{\frac{\|X\|}{d(\exp(X),e)}\,\middle|\, 1/2\leq \|X\|\leq 1\right \},$$ and then set $K:=\max\{K',1\}$. We again fix some $\varepsilon>0$ and find a smooth curve $\gamma:[0,1]\rightarrow G$ such that $\gamma(0)=e$, $\gamma(1)=g$, and $L(\gamma)<d(g,e)+\varepsilon$. We find a partition $a_0=0<a_1<\ldots<a_n=1$ such that 
\begin{itemize}
    \item for each $0\leq i<n$, $\gamma[a_i,a_{i+1}]\subseteq \gamma(a_i)U$;
    \item for the unique sequence $(X_i)_{i=0}^{n-1}\subseteq \LA$ such that for all $i<n$, $\gamma(a_{i+1})=\exp(X_0)\cdot\ldots\cdot\exp(X_i)$, we have $\|X_i\|\leq 1/2$ if and only if $i=n-1$.
\end{itemize}
It also follows that for every $i<n$, $\|X_i\|<1$ since $X_i\in V$ as $\exp(X_i)\in U$. Now for every $i<n-1$ we have $$\bel{G}(\exp(X_i))\leq \|X_i\|\leq K'd(\gamma(a_i),\gamma(a_{i+1}))\leq K'\int_{[a_i,a_{i+1}]} \|DL_{\gamma(t)^{-1}}(\dot \gamma(t))\|dt.$$ It follows that $$\bel{G}(g)\leq \sum_{i=0}^{n-1} \|X_i\|\leq \Big(K'\int_{[0,a_{n-1}]} \|DL_{\gamma(t)^{-1}}(\dot \gamma(t))\|dt\Big)+1/2\leq K\cdot L(\gamma)+K\leq K(d(g,e)+1+\varepsilon).$$ Since $\varepsilon$ was arbitrary, we are done.
\end{proof}

Next we consider distances on general Banach-Lie groups. Let $G$ be a connected Banach-Lie group with a Banach-Lie algebra $\LA$ equipped with a compatible norm $\norm$. The distance on $G$ may be defined as in the finite-dimensional Riemannian case using the norm $\norm$.
\begin{definition}
Let $G$ be a connected Banach-Lie group and let $\LA$ be its Banach-Lie algebra with a compatible norm $\norm$. A \emph{rectifiable path} is a function $\gamma:I\rightarrow G$, where $I$ is some interval, such that the function $t\to \|DL_{\gamma(t)^{-1}}(\dot{\gamma(t))}\|$ is integrable. The integral $\int_I \|DL_{\gamma(t)^{-1}}(\dot{\gamma(t))}\|dt$ is then the length of $\gamma$ denoted by $L(\gamma)$.

The \emph{Finsler distance} $d$ on $G$ is then defined, for $g,h\in G$, by \[d(g,h):=\inf\{L(\gamma)\mid \gamma:I\rightarrow G\text{ is rectifiable and }\gamma(0)=g,\gamma(1)=h\}.\]
\end{definition}
Clearly, the Finsler distance is a continuous left-invariant pseudometric on $G$. We now compare $\bel{G}$ with the Finsler distance on $G$. In the special case of unitary groups of $C^*$-algebras, this has been already done by Ringrose in \cite[Proposition 2.9]{Ri92} who proved that these two distances coincide (notice that Ringrose does not use the term `Finsler distance'). Applying the results of Larotonda from \cite{Lar19} we generalize Ringrose's result for SIN Banach-Lie groups, or equivalently Banach-Lie groups whose Banach-Lie algebra admits an $\mathrm{Ad}$-invariant compatible norm, and thus a bi-invariant Finsler distance - which is the case of unitary groups. In general, we show that these two distances are bi-Lipschitz and that $d_G$, the distance induced from $\bel{G}$, majorizes the Finsler distance.\medskip

First, we need the following important fact.
\begin{proposition}\label{prop:Finslercompatible}
Let $G$ be a connected Banach-Lie group and let $d$ be the Finsler distance induced from some compatible norm $\norm$ on the Banach-Lie algebra $\LA$. Then $d$ is a compatible metric.
\end{proposition}
\begin{proof}
By \cite[Example 12.32]{Up85}, $\norm$ induces a compatible tangent norm on $G$ (see \cite[Definition 12.19]{Up85}) and therefore by \cite[Proposition 12.22]{Up85}, the Finsler metric on $G$ induced from $\norm$ is compatible.
\end{proof}

\begin{proposition}\label{prop:Finslercomparedtobel}
Let $G$ be a connected Banach-Lie group with a Banach-Lie algebra $\LA$ equipped with a compatible norm $\norm$. Let $d_G$ be the metric induced by $\bel{G}$ and let $d$ be the Finsler distance on $G$.
\begin{enumerate}
    \item If $G$ is a SIN group and so without loss of generality, $d$ is bi-invariant, then $d_G=d$.
    \item In general, $d\leq d_G$ and the distances $d$ and $d_G$ are bi-Lipschitz equivalent.
\end{enumerate}
 
\end{proposition}
\begin{proof}
The proof that $d\leq d_G$ is completely analogous to the corresponding proof in Fact~\ref{fact:belandRiemann} or to the corresponding proof of the Ringrose' result in \cite[Proposition 2.9]{Ri92}.

Assume first that $d$ is bi-invariant. It suffices to check that for any $g\in G$, $\bel{G}(g)\leq d(g,e)$. Fix some $\varepsilon>0$ and let $\gamma:[0,1]\rightarrow G$ be a rectifiable curve satisfying $\gamma(0)=e$, $\gamma(1)=g$, and $L(\gamma)\leq d(g,e)+\varepsilon$. Let $R>0$ be such that $\exp: U\rightarrow V$ is a local diffeomorphism, where $U$ is the open ball of radius $R$ around $0$ in $\LA$. We can find a partition $a_0=0<a_1<\ldots<a_n=1$ such that for each $0\leq i<n$, $g_i:=\gamma(a_i)^{-1}\gamma(a_{i+1})\in V$. For each $0\leq i<n$, let $X_i\in\LA$ be such that $g_i=\exp(X_i)$, and let $\gamma_i:[0,1]\rightarrow G$ be the path $\gamma_i(t):=g_{i-1}\cdot \exp(tX_i)$, where we declare $g_{-1}=e$. By \cite[Theorem 13(2)]{Lar19}, each $\gamma_i$ is the shortest path between $g_i$ and $g_{i+1}$. Moreover, it is clear that $L(\gamma_i)=\|X_i\|$, for each $0\leq i<n$. It follows that $$\bel{G}(g)\leq \sum_{i=0}^{n-1} \|X_i\|=\sum_{i=0}^{n-1} L(\gamma_i)\leq L(\gamma)<d(g,e)+\varepsilon,$$ which finishes the proof that $d=d_G$ if $d_G$ is bi-invariant, since $\varepsilon$ was arbitrary.\medskip

Now we treat the general case. Since by Theorem~\ref{thm:elismaxmin}, $\bel{G}$ is both maximal and minimal, and we have $d\leq d_G$, it follows from the definition of minimality that $d$ is also minimal. Therefore both $d$ and $d_G$ are simultaneously maximal and minimal and so by \cite[Lemma 13]{Ro18-2} they are bi-Lipschitz equivalent.
\end{proof}

\begin{corollary}\label{cor:belcompatible}
For any connected Banach-Lie group, the pseudometric $d_G$ induced from $\bel{G}$ is a compatible metric.
\end{corollary}
\begin{proof}
By Proposition~\ref{prop:Finslercompatible}, $d$ is a compatible metric, where $d$ is the Finsler metric. Since by Proposition~\ref{prop:Finslercomparedtobel}, we have $d\leq d_G$, we immediately get that $d_G$ is a metric as well. The same inequality $d\leq d_G$ also gives that $d_G$ induces a topology at least as fine as the topology induced by $d$ which is the original topology of $G$. Since $d_G$ is by Proposition~\ref{prop:bel} continuous, we get that it must be compatible.
\end{proof}
\begin{remark}
It follows from Proposition~\ref{prop:Finslercomparedtobel}, or more directly from Lemma~\ref{lem:maximallengthmetric}, Proposition~\ref{prop:Finslercompatible}, and the realization that the Finsler distance is a length metric, that the Finsler distance is also a maximal distance.

As also showed in Proposition~\ref{prop:Finslercomparedtobel}, the Finsler distance and the metric $d_G$ induced from the exponential length in some cases coincide. We are not aware of any example where they differ. In any case, it seems that in practice it is more convenient to work with $\bel{G}$ as the computations involving this distance are often easier, and in some particular cases, $\bel{G}$ can be computed explicitly. This in particular holds true for unitary groups of $C^*$-algebras. See Section~\ref{section:abelian} in this article and references from Subsection~\ref{subsection:C-star-el-examples} on the $C^*$-exponential length.
\end{remark}

We present some basic observations concerning the exponential length.

Let $\norm$ and $\norm'$ be two equivalent norms on the Lie algebra $\LA$, and let $\bel{G}$, resp. $\bel{G}'$ be the corresponding exponential lengths. Since the equivalent norms are automatically bi-Lipschitz equivalent, the following lemma is easily verified.
\begin{lemma}
Let $\norm$ and $\norm'$ be two equivalent norms on the Lie algebra $\LA$, and let $\bel{G}$, resp. $\bel{G}'$ be the corresponding exponential lengths. Then $\bel{G}$ and $\bel{G}'$, and the distances they define, are bi-Lipschitz equivalent.
\end{lemma}

\begin{proposition}
Let $G$ and $H$ be connected Banach-Lie group and let $\Phi:G\rightarrow H$ be a topological isomorphism. Then it is bi-Lipschitz with respect to exponential lengths. In particular, it is a quasi-isometric equivalence.
\end{proposition}
\begin{proof}
Let $\mathfrak{g}$ and $\mathfrak{h}$ be the Banach-Lie algebras of $G$ and $H$ respectively. As in the finite-dimensional case, see \cite[Theorem IV.2 (b)]{Neeb04}, $\Phi$ induces an isomorphism $\phi:\mathfrak{g}\rightarrow\mathfrak{h}$ which is in particular a linear isomorphism, so bi-Lipschitz, between $\mathfrak{g}$ and $\mathfrak{h}$ as Banach spaces. Let $K\geq 1$ bound the norm of $\phi$ and $\phi^{-1}$. Then for every $g\in G$ and every $x_1,\ldots,x_n\in \mathfrak{g}$ such that $g=\exp(x_1)\cdot\ldots\cdot\exp(x_n)$ we have
$$\bel{H}(\Phi(g))= \bel{H}(\Phi(\exp(x_1))\cdot\ldots\cdot\Phi(\exp(x_n)))=\bel{H}(\exp(\phi(x_1))\cdot\ldots\cdot \exp(\phi(x_n)))\leq$$ $$ \sum_{i=1}^n \|\phi(x_i)\|_{\mathfrak{h}}\leq K\sum_{i=1}^n \|x_i\|_{\mathfrak{g}}.$$ 
It follows that $\bel{H}(\Phi(g))\leq K\bel{G}(g)$. A symmetric argument shows that $\bel{G}(g)\leq K\bel{H}(\Phi(g))$, which finishes the proof.
\end{proof}

Finally, we discuss the situation when a given Banach-Lie group $G$ is not necessarily connected. Let us denote in that case the connected component of the identity by $G_0$ and by $\Gamma_G$ the discrete quotient $G/G_0$. In some cases the large scale geometry of $G$ reduces to studying separately the large scale geometry of $G_0$ and the large scale geometry of $\Gamma_G$. This happens e.g. when $G$ is abelian and we comment on it in Section~\ref{section:abelian}.

In general, let $\bel{G}$ again be the exponential length on $G_0$ and let $\ell$ be an arbitrary compatible (i.e. discrete) length function on the discrete group $\Gamma_G$. Let $R'\subseteq G$ be a set of representatives for the left cosets $G/G_0$, where $G_0$ is represented by $1_G$, and set $R=R'\cup (R')^{-1}$. It follows that $R$ intersects each left coset in at most two elements. We construct a compatible length function $L$ on $G$ made from $\bel{G}$ and $\ell$ as follows. In the following, we denote by $P:G\rightarrow \Gamma_G$ the quotient map.

For any $g\in G$ we set \[L(g):=\inf\left \{\sum_{i=1}^n \Big (\ell(P(r_i))+\bel{G}(g_i)\Big )\,\middle|\, g=\prod_{i=1}^n r_ig_i,\; (r_i)_{i=1}^n\subseteq R,\; (g_i)_{i=1}^n\subseteq G_0\right \}.\]

We leave to the reader the straightforward verification that $L$ is a compatible length function.

\begin{corollary}\label{cor:G0coarseinG}
Suppose that $G/G_0$ is at most countable. Then the inclusion $G_0\subseteq G$ is a coarse embedding.

In particular, if $G$ is separable, then the inclusion $G_0\subseteq G$ is a coarse embedding.
\end{corollary}
\begin{proof}
We suppose that $G/G_0$ is infinite as otherwise the statement is trivial. Let $\ell$ be a proper length function on $\Gamma_G$ (which exists as $\Gamma_G$ is countable; see e.g. \cite[Proposition 1.2.2.]{NoYu12}) and let $L$ be constructed from $\bel{G}$ and $\ell$ as above. Pick a sequence $(g_n)_{n=1}^{\infty}\subseteq G_0$ such that $\bel{G}(g_n)\to\infty$.

We claim that $L(g_n)\to\infty$. If not, by passing to a subsequence if necessary, we may assume that there exists $M$ such that for all $n$, $L(g_n)< M$. It follows that for each $n$ there are elements $(r_i)_{i=1}^m\subseteq R$ and $(h_i)_{i=1}^m\subset G_0$ such that $g_n=\prod_{i=1}^m r_ih_i$ and $\sum_{i=1}^m \{\ell(P(r_i))+\bel{G}(h_i)\}<M$. However, since $\ell$ is proper, there are only finitely many finite sequences $\vec{r}=(r_1,\ldots,r_k)\subseteq R$ such that $\sum_{i=1}^k \ell(P(r_i))<M$. Let $\mathcal{R}$ be the finite set of such sequences. For each $\vec{r}=(r_1,\ldots,r_k)\in\mathcal{R}$, let $B_{\vec{r}}:=\prod_{j=1}^k r_jB$, where $B=\{g\in G_0\mid \bel{G}(g)<M\}$. Each $B_{\vec{r}}$ can be also written as $\big(\prod_{j=1}^k r_j\big)\cdot\big(\prod_{j=1}^k B_j\big)$, where each $B_j$, $j\leq k$, is a conjugate of $B$. Since $G_0$ is normal, for every $j\leq k$, $B_j\subseteq G_0$, and moreover since it is an image of a bounded set in $G_0$ via a topological automorphism of $G_0$ - the conjugation by an element of $G$, it is also bounded. Finally, notice that for any $\vec{r}\in\mathcal{R}$, $B_{\vec{r}}\cap G_0\neq\emptyset$ if and only if $B_{\vec{r}}\subseteq G_0$ if and only if $\prod_{j=1}^k r_k\in G_0$. It follows that for every $\vec{r}\in\mathcal{R}$, if $B_{\vec{r}}\cap G_0\neq \emptyset$, then $B_{\vec{r}}$ is a bounded subset of $G_0$. Consequently, $(g_n)_{n=1}^\infty$ is covered by finitely many bounded subsets of $G_0$, which is a contradiction.
\end{proof}

\begin{corollary}\label{cor:coarselyproperLFonG}
$G$ admits a coarsely proper compatible length function if and only if $G/G_0$ is at most countable.
\end{corollary}
\begin{proof}
Suppose that $\Gamma_G=G/G_0$ is at most countable. We suppose that it is infinite, otherwise the statement is again trivial. $\Gamma_G$ then admits a proper length function $\ell$, as mentioned above. We show that the length function $L$, defined above, constructed from $\bel{G}$ and $\ell$ is coarsely proper. Let $(g_n)_{n=1}^\infty\subseteq G$ be a given sequence. We claim that $(g_n)_{n=1}^{\infty}$ is unbounded if and only if either the sequence $(P(g_n))_{n=1}^{\infty}$ is unbounded in $\Gamma_G$ or this sequence is bounded, thus finite, and without loss of generality we then have $(P(g_n))_{n=1}^{\infty}$ is constant, and $\bel{G}(g_1^{-1}g_n)\to\infty$. Indeed, if $(P(g_n))_{n=1}^{\infty}$ is unbounded in $\Gamma_G$, then $(g_n)_{n=1}^\infty$ cannot be bounded since otherwise $(P(g_n))_{n=1}^{\infty}$ is an image of a bounded set by a quotient map, thus bounded. So assume that $(P(g_n))_{n=1}^{\infty}$ is bounded, thus finite. Without loss of generality, we assume that $(P(g_n))_{n=1}^{\infty}$ is constant. Then if $\bel{G}(g_1^{-1}g_n)$ is bounded, then $(g_1^{-1}g_n)_{n=1}^\infty$ is bounded in $G_0$, thus also in $G$, and consequently also $(g_n)_{n=1}^\infty$ is bounded in $G$. If $\bel{G}(g_1^{-1}g_n)\to\infty$, then by Corollary~\ref{cor:G0coarseinG}, $L(g_1^{-1}g_n)\to\infty$, so also $L(g_n)\to\infty$, and $(g_n)_{n=1}^\infty$ is unbounded.

So suppose now that $(g_n)_{n=1}^\infty$ is unbounded. Either $(P(g_n))_{n=1}^{\infty}$ is infinite, or without loss of generality it is constant. In the latter case, by Corollary~\ref{cor:G0coarseinG}, $L(g_1^{-1}g_n)\to \infty$, thus also $L(g_n)\to\infty$. In the former case, if $(L(g_n))_{n=1}^\infty$ is bounded, then as in the proof of Corollary~\ref{cor:G0coarseinG} we get that $(g_n)_{n=1}^\infty$ is covered by $\bigcup_{i=1}^k g_i B_i$, where $g_i\in G$ and $B_i\subseteq G_0$ is bounded. However, $P(\bigcup_{i=1}^k g_i B_i)$ is bounded in $\Gamma_G$, a contradiction.

For the other direction, we must show that if $\Gamma_G$ is uncountable, then $G$ does not admit a coarsely proper compatible length function. It is straightforward to check that for a sequence $(g_n)_{n=1}^{\infty}\subseteq G$, if the set $\{P(g_n)\mid n\in\Nat\}$ is infinite, then $(g_n)_{n=1}^{\infty}$ is unbounded in $G$. It follows that if $L$ were a coarsely proper length function on $G$, then the quotient length function $\ell$ defined on the quotient would be a proper length function on $\Gamma_G$, which is impossible since $\Gamma_G$ is uncountable.
\end{proof}

\begin{corollary}
$G$ admits a maximal compatible length function, i.e. $G$ has a well-defined quasi-isometry type, if and only if $G/G_0$ is finitely generated.
\end{corollary}
\begin{proof}
Suppose first that $G/G_0=\Gamma_G$ is finitely generated. Let $S\subseteq G$ be a finite symmetric set containing $1_G$ such that $P[S]\subseteq \Gamma_G$ generates $\Gamma_G$. We slightly modify the definition of the length function $L$ as follows. For any $g\in G$ we set \[L(g):=\inf\left \{\sum_{i=1}^n\Big ( \ell(P(s_i))+\bel{G}(g_i)\Big )\,\middle|\, g=\prod_{i=1}^n s_ig_i,\; (s_i)_{i\leq n}\subseteq S,\; (g_i)_{i\leq n}\subseteq G_0\right \},\] where $\ell:S\rightarrow\{0,1\}$ is defined so that $\ell(s)=0$ if and only if $s=1_G$.

The proof that $L$ is a coarsely proper compatible length function is as in the proof of Corollary~\ref{cor:coarselyproperLFonG}. So it is enough to just prove that $L$ is large scale geodesic (recall Proposition~\ref{prop:maximalmetric}). However this easily follows (with constant $K=2$) from the definition of $L$ and the fact that $\bel{G}$ is large scale geodesic with a constant arbitrarily close to $1$.\medskip

Conversely, assume that $G$ admits a maximal compatible length function $L$. Let $\ell$ be the quotient length function on $\Gamma_G$, which is compatible, thus discrete. We claim that $(\Gamma_G,\ell)$ is \emph{coarsely connected}, i.e. there exists $C>0$ such that for every $\gamma\in\Gamma_G$ there are $\gamma_1,\ldots,\gamma_n\in\Gamma_G$ such that $\gamma=\prod_{i=1}^n \gamma_i$ and $\ell(\gamma_i)<C$, for every $i\leq n$. We prove that. Pick $\gamma\in \Gamma_G$ and let $g\in G$ be such that $p(g)=\gamma$. Since $L$ is large scale geodesic with some constant $K>0$, there exist $g_1,\ldots,g_n\in G$ such that $g=\prod_{i=1}^n g_i$ and $L(g_i)<K$, for all $i\leq n$. It suffices to set $\gamma_i=P(g_i)$, for $i\leq n$, and we have obtained a sequence witnessing that $\ell$ is coarsely connected with constant $C=K$, since $\ell(\gamma_i)\leq L(g_i)<K$, for $i\leq n$. Now $\Gamma_G$ must be finitely generated by \cite[Proposition 1.A.1 (fg)]{CdH16}.
\end{proof}

\subsection{Examples: The $C^*$-exponential length and rank}\label{subsection:C-star-el-examples}
Here we briefly review our main example of an exponential length function on a Banach-Lie group that has inspired the general definition - the $C^*$-exponential length on connected components of unitary groups of $C^*$-algebras.

Let $A$ be a unital $C^*$-algebra, denote by $\U(A)$ its unitary group equipped with the norm topology, and by $\U_0(A)$ the connected component of the identity of $\U(A)$. $\U(A)$ is an algebraic subgroup of $\GL(A)$, therefore a Banach-Lie group (see \cite[Theorem 1]{HK77} and \cite[Proposition IV.9]{Neeb04}). A standard argument shows that the Banach-Lie algebra of $\U(A)$ is the set of skew-adjoint elements of $A$. Here we shall however follow the literature on $C^*$-algebras, that is also compatible with the practice in mathematical physics, and we identify the Lie algebra of $\U(A)$ with the set $A_{\rm{sa}}$ of self-adjoint elements of $A$; the exponential map has therefore the form $a\to \exp(ia)$.\medskip

The exponential length for the identity components of unitary groups of $C^*$-algebras has been introduced in \cite{Ri92} (for non-unital algebras $A$, $\U(A)$ may be defined as $\{1+a\in \U(\tilde{A})\mid a\in A\}$, where $\tilde{A}$ is the minimal unitization of $A$). Since then, a substantial research in the theory of $C^*$-algebras has been done on computing the $C^*$-exponential length, which in our terms can be described as determining for which $C^*$-algebras $A$, the unitary group $\U_0(A)$ is bounded, or unbounded \cite{AM20}. We refer the reader to \cite{Ri92}, \cite{Ph93}, \cite{Zhang93}, \cite{Lin14}, \cite{PW14}, and references therein.\medskip

For any $C^*$-algebra $A$, since $\U(A)$ is a Banach-Lie group, $\U_0(A)=\{\prod_{j=1}^n \exp(\ri a_j)\mid a_1,\ldots,a_j\in A_{\rm{sa}}\}$. Following \cite{PhRi91}, define $\cer(A)$, the $C^*$-exponential rank of $A$, to be the minimal $n$ (if it exists, otherwise set $\cer(A)=\infty$) such that for every $u\in \U_0(A)$ there are $a_1,\ldots,a_n\in A_{\rm{sa}}$ such that $u=\prod_{j=1}^n \exp(\ri a_j)$. Note that originally $\cer$ has been defined to attain values $1,1+\varepsilon,2,2+\varepsilon,\ldots$, however this more precise definition is not relevant for our purposes. The main original reason for defining the $C^*$-exponential length in \cite{Ri92} was to find in some cases bounds on the $C^*$-exponential rank. The following fact, proved more quantitatively in \cite[Corollary 2.7]{Ri92}, is immediate from the general Lie theory.
\begin{fact}
Let $A$ be a $C^*$-algebra. If $\U_0(A)$ is bounded, i.e. $\mathrm{diam}_{\mathrm{cel}}(\U_0(A))<\infty$, then $\cer(A)$ is finite.
\end{fact}
Since from our point of view, which is the point of view of large scale geometry, unbounded groups are more interesting, the contraposition of the fact is more important. If $\cer(A)=\infty$, then the group $\U_0(A)$ is unbounded.\medskip

Based on the discussion above, we present some examples.
\begin{enumerate}
    \item If $A$ is a unital $C^*$-algebra of real rank zero, then $\U_0(A)$ is bounded; more precisely, $\mathrm{diam}_{\mathrm{cel}}(\U_0(A))=\pi$ (see \cite[Theorem 3.5]{Ph95}).
    \item If $X$ is a compact Hausdorff space and $A=C(X)\otimes \mathbb{B}(\Hil)$, then $\U_0(A)$ is bounded (see \cite{Ri92} and also \cite{Zhang93} for far reaching generalizations).
    \item If $X$ is a compact Hausdorff space which is not totally disconnected, $n\geq 2$, and $B$ is a UHF $C^*$-algebra, then the groups $\U_0(C(X))$, $\U_0(C(X)\otimes M_n)$, and $\U_0(C(X)\otimes B)$ are unbounded (see \cite[Theorem 6.7]{Ph93}).
    \item More generally, if $A$ is a $C^*$-algebra and there are two different traces on $A$ that induce the same homomorphisms from $K_0(A)$, then $\U_0(A)$ is unbounded (see \cite[Corollary 3.2]{Ph95}).
    \item Let $F_\infty$ be the free group on countably many generators and let $A=C^*(F_\infty)$, the full group $C^*$-algebra of $F_\infty$. Then $\cer(A)=\infty$ and so $\U_0(A)$ is unbounded (see \cite[Corollary 2.7]{Ph94}).
    
\end{enumerate}

Let us note that it often happens that for a $C^*$-algebra $A$, $\cer(A)<\infty$, however $\U_0(A)$ is unbounded. Consider e.g. abelian unital $C^*$-algebras as a prime example when this occurs.

\subsection{The reduced exponential length}
In \cite{Ph95}, Phillips introduced the \emph{reduced $C^*$-exponential length}. Besides computing the exponential rank, it turns out this is one of the main tools how to show that a particular unitary group (more precisely, its connected component of the identity) of a $C^*$-algebra has infinite $\mathrm{cel}$, and thus it is unbounded. Here we generalize the reduced exponential length to general Banach-Lie groups and show some of its applications.
\begin{remark}
For many $C^*$-algebras $A$, this method of Phillips from \cite{Ph95} is essentially the only way how to show that $\U_0(A)$ is unbounded. Indeed, as proved in \cite{Ph95} and also more generally in Proposition~\ref{prop:relbasicfacts}, the reduced exponential length is the distance of an element to the commutator subgroup. If $A$ is a simple $\mathcal{Z}$-stable $C^*$-algebra of rational tracial rank at most one, then by \cite{Lin14}, the diameter of the commutator subgroup of $\U_0(A)$ is bounded by $2\pi$ in the exponential length.
\end{remark}
\begin{definition}
Let $G$ be a connected Banach-Lie group with a Banach-Lie algebra $\LA$. Let $g\in G$ and let us define the \emph{reduced exponential length} $\rbel{G}(g)$ as $$\inf \left \{\|\sum_{i=1}^n X_i\|\,\middle|\, g=\exp (X_1)\cdots \exp (X_n),\,X_i\in \mathfrak{g}\,(1\le i\le n)\right \}.$$
\end{definition}
\begin{proposition}\label{prop:relbasicfacts}
Let $G$ be a connected Banach-Lie group with a Banach-Lie algebra $\LA$.
\begin{list}{}{}
    \item[{\rm{(i)}}] The function $\mathrm{rel}:G\rightarrow \left[0,\infty\right)$ is a continuous pseudo-length function on $G$ and we have $\mathrm{rel}\leq \mathrm{el}$.
    \item[{\rm{(ii)}}] If $G'$ is the closure of the derived subgroup of $G$, then $\mathrm{rel}$ is the quotient length function, of $\mathrm{el}$, on $G/G'$. That is, for every $g\in G$ $$\rbel{G}(g)=\inf\left\{\bel{G}(gh)\,\middle|\, h\in G'\right\}.$$ In particular, it is a maximal compatible length function on the abelianization $G/G'$.
\end{list}
\end{proposition}
\begin{proof}
It is clear that for every $g,h\in G$ we have $\rbel{G}(g)=\rbel{G}(g^{-1})$ and $\rbel{G}(gh)\leq \rbel{G}(g)+\rbel{G}(h)$. Thus it is a pseudo-length function on $G$. By definition, for any $g\in G$, $\rbel{G}(g)\leq \bel{G}(g)$. So since $\mathrm{el}$ is continuous by Proposition~\ref{prop:bel}, $\mathrm{rel}$ is continuous as well. This proves (i).\medskip

In order to prove (ii), we shall need the following lemma (cf. \cite[Lemmas 2.4 and 2.5]{Ph95}).
\begin{lemma}\label{lem:Trotterandrel}
For every $X,Y\in\LA$ and for every $g\in G'$ we have
\begin{enumerate}
    \item\label{rel:it1} $\exp(X)\exp(Y)\exp(X+Y)^{-1},\exp([X,Y])\in G'$;
    \item\label{rel:it2} $\rbel{G}(\exp(X)\exp(Y))=\rbel{G}(\exp(X+Y))$, $\rbel{G}(g)=0$.
\end{enumerate}
\end{lemma}
\begin{proof}[Proof of Lemma~\ref{lem:Trotterandrel}]
\eqref{rel:it2} follows from \eqref{rel:it1} as in \cite[Lemma 2.5]{Ph95}, while \eqref{rel:it1} follows exactly as in \cite[Lemma 2.4]{Ph95} from the Trotter product formula $$\exp(X+Y)=\lim_{n\to\infty} \Big(\exp(\tfrac{1}{n}X)\exp(\tfrac{1}{n}Y)\Big)^n$$ and the commutator formula $$\exp([X,Y])=\lim_{n\to\infty} \Big(\exp(\tfrac{1}{n}X)\exp(\tfrac{1}{n}Y)\exp(-\tfrac{1}{n}X)\exp(-\tfrac{1}{n}Y)\Big)^{n^2}$$ respectively, which hold in general Banach-Lie groups (see \cite[Theorem IV.2]{Neeb04}).
\end{proof}

Fix some $g\in G$ and let $h\in G'$ be arbitrary. We get $\rbel{G}(gh)=\rbel{G}(g)$ since $\rbel{G}(h)=0$ by Lemma~\ref{lem:Trotterandrel}. Thus $\rbel{G}(g)=\rbel{G}(gh)\leq\bel{G}(gh)$. Since $h$ was arbitrary, we get $\rbel{G}(g)\leq \inf\left\{\bel{G}(gh)\,\middle|\, h\in G'\right\}$. For the other inequality, fix some $\varepsilon>0$ and find $X_1,\ldots,X_n\in\LA$ such that $g=\exp(X_1)\cdots \exp(X_n)$ and $\|\sum_{i=1}^n X_i\|< \rbel{G}(g)+\varepsilon$. Set $g':=\exp(X_1+\ldots +X_n)$. By Lemma~\ref{lem:Trotterandrel}, $h:=g^{-1}g'\in G'$, so we have \[\bel{G}(gh)=\bel{G}(g')\leq \|\sum_{i=1}^n X_i\|<\rbel{G}(g)+\varepsilon.\] Since $\varepsilon>0$ was arbitrary, we get $\inf\left\{\bel{G}(gh)\mid h\in G'\right\}\leq \rbel{G}(g)$, and we are done.

It follows that $\mathrm{rel}$ is compatible with the quotient topology on $G/G'$. Maximality of $\mathrm{rel}$ follows by a straightforward argument from maximality of $\mathrm{el}$ - this is left to the reader since it is not needed in the sequel. This finishes the proof.
\end{proof}
In what follows, we shall consider $\mathrm{rel}$ both as a pseudo-length function defined on $G$ and a length function defined on $G/G'$.\medskip

Next we define a map analogous to the homomorphism from \cite[Theorem 2.8]{Ph95} using which we later show that $\mathrm{rel}$ is unbounded on many Banach-Lie groups, thus these groups are unbounded.

For the rest of the section a connected Banach-Lie group $G$ and its Banach-Lie algebra $\LA$ are fixed. Denote by $\LA'$ the closed real linear span of commutators in $\LA$, i.e. $$\LA':=\overline{\Rea-\Span}\left\{[X,Y]\mid X,Y\in\LA\right\}.$$ That is, $\LA'$ is the derived Banach-Lie algebra of $\LA$. In the sequel, we shall sometimes denote the Banach space $\LA/\LA'$ by $E$.

\begin{proposition}\label{prop:mapPhi}
The map $\Phi':\LA\rightarrow G/G'$ defined by $X\to P_G(\exp(X))$, where $P_G:G\rightarrow G/G'$ is the projection, induces a continuous open surjective homomorphism, which vanishes on $\LA'$, thus it factorizes through a continuous open surjective homomorphism $\Phi: E\rightarrow G/G'$. In particular, $G/G'\simeq E/\mathrm{ker}(\Phi)$.
\end{proposition}
\begin{proof}
The proof mostly follows the lines of the proof of \cite[Theorem 2.8]{Ph95}. By Lemma~\ref{lem:Trotterandrel} \eqref{rel:it1}, we get that $\Phi'$ is a homomorphism that vanishes on commutators $[X,Y]\in\LA$. Surjectivity follows since $G$ is connected, and continuity is clear. In order to show that $\Phi'$ is open, it suffices to show that for any $r>0$, $\Phi'[B_r]$ is open in $G/G'$, where $B_r:=\{X\in\LA\mid \|X\|<r\}$. Pick $r>0$. It follows from the definition of $\mathrm{rel}$ and from Proposition~\ref{prop:relbasicfacts} (ii) that $$\Phi'[B_r]=\{g\in G/G'\mid \rbel{G}(g)<r\},$$ which shows the claim.

It follows from the continuity that $\Phi'$ vanishes on the whole $\LA'$, thus we obtain a continuous onto homomorphism $\Phi: E\rightarrow G/G'$. Since $\Phi'=\Phi\circ P_\LA$, where $P_\LA: \LA\rightarrow E$ is the projection, we see that $\Phi$ is open as well. Therefore it is a quotient map and we obtain the topological group isomorphism $G/G'\simeq E/\mathrm{ker}(\Phi)$.
\end{proof}
In view of the preceding proposition, it is important to understand the kernel of $\Phi$. Here we digress from \cite{Ph95} and use more general ideas. In several cases, we will be able to obtain that $E/\mathrm{ker}(\Phi)$, and thus also $G/G'$ and $G$, are unbounded groups.\bigskip

View the Banach space $E=\LA/\LA'$ as both a Banach-Lie group and a Banach-Lie algebra with a trivial Lie bracket (and exponential being the identity). It follows that the projection $P_\LA: \LA\rightarrow E$ is a continuous Lie algebra homomorphism. As the Baker-Campbell-Hausdorff formula holds true also for Banach-Lie groups (see \cite[Section 9.2.5]{HiNe-book} for the BCH-formula for classical Lie groups and notice that it has been generalized for Banach-Lie groups already by Birkhoff in \cite{Bir38}), we can argue as in the finite-dimensional case (see e.g. \cite[Theorem 9.5.9]{HiNe-book}) that $P_\LA$ induces a local continuous homomorphism $\psi: U\rightarrow E$ defined on some neighborhood of the identity $U\subseteq G$ (for $g\in U$, we have $\psi(g)=P_\LA(\log(g))$).

Let now $f:[0,1]\rightarrow G$ be an arbitrary continuous map. We define $\Psi(f)\in E$ as follows. We choose a partition $0=t_0<t_1<\ldots<t_n=1$ so that for every $0<i\leq n$, $f(t_{i-1})^{-1}f(t_i)\in U$, and we set $$\Psi(f):=\sum_{i=1}^n \psi\big(f(t_{i-1})^{-1}f(t_i)\big).$$ Using the fact that $\psi$ is a local homomorphism on $U$, it is easy to check that $\Psi(f)$ does not depend on the partition $0=t_0<t_1<\ldots<t_n=1$. A standard and similar argument, again using that $\psi$ is a local homomorphism, in fact shows that if for some other $f':[0,1]\rightarrow G$, if $f\sim_h f'$, i.e. $f'(0)=f(0)$, $f'(1)=f(1)$, and $f$ and $f'$ are homotopy equivalent, then $\Psi(f)=\Psi(f')$. In particular, we record the following lemma as a corollary of the preceding discussion.
\begin{lemma}
The map $\Psi$ induces a homomorphism, with the same name, $\Psi:\pi_1(G)\rightarrow E$.
\end{lemma}

A classical result from Lie theory, valid also for Banach-Lie groups, says that if $G$ is simply connected, then the local homomorphism uniquely extends to a full homomorphism from $G$ to $E$. Since we do not want to assume simple-connectedness in general, we define the following homomorphism.
\begin{definition}
Let $F:=\overline{\Psi[\pi_1(G)]}$ be the closure of the image of $\pi_1(G)$ in $E$ via $\Psi$. Set $Z:=E/F$ and let $P_Z:E\rightarrow Z$ be the projection. We define a map $\Delta: G\rightarrow Z$ as follows. Since $G$ is connected, it is generated by the neighborhood $U$. Therefore for any $g\in G$, find a discrete path $g_1,\ldots,g_n\in U$ such that $g=g_1\cdots g_n$ and set $$\Delta(g):=\sum_{i=1}^n P_Z\big(\psi(g_i)\big).$$ To see that $\Delta$ is well-defined, notice that if $h_1,\ldots,h_m\in U$ is another path such that $g=h_1\cdots h_m$, then the elements $1_G, g_1, g_1g_2,\ldots,g=g_1\cdots g_n=h_1\cdots h_m,h_1\cdots h_{m-1},\ldots,h_1,1_G$ lie on some closed curve whose homotopy equivalence class represents some element $x\in\pi_1(G)$. It follows that $P_Z\big(\sum_{i=1}^n \psi(g_i)\big)=P_Z\big(\sum_{j=1}^m \psi(h_j)\big)$ showing that the definition of $\Delta(g)$ does not depend on the choice of the path.

It follows that $\Delta:G\rightarrow Z$ is a continuous homomorphism. Since its target is abelian, it factorizes through a continuous homomorphism, denoted the same, $\Delta: G/G'\rightarrow Z$.
\end{definition}
\begin{theorem}\label{thm:commutativediagram}
The following diagram commutes.
\[
  \begin{tikzcd}
    E  \arrow{d}{\Phi} \arrow{dr}{P_Z}\\
    G/G' \arrow{r}{\Delta}  & Z
  \end{tikzcd}
\]
\end{theorem}
\begin{proof}
We must show that for every $[X]\in E$ we have $P_Z([X])=\Delta\circ \Phi([X])$, where $[X]\in E=\LA/\LA'$ is an equivalence class of some $X\in\LA$. Since all the maps are homomorphisms, it suffices to check it for $[X]\in E$ such that $\exp(X)\in U$. However then we have $$\Delta\circ\Phi([X])=\Delta\Big(P_G\big(\exp(X)\big)\Big)=P_Z\Big(\psi\big(\exp(X)\big)\Big)=P_Z\Big(P_\LA\big(\log(\exp(X))\big)\Big)=P_Z([X]).$$
\end{proof}
The commutative diagram immediately provides a desired estimate on the kernel of $\Phi$.
\begin{corollary}\label{cor:kernelofPhi}
We have $\mathrm{ker}(\Phi)\subseteq F$.
\end{corollary}
The last theorem utilizes the machinery we have built to provide a criterion for a Banach-Lie group to be unbounded.
\begin{theorem}
Let $G$ be an infinite-dimensional connected Banach-Lie group. If $\pi_1(G)$ is finitely generated, in particular, when $G$ is simply connected, and if the finite rank of $\pi_1(G)$ is strictly less than the dimension of the Banach space $\LA/\LA'$ (thus in the case that $G$ is simply connected, we merely want that $\LA'\neq \LA$), then $G/G'$ and so also $G$ are unbounded.
\end{theorem}
\begin{proof}
By Proposition~\ref{prop:mapPhi}, we get that $G/G'\simeq E/\mathrm{ker}(\Phi)$, and by Corollary~\ref{cor:kernelofPhi}, we have $\mathrm{ker}(\Phi)\subseteq \overline{\Psi[\pi_1(G)]}$. Let $H$ be the finite-dimensional Banach subspace of $E$ generated by $\Psi[\pi_1(G)]$ (that is, its basis is a subset of $\Psi[S]$, where $S$ is a finite generating set of $\pi_1(G)$). We have $\overline{\Psi[\pi_1(G)]}\subseteq H$ and the dimension of $H$ is at most the finite rank of $\pi_1(G)$. Therefore the dimension of $H$ is strictly less than the dimension of $E$, we get that $E/H$ is a non-trivial Banach space, thus an unbounded Banach-Lie group. Since there is a projection $E/\mathrm{ker}(\Phi)\to E/H$ and a quotient of a bounded group is bounded, we get that $E/\mathrm{ker}(\Phi)$, and thus also $G/G'$ and in turn also $G$, are unbounded.
\end{proof}
\section{Large scale geometry of abelian unitary groups}\label{section:abelian}
In this section we illustrate the notions of the previous section on the example of abelian unitary groups on which the exponential length can be explicitly computed in an elementary way. It also demonstrates how the notions of coarse equivalence, quasi-isometry, and isomorphism differ.

Before specializing to unitary groups, we start with some observations on large scale geometry of general abelian Banach-Lie groups.
\subsection{Abelian Banach-Lie groups}
The main examples of abelian Banach-Lie groups are Banach spaces themselves, in which case the Banach-Lie algebra of a Banach space $X$ is $X$ with trivial Lie bracket and $\exp:X\rightarrow X$ being the identity map.

We describe general connected abelian Banach-Lie groups. Let $G$ be a connected abelian Banach-Lie group with Banach-Lie algebra $X$ and exponential map $\exp: X\rightarrow G$. Some parts of the following theorem are probably known even in this generality, however lacking a proper reference, we provide a complete proof.
\begin{theorem}\label{thm:generalabelianliegrps}
Let $G$ be a connected abelian Banach-Lie group whose Banach-Lie algebra is a Banach space $X$. Then there exists a discrete subgroup $\Gamma$ of $X$ and a topological group isomorphism between $G$ and $X/\Gamma$. The group $\Gamma$ is isomorphic to the fundamental group $\pi_1(G)$.  Further, let $T$ be the closure of the torsion subgroup of $G$ and let $Y=\overline{\Rea-\rm{span}[\Gamma]}\subseteq X$. Then $G/T$ is topologically isomorphic to the real Banach space $X/Y$.

Moreover, the subgroup $T$ is bounded if and only if $\Gamma$ is cobounded in $Y$, i.e. $Y/\Gamma$ is bounded. In particular, if $\Gamma$ is cobounded in $Y$, then $G$ is quasi-isometric to the Banach space $X/Y$
\end{theorem}
\begin{proof}
Note that $\exp$ is, in this case, a continuous surjective open homomorphism of abelian groups. Set $\Gamma:=\rm{Ker}(\exp)$. Since $\exp$ is a local homeomorphism, it follows that $\Gamma$ is a discrete subgroup of $X$. The isomorphism $G\simeq X/\Gamma$ then immediately follows.\medskip

We show that $\Gamma\simeq \pi_1(G)$. To see this, observe first that the Banach space $X$ is the universal cover of $G$. Indeed, clearly $X$ is simply connected and it is straightforward to check that the quotient map $p:X\rightarrow X/\Gamma\simeq G$ is a covering map. The rest is a standard fact of algebraic topology that whenever a discrete group $\Gamma$ acts freely and properly discontinuously by homeomorphism on a path connected and simply connected space $X$ (which $\Gamma$ does in this case), then $\pi_1(X/\Gamma)\simeq \Gamma$.\medskip

Next we claim that $T=\exp[Y]$. First we show that $\exp[Y]\subseteq T$. Since $T$ is a closed subgroup and $\exp$ is continuous it suffices to show that for every $\gamma\in\Gamma$ and every $\alpha\in \Rea$ we have $\exp(\alpha\gamma)\in T$. Further application of the continuity of $\exp$ reduces the general case first to the case when $\alpha=p/q$, for $p,q\in\Int$, and then, using that $\exp$ is a homomorphism, to the case when $\alpha=1/q$, for $q\in\Int$. However then $\exp(1/q\cdot \gamma)$ is an element with torsion $q$, thus it belongs to $T$.

In order to show that $T\subseteq \exp[Y]$, we show that $\exp[Y]$ is closed and that for every torsion element $g\in G$, we have $g\in\exp[Y]$. Let us do the the former. Let $(g_n)_{n=1}^{\infty}\subseteq \exp[Y]$ be a sequence converging to some $g\in G$. Since $\exp$ is surjective and a local homeomorphism, there are $x\in X$ with $\exp(x)=g$, open neighborhoods $U$, $V$ of $g$ in $G$ and $x$ in $X$ respectively, so that the restriction $\exp: V\rightarrow U$ is a homeomorphism. Without loss of generality, we may assume that $(g_n)_{n=1}^{\infty}\subseteq U$. We can then find a sequence $(x_n)_{n=1}^{\infty}\subseteq V\cap Y$ such that $\exp(x_n)=g_n$, for all $n$, and $x_n\to x$. Since $Y$ is a closed subspace, we get $x\in Y$.

Now we show the latter, e.g. that any torsion element $g\in G$ is in $\exp[Y]$. Let $n$ be the order of $g$ and let $x\in X$ be such that $\exp(x)=g$. It follows that $\exp(nx)=g^n=1$, so $nx\in\Gamma$ and therefore $x\in Y$.

We showed that $T=\exp[Y]$ and it immediately follows that $\exp$ induces an isomorphism between $G/T$ and $X/Y$.\medskip

Finally, we prove that $\Gamma$ is cobounded in $Y$ if and only if $T$ is bounded. Suppose first that $\Gamma$ is cobounded in $Y$, so there exists a bounded subset $A\subseteq Y$ such that $Y=\bigcup_{\gamma\in\Gamma} \gamma+A$. Pick $y\in Y$, $\gamma\in\Gamma$ and $a\in A$ such that $y=\gamma+a$. It follows that \[\exp(y)=\exp(\gamma)\cdot\exp(a)=\exp(a),\] thus $\exp[Y]=\exp[A]$, and since $A$ is bounded, $\exp[A]$ is bounded as well, as an image of a bounded set by a continuous homomorphism.

Conversely, assume that $\Gamma$ is not cobounded in $Y$ and let $(y_n)_{n=1}^{\infty}\subseteq Y$ be such that $\mathrm{dist}(y_n,\Gamma)\to \infty$. It follows that the projection of the sequence $(y_n)_{n=1}^{\infty}$ is unbounded in the quotient group $X/\Gamma$ and since $\exp$ induces a topological isomorpism between $X/\Gamma$ and $G$, we get $(\exp(y_n))_{n=1}^{\infty}$ is unbounded in $G$. Since for each $n$, $\exp(y_n)\in T$, we are done.

The final claim of the statement follows. If $\Gamma$ is cobounded and therefore $T$ is bounded in $G$, it is straightforward to see that the projection $G\rightarrow G/T$ is a quasi-isometry, and $G/T$ is isomorphic to the Banach space $X/Y$.
\end{proof}

In this paper, we focus mainly on connected Banach-Lie groups. Section~\ref{section:explength} however contained some observations on large scale geometry of not necessarily connected Banach-Lie groups. We complete this discussion here. It turns out that for abelian Banach-Lie groups, their large scale geometry can be independently divided into the large scale geometry of the connected groups and the discrete groups. This follows from the following basic fact.

\begin{fact}\label{fact:splitabeliangrps}
Let $G$ be an abelian Banach-Lie group and let $G_0$ be the connected component of the identity. Then $G\simeq G_0\oplus G/G_0$.
\end{fact}
\begin{proof}
By Theorem~\ref{thm:generalabelianliegrps}, $G_0$ is a quotient of a Banach space, therefore a quotient of a divisible abelian group, so $G_0$ itself is divisible. Since $G_0$ is also an open subgroup of $G$, the short exact sequence \[0\to G_0\to G\to G/G_0\to 0\] topologically and algebraically splits. This means that $G$ is topologically isomorphic to $G_0\oplus G/G_0$.
\end{proof}
Suppose that $d$ is a maximal compatible metric on $G_0$. By abusing the notation, denote still by $d$ its canonical extension to a pseudometric on $G_0\oplus G/G_0$ by defining $d((a,g),(b,h)):=d(a,b)$. Do the same for a maximal (if $G/G_0$ is finitely generated) or coarsely proper (if $G/G_0$ is countable) metric $p$ on $G/G_0$. It follows that a maximal (or coarsely proper) compatible invariant metric on $G$ can be obtained as a sum $d$ and $p$. This reduces the investigation of large scale geometry of $G$ to separately investigating the geometry of $G_0$ and separately the geometry of $G/G_0$. 
\subsection{Abelian unitary groups}
Here we start our investigation of the large scale geometry of unitary groups, with the norm topology, of unital commutative $C^*$-algebras. Let $A$ be a unital commutative $C^*$-algebra and let $X$ be a compact Hausdorff space such that $A=C(X)$. Denote the discrete abelian group $\U(A)/\U_0(A)$ by $\Gamma_A$. Then $\U(A)=C(X,\T)$ and in fact, we have the following exact sequence (see \cite[Proposition 8.50]{HoMo13})
\[0\to C(X,\Int)\to C(X,\Rea)\to C(X,\T)\to \Gamma_A\to 0.\]
Moreover, as follows from Fact~\ref{fact:splitabeliangrps}, we have $\U(A)\simeq \U_0(A)\oplus \Gamma_A$.

In the sequel, we shall focus only on the connected component $\U_0(A)$.
\medskip

First we provide an explicit formula on the exponential length $\rm{cel}$ which is useful in many computations.

Let $P:\Rea\rightarrow \Rea/\Int$ be the canonical projection. Denote by $\Omega:\U_0(A)\rightarrow C(X)/C(X,\Int)$ the canonical isomorphism between $\U_0(A)=C(X,\T)_0$ and $C(X,\Rea)/C(X,\Int)=C(X,\Rea/\Int)_0$ (induced from the isomorphism between $\T$ and $\Rea/\Int$). Moreover, denote by $\|\cdot\|_Q$ the quotient norm on $C(X,\Rea)/C(X,\Int)$. That is, for every $f\in C(X,\Rea/\Int)_0$, set \[\|f\|_Q:=\inf\{\|f'\|\mid f'\in C(X,\Rea),\; P\circ f'=f\},\]
where $\|f'\|$ is the standard supremum norm of $f'$.
\begin{proposition}\label{prop:abeliancel}
Let $X$ be a compact Hausdorff space and set $A=C(X)$. Then for every $f\in \U_0(A)$, we have \[\cel f=2\pi\|\Omega(f)\|_Q.\]
\end{proposition}
\begin{proof}
The constant $2\pi$ in the equality appears because of the two different lengths of two canonical models of the circle group: either as $\T=\{t\in\Com\mid |t|=1\}$, or $\Rea/\Int$. Denote by $G$ the group $C(X,\Rea/\Int)$ and $G_0$ its component of the identity. In this notation we have $\Omega:\U_0(A)\rightarrow G_0$. If for each $f\in G_0$ we define \[|f|:=\inf\{\sum_{k=1}^n \|x_k\|\mid f=\sum_{k=1}^n P\circ x_k,\; (x_k)_{k\leq n}\subseteq C(X,\Rea)\},\]
we clearly get that for all $a\in\U_0(A)$, $\cel a=2\pi |\Omega(a)|$. Thus it suffices to show that for every $f\in G_0$ we have $|f|=\|f\|_Q$. Fix $f\in G_0$. It is immediate from the definition that $|f|\leq \|f\|_Q$. Conversely, fix $\varepsilon>0$ and find $(x_k)_{k\leq n}\subseteq C(X,\Rea)$ such that $f=\sum_{k=1}^n P\circ x_k$ and $\sum_{k=1}^n \|x_k\|< |f|+\varepsilon$. Set $x=\sum_{k=1}^n x_k$. Since $P$ is a homomorphism from $\Rea$ to $\Rea/\Int$ we have \[f=\sum_{k=1}^n P\circ x_k=P\circ \sum_{k=1}^n x_k=P\circ x,\] and thus \[\|f\|_Q\leq \|x\|\leq \sum_{k=1}^n \|x_k\|<|f|+\varepsilon.\] Since $\varepsilon$ was arbitrary, we are done.
\end{proof}
With this proposition in hand, we can now easily recover the following result of Phillips (see \cite[Corollary 3.3]{Ph95} for one implication; for the other implication, if $X$ is a totally disconnected compact Hausdorff space, then $C(X)$ is an AF-algebra, i.e. a direct limit of finite-dimensional algebras $(A_n)_{n=1}^{\infty}$, and we have by \cite[Lemma 3.8]{Ph95} that $\cel{C(X)}=\lim_n \cel{A_n}=\pi$, since it is easy to check that $\cel{A}=\pi$ for $A$ finite-dimensional).
\begin{corollary}\label{cor:totallydisconnectedbounded}
Let $X$ be a compact Hausdorff space and let $A=C(X)$. Then $\mathrm{cel}$ is bounded on $\U_0(A)$ if and only if $X$ is totally disconnected.
\end{corollary}
As another consequence using that $\mathrm{cel}$ is a maximal compatible length function, we recover the recent result of Ando and Matsuzawa \cite[Proposition 4.6]{AM20}.
\begin{corollary}
Let $X$ be a compact Hausdorff space and let $A=C(X)$. Then $\U_0(A)$ is bounded if and only if $X$ is totally disconnected.
\end{corollary}

As being abelian and unitary, the groups $\U_0(C(X))$ seem to be candidates for groups with the Haagerup property. The following theorem disproves this hope in the non-trivial case when these groups are unbounded.
\begin{theorem}\label{thm:Haagerupunitaryabelian}
Let $X$ be a compact Hausdorff space and let $A=C(X)$. Then $\U_0(A)$, resp. $\U(A)$  are quasi-isometrically, resp. coarsely universal (for separable metric spaces) if and only if $X$ is not totally disconnected. Analogously, $\U_0(A)$ and $\U(A)$ admit coarse embedding into Hilbert space as well as the Haagerup property if and only if $X$ is totally disconnected.
\end{theorem}
\begin{proof}
If $X$ is totally disconnected, then $\U_0(A)$ is bounded, therefore it is not coarsely universal (nor quasi-isometrically universal); on the other hand, it trivially coarsely embeds into Hilbert space and has the Haagerup property. Moreover, since in this case $\U(A)=\U_0(A)$, we get the same conclusion for $\U(A)$ as well.\medskip

We now assume that $X$ is not totally disconnected. To simplify the notation, in the following we identify $\U_0(A)=C(X,\T)_0$ with $C(X,\Rea/\Int)_0$ which is canonically isometrically isomorphic with $C(X,\Rea)/C(X,\Int)$. We shall denote $C(X,\Rea/\Int)_0$ by $G_0$. Let $H$ be the closed subgroup of $G_0$ obtained by taking closure of the subgroup $\{f\in C(X,\Rea/\Int)\mid \text{the range of }f\text{ is finite}\}$. A straightforward computation shows that
\begin{gather*}
G_0/H\simeq \Big(C(X,\Rea)/C(X,\Int\Big)/\overline{\{f\in C(X,\Rea)\mid \text{the range of }f\text{ is finite}\}}\simeq\\ C(X,\Rea)/\overline{\{f\in C(X,\Rea)\mid \text{the range of }f\text{ is finite}\}}.
\end{gather*}
It follows that $G_0/H$ is isomorphic to a Banach space, since $\overline{\{f\in C(X,\Rea)\mid \text{the range of }f\text{ is finite}\}}$ is a Banach subspace of $C(X,\Rea)$. Denote this Banach space isomorphic to $G_0/H$ by $Z$.

Let $Q:G_0\rightarrow G_0/H$ be the projection. We claim that it is a quasi-isometry. Since $Q$ is onto, it suffices to show that $H$ is bounded. However, $H$ is isomorphic to the group\\ $\overline{\{f\in C(X,\Rea)\mid \text{the range of }f\text{ is finite}\}}/C(X,\Int)$. Since the group\\ $\{f\in C(X,\Rea)\mid \text{the range of }f\text{ is finite}\}/C(X,\Int)$ is clearly bounded, its closure is bounded as well.

\begin{lemma}\label{lem:cotypelemma}
$Z$ is bi-Lipschitz universal, i.e. every separable metric space embeds into $Z$ by a bi-Lipschitz embedding.
\end{lemma}
Assume for a moment that we have proved the lemma. Then since $\U_0(A)$ is isomorphic to $G_0$, which is quasi-isometric to $Z$, we get that $\U_0(A)$ is quasi-isometrically universal. Further, since $\U_0(A)$ coarsely embeds into $\U(A)$, we get that $\U(A)$ is coarsely universal. It follows that none of them can coarsely embeds into Hilbert space or have the Haagerup property. Indeed, the latter implies the former, and the former is in contradiction with coarse universality since Hilbert space is known not to be coarsely universal (see \cite{DGLY02}). It remains to prove the lemma.\medskip

\begin{proof}[Proof of Lemma~\ref{lem:cotypelemma}]
Since $X$ is not totally disconnected, it contains an infinite closed connected subset $K\subseteq X$. By the Urysohn lemma, we can find a continuous onto map $\phi':K\rightarrow [0,1]$ and by the Tietze extension theorem, we can extend it to a continuous function $\phi: X\rightarrow [0,1]$. Set $Y$ to be the Banach space $C([0,1],\Rea)/\{f\in C([0,1],\Rea)\mid f\text{ is constant}\}$.

Now it clearly suffices to prove the following two claims.\\

\noindent {\bf Claim 1.} $Y$ linearly isometrically embeds into $Z$.\\

\noindent {\bf Claim 2.} $Y$ is bi-Lipschitz universal.\\

We start with {\bf Claim 1}. First, we set $\psi_0:C([0,1],\Rea)\rightarrow C(X,\Rea)$ to be the dual map to $\phi$; that is, $\psi_0(f):=f\circ \phi$, for all $f\in C([0,1],\Rea)$. Second, we define a map $\psi:Y\rightarrow Z$ as follows. Let $[f]_Y$ be an element of $Y$ viewed as an equivalence class of some $C([0,1],\Rea)$ in the quotient $C([0,1],\Rea)/\{f\in C([0,1],\Rea)\mid f\text{ is constant}\}$. We set \[\psi([f]_Y):=[\psi_0(f)]_Z,\] where $[\psi_0(f)]_Z$ is the equivalence class of the function $\psi_0(f)=f\circ \phi: X\rightarrow \Rea$ in the quotient\\ $Z=C(X,\Rea)/\overline{\{f\in C(X,\Rea)\mid \text{the range of }f\text{ is finite}\}}$.\\ \\
\textbf{Step 1. The map $\psi$ is well defined.} To prove this, we must verify that $[\psi_0(f)]_Z=[\psi_0(f')]_Z$ for all $f,f'\in C([0,1],\Rea)$ such that $[f]_Y=[f']_Y$. However, $[f]_Y=[f']_Y$ implies that $f-f'\in \{f\in C([0,1],\Rea)\mid f\text{ is constant}\}$. Since clearly, \[\psi_0[\{f\in C([0,1],\Rea)\mid f\text{ is constant}\}]\subseteq \overline{\{f\in C(X,\Rea)\mid \text{the range of }f\text{ is finite}\}},\] we get that $\psi_0(f-f')\in \overline{\{f\in C(X,\Rea)\mid \text{the range of }f\text{ is finite}\}}$, so $[\psi_0(f)]_Z=[\psi_0(f')]_Z$, so $[\psi([f]_Y)]_Z$ is well defined.\\ \\
    \textbf{Step 2. The map $\psi$ is linear.} This is straightforward and left to the reader.\\ \\
    \textbf{Step 3. The map $\psi$ is isometric.} Pick $f\in C([0,1],\Rea)$. Since
    \begin{itemize}
        \item $\psi_0:C([0,1],\Rea)\rightarrow C(X,\Rea)$ is isometric,
        \item $\|[f]_Y\|=\inf\{\|f+c\|\mid c:[0,1]\rightarrow \Rea\text{ is constant}\}$,
        \item $\|\psi([f]_Y)\|=\inf\{\|\psi_0(f)+c\|\mid c\in \overline{\{f\in C(X,\Rea)\mid \text{the range of }f\text{ is finite}\}}\}$,
        \item $\psi_0[\{f\in C([0,1],\Rea)\mid f\text{ is constant}\}]\subseteq \overline{\{f\in C(X,\Rea)\mid \text{the range of }f\text{ is finite}\}}$,
    \end{itemize}
    we get that $\|\psi([f]_Y)\|\leq \|[f]_Y\|$.
    
    To prove the reverse inequality, set $R=\|\psi([f]_Y)\|$ and fix $\varepsilon>0$. By definition, we can find $c'\in \overline{\{f\in C(X,\Rea)\mid \text{the range of }f\text{ is finite}\}}$ such that $R\leq \|\psi_0(f)+c'\|<R+\varepsilon/2$ and then, by further approximation, we can find $c\in C(X,\Rea)$ with finite range such that $R\leq \|\psi_0(f)+c\|<R+\varepsilon$. Since $c$ has finite range, say $\{r_1,\ldots,r_n\}\subseteq \Rea$, by setting $K_i:=c^{-1}(r_i)$, for $i\leq n$, we obtain a partition $(K_i)_{i=1}^n$ of $X$ consisting of clopen sets. Since $K$ is connected, there exists $i\leq n$ such that $K\subseteq K_i$. It follows that $c\upharpoonright K$ is constant with value $r_i$. Let $d\in C([0,1],\Rea)$ be a constant function with the same value $r_i$. It follows that \[\|\psi([f]_Y)\|+\varepsilon\geq\|\psi_0(f)+c\|\geq \sup_{x\in K} |\psi_0(f)(x)+c(x)|=\sup_{x\in K} |f\circ \phi(x)+c(x)|=\sup_{y\in [0,1]} |f(y)+d(y)|\geq \|[f]_Y\|.\] Since $\varepsilon>0$ was arbitrary, we are done. This finishes the proof of the first claim.\\
    
    Now we prove {\bf Claim 2}. $Y$ is a quotient of $C([0,1],\Rea)$ by a one-dimensional subspace of constant functions. It is easy to check that it is linearly isomorphic, thus bi-Lipschitz equivalent, with the subspace of $C([0,1],\Rea)$ of functions with zero integral with respect to the Lebesgue measure: by the map $[f]\to f-\big(\int_0^1 f d\lambda\big)\cdot \mathbbm{1}$, where $\mathbbm{1}$ is the constant function with value $1$ and $[f]$ is an equivalence class of some $f\in C([0,1],\Rea)$. We therefore identify $Y$ with such a space and prove an isometric universality for this space, by abusing the notation still denoted by $Y$. Let $M$ be an arbitrary separable metric space. It is well known that $C([0,1],\Rea)$ is isometrically universal, i.e. there exists an isometric embedding $\varphi: M\rightarrow C([0,1],\Rea)$. We produce an isometric embedding $\phi:M\rightarrow Y$. For every $m\in M$ we define \[\phi(m)(x):=\begin{cases} \varphi(m)(3x) & 0\leq x\leq 1/3,\\
    -\varphi(m)(3-3x) & 2/3\leq x\leq 1,\end{cases}\] and we extend $\phi(m)$ on the interval $[1/3,2/3]$ linearly. It is now straightforward to check that for every $m,m'\in M$
    \begin{itemize}
        \item $\int_0^1 \phi(m)d\lambda=0$, thus $\phi(m)\in Y$,
        \item $\|\phi(m)-\phi(m')\|=\|\varphi(m)-\varphi(m')\|=d_M(m,m'),$
    \end{itemize}
    so $\phi$ is the desired isometric embedding. This finishes the proof of the second claim, and so also the proof of Lemma~\ref{lem:cotypelemma} and the proof of Theorem~\ref{thm:Haagerupunitaryabelian}.
\end{proof}

\end{proof}

In the special case when $X$ has just finitely many connected components, we obtain a stronger result.

First we identify the fundamental group $\pi_1(\U_0(A))$, for an abelian unital $C^*$-algebra $A$, in the following lemma which is likely well known.
\begin{lemma}\label{lem:fundamentalgrpabelianunitary}
Let $X$ be a compact Hausdorff space and $A=C(X)$. Then the fundamental group $\pi_1(\U_0(A))$ is equal to $C(X,\Int)$.

In particular, $\pi_1(\U_0(A))$ is finitely generated if and only if $X$ has only finitely many connected components, or equivalently, when the algebra of clopen subsets of $X$ is atomic. In that case, $\pi_1(\U_0(A))\simeq \Int^n$, where $n$ is the number of connected components of $X$.
\end{lemma}
\begin{proof}
By Theorem~\ref{thm:generalabelianliegrps}, $\pi_1(\U_0(A))=\rm{Ker}(\exp)$, for $\exp: C(X,\Rea)\rightarrow \U_0(A)$. This has been already identified with $C(X,\Int)$. The latter statements are then easy to verify.
\end{proof}
In the following we show a strong distinction between the equivalence relations of topological isomorphism and quasi-isometry, within the class of separable abelian unitary groups. We restrict our attention to the separable abelian unitary groups with finite rank fundamental groups. We leave the general case for some future investigation.

We start with classification of this class up to topological isomorphism.

\begin{theorem}\label{thm:classificationabelianunitarygrps}
Let $\U_0(A)$ and $\U_0(B)$ be connected components of separable infinite-dimensional abelian unitary groups with finitely generated fundamental groups. Then $\U_0(A)$ and $\U_0(B)$ are topologically isomorphic  if and only if $\pi_1(\U_0(A))$ is isomorphic to $\pi_1(\U_0(B))$, i.e. the fundamental group is a complete invariant for topological isomorphism.
\end{theorem}
\begin{proof}
Fix $\U_0(A)$ and $\U_0(B)$ satisfying the requirements of the theorem. Let $K$ and $L$ be the Gelfand spaces of $A$ and $B$ respectively (necessarily uncountable and metrizable). If $\U_0(A)$ and $\U_0(B)$ are topologically isomorphic, then clearly their fundamental groups are isomorphic.

So we prove the converse and we now assume that $\pi_1(\U_0(A))=\pi_1(\U_0(B))$. Let $n$ be the common rank of these fundamental groups. It follows that we can write $K$ as a disjoint sum of clopen connected components $K_1,\ldots,K_n$ and similarly $L$ as the disjoint sum of clopen connected components $L_1,\ldots,L_n$. Each component is either a singleton or uncountable metrizable. 
For $i\leq n$, we set $e_i:=\chi_{K_i}\in C(K,\Int)$, and $f_i:=\chi_{L_i}\in C(L,\Int)$, the characteristic functions of the corresponding sets. Clearly, $e_1,\ldots,e_n$ generate $C(K,\Int)$ and an analogous statement for $L$ holds. Since $\U_0(A)$ can be identified with $C(K,\Rea)/C(K,\Int)$ we have $\U_0(A)=C(K,\Rea)/\langle e_1,\ldots,e_n\rangle$, and analogously $\U_0(B)=C(L,\Rea)/\langle f_1,\ldots,f_n\rangle$. Since $K$ and $L$ are uncountable and metrizable, by the Miljutin theorem \cite[Theorem 4.4.8]{Albiacbook} there exists a linear isomorphism $\phi: C(K,\Rea)\rightarrow C(L,\Rea)$. If for every $i\leq n$ we have $\phi(e_i)=f_i$, then $\phi[C(K,\Int)]=C(L,\Int)$ and so $\phi$ also induces a topological isomorphism between $\U_0(A)$ and $\U_0(B)$. This is not the case in general, but we prove that we can find a linear isomorphism between $C(K,\Rea)$ and $C(L,\Rea)$ with this property.

We write $C(K,\Rea)$ as a Banach space sum $X'\oplus V'$, where $V'=\Span\{e_1,\ldots,e_n\}$ and $X'$ is the complement of $V'$ in $C(K,\Rea)$. Analogously, we write $C(L,\Rea)$ as a Banach space sum $Y'\oplus W'$. For every $i\leq n$, we have $\phi(e_i)=\big(\sum_{j=1}^n \alpha_j^i f_i\big)+y_i$, where $y_i\in Y'$, and we have $\phi^{-1}(f_i)=\big(\sum_{j=1}^n \beta_j^i e_i\big)+x_i$, where $x_i\in X'$. Set now $V:=\Span\{e_1,x_1,\ldots,e_n,x_n\}$, $W:=\Span\{f_1,y_1,\ldots,f_n,y_n\}$ and let $X$, resp. $Y$ be their complements in $C(K,\Rea)$, resp. $C(L,\Rea)$. It follows that for every $i\leq n$ we have $\phi(x_i)=f_i+\sum_{j=1}^n \beta_j^i\phi(e_i)\in W$. Analogously, for every $i\leq n$, $\phi^{-1}(y_i)\in V$. We obtain that the restriction of $\phi$ to $V$ induces an isomorphism between $V$ and $W$. Without loss of generality, we may also assume that $\phi$ induces an isomorphism between $X$ and $Y$. Indeed, let $P_X$ be the linear projection from $C(K,\Rea)$ onto $X$; analogously, we define $P_Y$. Set $\phi_0:=P_Y\circ \phi\upharpoonright X$. Then we claim that $\phi_0$ is a linear isomorphism between $X$ and $Y$.

\begin{itemize}
    \item It is injective since if form $z\in X$ we have $\phi_0(z)=0$, then by definition $\phi(z)\in W$, thus $z\in V\cap X$, as $\phi$ induces an isomorphism between $V$ and $W$, so $z=0$.
    \item It is surjective. Pick $z\in Y$. Then there is $q'\in C(K,\Rea)$ such that $\phi(q)=z$. Set $q=P_X(q')$. We have that $q'=q+s$, where $s\in V$, thus \[z=\phi_0(q')=\phi_0(q)+\phi_0(s)=\phi_0(q),\] which shows the surjectivity.
\end{itemize}
Since $\phi_0$ is also continuous, by the bounded inverse theorem it is a linear isomorphism. In the sequel, we may therefore assume that $\phi_0=\phi\upharpoonright X$.

We define a new linear operator $\psi:C(K,\Rea)\rightarrow C(L,\Rea)$ as follows:
\begin{itemize}
    \item We set $\psi\upharpoonright X=\phi\upharpoonright X$;
    \item For $i\leq n$, we set $\phi(e_i)=f_i$;
    \item For $i\leq n$, we set $\psi(x_i)=y_i$.
\end{itemize}
Therefore we have got a linear isomorphism between $C(K,\Rea)$ and $C(L,\Rea)$ satisfying $\psi[C(K,\Int)]=C(L,\Int)$, which thus induces a topological isomorphism between $\U_0(A)$ and $\U_0(B)$. This concludes the proof.

\end{proof}

We show that the classification up to topological isomorphism is in stark contrast to the classification up to quasi-isometry.
\begin{theorem}
All identity components of separable infinite-dimensional abelian unitary groups with finitely generated fundamental group are quasi-isometric. In fact, they are also quasi-isometric to the Banach space $C([0,1])$.
\end{theorem}
\begin{proof}
Let $\U_0(A)$ be a separable infinite-dimensional abelian unitary group with finitely generated fundamental group, where $A=C(K)$. By Lemma~\ref{lem:fundamentalgrpabelianunitary}, the fundamental group of $\U_0(A)$ is isomorphic to $C(K,\Int)$. Since it is finitely generated, $X:=\mathrm{Span}[C(K,\Int)]=\{f\in C(K,\Rea)\mid f\text{ has finite range}\}$, and so $C(K,\Int)$ is clearly cobounded in $X$. By Theorem~\ref{thm:generalabelianliegrps}, $\U_0(A)$ is quasi-isometric to the Banach space $C(K,\Rea)/X$. Since all hyperplanes in $C(K,\Rea)$ are isomorphic to $C(K,\Rea)$ itself (see \cite[Proposition 4.4.1]{Albiacbook}) and $X$ is finite-dimensional, we get that $C(K,\Rea)\simeq C(K,\Rea)\oplus X$, and thus $C(K,\Rea)/X\simeq C(K,\Rea)$. It follows that $\U_0(A)$ is quasi-isometric to $C(K,\Rea)$ and that is in turn linearly isomorphic, and so quasi-isometric, to $C([0,1])$ by the Miljutin theorem. This finishes the proof.
\end{proof}
\section{The $p$-unitary group $\mathcal{U}_p(M,\tau)$}\label{section:Up}
In this section we study the $p$-unitary groups (see \cite{AM12} where these groups have been introduced) associated with a semifinite von Neumann algebra. 
\begin{definition}\label{def: p-unitary group} Let $(M,\tau)$ be a semifinite von Neumann algebra and a normal faithful semifinite trace on it. Let $1\le p<\infty$. The the {\it $p$-unitary group} of $M$, denoted by $\mathcal{U}_p(M,\tau)$ is the group of all unitaries $u\in \mathcal{U}(M)$ such that $u-1\in L^p(M,\tau)$. The group is endowed with the metric topology given by the following bi-invariant metric 
\[d(u,v):=\|u-v\|_{p},\ \ u,v\in \mathcal{U}_p(M,\tau).\]
In the case $M=\mathbb{B}(\Hil)$ for a Hilbert space $\Hil$ and $\tau$ is the usual trace, the group is called the {\it Schatten $p$-unitary group}  and is denoted by $\mathcal{U}_p(\Hil)$. The group $\mathcal{U}_{\infty}(\Hil)=\{u\in \mathcal{U}(\Hil)\mid u-1\in \mathbb{K}(\Hil)\}$ is sometimes called the {\it Fredholm unitary group}. If $M$ is finite, then $\mathcal{U}_p(M,\tau)=\mathcal{U}(M)_s$ is independent of the choise of $\tau$ and $p$. 
\end{definition}

\begin{fact}\label{fact: U_pM not BanachLie}
We remark that if $M$ is a factor, then $\mathcal{U}_p(M,\tau)$ is a Banach-Lie group if and only if $M$ is of type I.
\end{fact}
\begin{proof}
If $M$ is of type I, then it is well-known (see. e.g., \cite{dlHbook}) that $\mathcal{U}_p(M,\tau)$ is a Banach-Lie group. On the other hand, suppose that $M$ is a type II factor and $\tau$ is a normal faithful semifinite trace on it. 
Then the group $\mathcal{U}_p(M,\tau)$ has small subgroups, which implies that $\mathcal{U}_p(M,\tau)$ is not a Banach-Lie group. Indeed, for the type II$_1$ case it has small subgroups by \cite[Corollary 3.12]{AM20}, so assume that $M$ is of type II$_{\infty}$. Let $e\in M$ be a nonzero finite projection. Then the group $G_e:=\{u+e^{\perp}\mid u\in \mathcal{U}(eMe)\}$ is a closed subgroup of $\mathcal{U}_p(M,\tau)$. Morevoer, the norm $\|\cdot\|_p$ induces the strong operator topology on $G_e$ (cf. \cite[Lemm 4.18]{AM20}) and $\mathcal{U}(eMe)\ni u\mapsto u+e^{\perp}\in G_e$ is a topological group isomorphism. Again by \cite[Corollary 3.12]{AM20}, $\mathcal{U}(eMe)$ has small subgroups, whence so does $\mathcal{U}_p(M,\tau)$. 
\end{proof}
\begin{remark}
Even though $\mathcal{U}_p(M,\tau)$ is not necessarily a Banach-Lie group, it is easy to check, using Lemma~\ref{lem: Lpunitary}, that for every $u\in \mathcal{U}_p(M,\tau)$ the formula \[\bel{M}(u):=\inf\left \{\sum_{i=1}^n \|a_i\|_p\,\middle|\, (a_i)_{i\leq n}\subseteq L^p(M,\tau)_{sa}\cap M, u=e^{\ri a_1}\cdots e^{\ri a_n}\right \},\] defines a compatible maximal length function as for Banach-Lie groups.
However, because it is more convenient in this case to use the $p$-norm distance, we shall use the distance $d$. Note also that although $\bel{G}$ is both maximal and minimal for a general Banach-Lie group $G$ by Theorem \ref{thm:elismaxmin}, the metric $\bel{M}$ can never be minimal because $\mathcal{U}_p(M,\tau)$ has small subgroups and a topological group with small subgroups does not admit a minimal metric (see \cite[$\S$4]{Ro18-2}).
\end{remark}
\subsection{Haagerup property}
It is shown in \cite[Theorem 4.17 and Theorem 4.28]{AM20} that if $M$ is a type II$_{\infty}$ factor, then $\mathcal{U}_p(M)$ does not have Property (FH). We show here that it actually has the Haagerup property. 

\begin{theorem}\label{thm: properaction of U_p(M)}Let $(M,\tau)$ be as in Definition \ref{def: p-unitary group} and $1\le p<\infty$. The following statements hold. 
\begin{list}{}{}
\item[{\rm{(i)}}] The compatible bi-invariant metric on $\mathcal{U}_p(M,\tau)$ given by $d(u,v):=\|u-v\|_p\ (u,v\in \mathcal{U}_p(M,\tau))$ is a maximal metric. It is unbounded if and only if $M$ has infinite direct summand.    
\item[{\rm{(ii)}}] 
The group $\mathcal{U}_p(M,\tau)$ admits the following continuous coarsely proper affine isometric action on $L^p(M,\tau)$ given by  
\[u\cdot x:=ux+(u-1),\ \ u\in \mathcal{U}_p(M,\tau),\,\,x\in L^p(M,\tau).\]
In particular, the group $\mathcal{U}_2(M,\tau)$ has the Haagerup property.
\end{list} 
\end{theorem} 
\begin{lemma}\label{lem: Lpunitary}
Let $M$ be a semifinite von Neumann algebra with a faithful normal semifinite trace $\tau$. Then every $u\in \mathcal{U}_p(M,\tau)\ (1\le p<\infty)$ is of the form $u=e^{\ri a}$, where $a\in L^p(M,\tau)_{\rm{sa}}\cap M$. Moreover, the following inequality holds. 
\begin{equation}
\tfrac{1}{2}\|a\|_p\le \|e^{\ri a}-1\|_p\le \|a\|_p,\,\,a\in L^p(M,\tau)_{\rm{sa}}\label{eq: exp length and p-norm are bilipschitz}
\end{equation}
\end{lemma}
\begin{proof}
Let $u\in \mathcal{U}_p(M,\tau)$. Then by spectral theory $u=e^{\ri a}$ for some $a=a^*\in M$ with $\|a\|_{\infty}\le \pi$. Let $a=\int_{[-\pi,\pi]}\lambda\,de(\lambda)$ be the spectral resolution of $a$. Then 
\[\|u-1\|_p^p=\tau(|e^{\ri a}-1|^p)=\int_{[-\pi,\pi]}\left (2(1-\cos \lambda)\right )^{\frac{p}{2}}\,d\tau(e(\lambda)).\]
By the elementary inequality 
\begin{equation}
\tfrac{1}{4}x^2\le 2(1-\cos x)\le x^2\ (|x|\le \pi)\label{eq: very elementary ineq},
\end{equation}
it follows that 
\[2^{-p}\|a\|_p^p=\int_{[-\pi,\pi]}(\tfrac{\lambda^2}{4})^{\frac{p}{2}}\,d\tau(e(\lambda))\le \|u-1\|_p^p<\infty,\]
which shows that $a\in L^p(M,\tau)$ and $\tfrac{1}{2}\|a\|_p\le \|e^{\ri a}-1\|_p$. 
By (\ref{eq: very elementary ineq}) again, we see that $\tfrac{1}{2}\|a\|_p\le \|e^{\ri a}-1\|_p\le \|a\|_p$ holds for every $a\in M\cap L^p(M,\tau)_{\rm{sa}}$, whence by density for every $a\in L^p(M,\tau)_{\rm{sa}}$.  
\end{proof}
\begin{proof}[Proof of Theorem \ref{thm: properaction of U_p(M)}]

(i) We use Proposition~\ref{prop:maximalmetric}: $d$ is maximal if and only if $(\mathcal{U}_p(M,\tau),d)$ is coarsely proper and large scale geodesic (recall the definitions from Definitions~\ref{def:coarselyproper} and~\ref{def:largescalegeodesic}, respectively). 
First we show that $d$ is coarsely proper.
Let $\Delta,\delta>0$, and $u\in \mathcal{U}_p(M,\tau)$ with $\|u-1\|_p<\Delta$ be given. Fix $k\in \mathbb{N}$ such that $\max(\frac{\pi}{\delta},\frac{2\Delta}{\delta})<k$ holds. By Lemma \ref{lem: Lpunitary}, $u$ is of the form $u=e^{\ri a}$ where $a=a^*\in L^p(M,\tau)\cap M_{{\rm sa}}$ and the inequality (\ref{eq: exp length and p-norm are bilipschitz}) in Lemma \ref{lem: Lpunitary} holds. 

Define $u_j=e^{\ri \frac{ja}{k}}\ (0\le j\le k)$. Then $u_0=1,\,u_k=u$ and for each $0\le j\le k-1$, 
$$d(y_{j+1},y_j)=\|e^{\ri \frac{a}{k}}-1\|_p\le \tfrac{1}{k}\|a\|_p\le \tfrac{2}{k}\|e^{\ri a}-1\|_p<\delta.$$
Therefore $d$ is coarsely proper. 

Next we check that $d$ is large scale geodesic. Again we may check this condition only for $1$ and $u$ arbitrary. Let $u\in \mathcal{U}_p(M,\tau)$ and write $u=e^{\ri a}$ where $a=a^*\in L^p(M,\tau)\cap M_{\rm{sa}}$ and $a$ satisfies (\ref{eq: exp length and p-norm are bilipschitz}). We will show that $K=2$ works and that $n$ is chosen to satisfy $n\le 2d(x,y)$.  
If $\|a\|_p\le 1$, then $\|u-1\|_p\le \|a\|_p\le 2$. Assume next that $\|a\|_p\ge 1$. Then choose $n\in \mathbb{N}$ so large that $\|a\|_2\le 2n\le 2\|a\|_2$. 
Note that $n\le \|a\|_p\le 2\|u-1\|_p$ holds. Then let $u_j=e^{\ri \frac{j}{n}a}\ (0\le j\le n)$, so that $u_0=1,\,u_n=u$ and for each $0\le j\le n-1$, 
$$\|u_{j+1}-u_j\|_p=\|e^{\ri \frac{a}{n}}-1\|_p\le \tfrac{1}{n}\|a\|_p\le 2.$$
 Moreover, 
$$\sum_{j=0}^{n-1}\|u_{j+1}-u_j\|_p=n\|e^{\ri \frac{a}{n}}-1\|_p\le \|a\|_p\le 2\|e^{\ri a }-1\|_p=2\|u-1\|_p.$$
This shows that $d$ is large scale geodesic with constant $K=2$. Therefore, $d$ is a maximal metric. The last statement is immediate, because $d$ is unbounded if and only if $\tau$ is not finite, if and only if $M$ has nonzero infinite direct summand. 

(ii) It is straightforward to see that the action is affine isometric and continuous. Note that the associated 1-cocycle $b(u)=u-1\ (u\in \mathcal{U}_p(M,\tau))$ defines a coarse embedding of $\mathcal{U}_p(M,\tau)$ into $L^p(M,\tau)$ because of the equality $\|b(u)-b(1)\|_p=d(u,1)\,(u\in \mathcal{U}_p(M,\tau))$.  
This finishes the proof.
\end{proof}

\subsection{Amenability} 
Recall that a topological group $G$ is \textit{amenable} if any continuous affine action of $G$ on a non-empty compact convex subset of a locally convex topological vector space has a fixed point (see \cite[Appendix G]{BdHV} for equivalent definitions and basic properties).  
It is known that locally compact amenable Polish groups have the Haagerup property (see \cite[Examples 2.4 (2)]{Jol00}). The same implication is no longer true for general Polish groups. In \cite[Theorem 3]{Ros17}, Rosendal showed that an amenable Polish group has the Haagerup property if and only if it coarsely embeds into a Hilbert space. We show that the groups $\mathcal{U}_p(M,\tau)$ are typically non-amenable.

\begin{theorem}\label{thm: U_p(M) amenable iff M amenable}
Let $M$ be a ${\rm II}_{\infty}$ factor acting on a separable Hilbert space and $\tau$ a normal faithful semifinite trace on $M$.The following three conditions are equivalent. 
\begin{itemize}
    \item[{\rm{(i)}}] $\U_p(M,\tau)$ is amenable for some $1\le p<\infty$.
    \item[{\rm{(ii)}}] $\U_p(M,\tau)$ is amenable for every $1\le p<\infty$.
    \item[{\rm{(iii)}}] $M$ is hyperfinite. 
\end{itemize}
\end{theorem}
Thus $\mathcal{U}_p(M,\tau)$ is nonamenable if $M$ is not hyperfinite. The proof of Theorem \ref{thm: U_p(M) amenable iff M amenable} is done by standard approximation arguments. However, because $M$ is not of finite type, a care must be taken of the subtle relationship among the strong operator topology,  the topologies given by the $\|\cdot\|_p$-norms and the operator norm. We denote the operator norm by $\|\cdot\|_{\infty}$.  

\begin{lemma}\label{lem: approx by finspec}
Let $M$ be a semifinite von Neumann algebra with a normal faithful semifinite trace $\tau$ and let $1\le p<\infty$. 
Let $a\in L^p(M,\tau)_{\rm{sa}}$ and $\varepsilon>0$. Then there exists $a_0\in M_{\rm sa}\cap L^p(M,\tau)$ with finite spectrum such that $\|a-a_0\|_p<\varepsilon$ holds.  If moreover $a\in M$, then $a_0$ can be chosen to satisfy $\|a_0\|_{\infty}\le \|a\|_{\infty}$ and $a_0a=aa_0$.
\end{lemma}
\begin{proof}
Let $\disp a=\int_{\R}\lambda\,{\rm d}e(\lambda)$ be the spectral resolution of $a$. For each $n\in \N$ and $k=1,\dots, n2^n$, let $I_{n,k}=\left [\tfrac{k-1}{2^n},\tfrac{k}{2^n}\right )$. Also let $J_{n,k}=\left (-\tfrac{k}{2^n},-\tfrac{k-1}{2^n}\right ]$, for $k=2,\dots, n2^n$, and $J_{n,1}=\left (-\tfrac{1}{2^n},0\right )$. Then $\{I_{n,k},J_{n,k}\}_{k=1}^{n2^n}$ is a measurable partition of $(-n,n)$. Define $a_n=\sum_{k=1}^{n2^n}\frac{k-1}{2^n}(1_{I_{n,k}}-1_{J_{n,k}})\in M_{\rm sa}$, which has finite spectrum and $a_n\in L^p(M,\tau)$ by 
$$\|a_n\|_p^p\le \int_{(-n,n)}\lambda^p\,{\rm d}\tau(e(\lambda))\le \|a\|_p^p<\infty.$$
It is also clear that when $a$ is bounded, then $aa_n=a_na$ and $\|a_n\|_{\infty}\le \|a\|_{\infty}$ hold for every $n$. 
By a standard argument using spectral theory, one can then show that  $\disp \lim_{n\to \infty}\|a-a_n\|_p=0$. This implies the claim.
\end{proof}
It is not clear to us whether the map $L^p(M,\tau)_{\rm{sa}}\ni a\mapsto e^{\ri a}\in \mathcal{U}_p(M,\tau)$ is continuous. However, the next lemma is sufficient for our purpose. 
\begin{lemma}\label{lem: p-continuity of exp}
Let $M$ be a semifinite von Neumann algebra with a normal faithful semifinite trace $\tau$, and let $1\le p<\infty$. Let $a\in M_{\rm sa}\cap L^p(M,\tau)$ and $(a_n)_{n=1}^{\infty}$ be a sequence of elements in $M_{\rm sa}\cap L^p(M,\tau)$ such that $\sup_n\|a_n\|_{\infty}<\infty$ and $\disp \lim_{n\to \infty}\|a_n-a\|_p=0$. Then 
\begin{list}{}{}
\item[{\rm (i)}] $\disp \lim_{n\to \infty}\left \|e^{\ri a}-\sum_{k=0}^n\frac{\ri^ka^k}{k!}\right \|_p=0.$
\item[{\rm (ii)}]$\disp \lim_{n\to \infty}\|e^{\ri a_n}-e^{\ri a}\|_p=0$. 
\end{list}
\end{lemma}

\begin{proof}
Let $a=\int_{\R}\lambda\,{\rm d}e(\lambda)$ be the spectral resolution of $a$. Then for each $n\in \N$, $\sum_{k=1}^n\frac{\ri^ka^k}{k!}\in L^p(M,\tau)$, and 
\eqa{
\left \|e^{\ri a}-\sum_{k=0}^n\frac{\ri^ka^k}{k!}\right \|_p^p&=\int_{\R}\left |e^{\ri \lambda}-\sum_{k=0}^n\frac{\ri^k\lambda^k}{k!}\right |^p\,{\rm d}\tau(e(\lambda))\cdots (\star).
}
By the Taylor's theorem, there exists $\mu\in \R$ between 0 and $\lambda$ such that 
\[e^{\ri \lambda}-\sum_{k=0}^n\frac{\ri^k\lambda^k}{k!}=\frac{\ri^{n+1}e^{\ri \mu}}{(n+1)!}\lambda^{n+1}.\]

Thus, because $\tau(e(\cdot))$ is supported on $[-\|a\|_{\infty},\|a\|_{\infty}]$, we have

\eqa{
(\star)&\le \int_{[-\|a\|_{\infty},\|a\|_{\infty}]}\left \{\frac{1}{(n+1)!}\|a\|_{\infty}^n|\lambda|\right \}^p\,{\rm d}\tau(e(\lambda))\\
&=\left \{\frac{1}{(n+1)!}\|a\|_{\infty}^n\|a\|_p\right \}^p\stackrel{n\to \infty}{\to}0.
}
This proves (i).\\
(ii) First, we assume that $\sup_n\|a_n\|_{\infty}\le 1$ and $\|a\|_{\infty}\le 1$. 
By (i), for each $n\in \N$, 
\eqa{
\|e^{\ri a_n}-e^{\ri a}\|_p&=\lim_{N\to \infty}\left \|\sum_{k=0}^N\frac{\ri^k}{k!}(a_n^k-a^k)\right \|_p.
}
Then 
\eqa{
\left \|\sum_{k=0}^N\frac{\ri^k}{k!}(a_n^k-a^k)\right \|_p&\le \|a_n-a\|_p+\sum_{k=2}^N\frac{1}{k!}\|a_n^k-a^k\|_p\cdots (\star \star).
}
For each $N\ge 2$ and $k=2,\dots, N$, 
\eqa{
\|a_n^k-a^k\|_p&\le \|(a_n-a)a_n^{k-1}\|_p+\|a(a_n^{k-1}-a^{k-1})\|_p\\
&\le \|a\|_{\infty}^{k-1}\|a_n-a\|_p+\|a\|_{\infty}\|a_n^{k-1}-a^{k-1}\|_p\\&\le \|a_n-a\|_p+\|a_n^{k-1}-a^{k-1}\|_p\le \dots\\
&\le k\|a_n-a\|_p.
}
Therefore 
\[(\star \star)\le \|a_n-a\|_p\left (1+\sum_{k=2}^{\infty}\frac{1}{(k-1)!}\right )\stackrel{n\to \infty}{\to}0.\]
This shows (ii) in the case where all $a_n, a$ are contractions.\\
Now we consider the general case. Take $m\in \N$ such that $\|a_n\|_{\infty},\|a\|_{\infty}\le m\,(n\in \N)$ holds. Then $\tfrac{1}{m}a_n,\tfrac{1}{m}a\,(n\in \N)$ are contractions. Thus by the above argument, 
\eqa{
\|e^{\ri a_n}-e^{\ri a}\|_p&=\|e^{\ri \tfrac{1}{m}a_n}\cdots e^{\ri \tfrac{1}{m}a_n}-e^{\ri \tfrac{1}{m}a}\cdots e^{\ri \tfrac{1}{m}a}\|_p\\
&\le \|(e^{\ri \tfrac{1}{m}a_n}-e^{\ri \tfrac{1}{m}a})e^{\ri \tfrac{1}{m}a_n}\cdots e^{\ri \tfrac{1}{m}a_n}\|_p+\|e^{\ri \tfrac{1}{m}a}(e^{\ri \tfrac{1}{m}a_n}-e^{\ri \tfrac{1}{m}a})e^{\ri \tfrac{1}{m}a_n}\cdots e^{\ri \tfrac{1}{m}a_n}\|_p+\cdots \\
&\hspace{0.5cm}+\|e^{\ri \tfrac{1}{m}a}\cdots e^{\ri \tfrac{1}{m}a}(e^{\ri \tfrac{1}{m}a_n}-e^{\ri \tfrac{1}{m}a})\|_p\\
&=m\|e^{\ri \tfrac{1}{m}a_n}-e^{\ri \tfrac{1}{m}a}\|_p\stackrel{n\to \infty}{\to}0.
}
Thish proves (ii). 
\end{proof}

Until the end of this section, we fix notation, which will be used in the proof of (ii)$\Rightarrow$(i) in Theorem \ref{thm: U_p(M) amenable iff M amenable}. Fix $1\le p<\infty$ and a separable infinite-dimensional Hilbert space $\Hil$ with an orthonormal basis $(\xi_n)_{n=1}^{\infty}$. For each $n\in \N$, let $f_n$ be the projection onto the subspace spanned by $\{\xi_1,\dots,\xi_n\}$. Let $R_n:=M_2(\C)^{\otimes n}\cong M_{2^n}(\C)$ and 
\begin{equation}M_n:=R_n\otimes f_n\mathbb{B}(\Hil)f_n\oplus \C 1_{R_n}\otimes f_n^{\perp}\subset R\overline{\otimes}\mathbb{B}(\Hil).
\end{equation}
Then $(M_n)_{n=1}^{\infty}$ is an increasing sequence of finite-dimensional $*$-subalgebras of the hyperfinite II$_{\infty}$ factor $M:=R\overline{\otimes}\mathbb{B}(\Hil)$ endowed with the normal faithful semifinite trace $\tau:=\tau_0\otimes {\rm Tr}$ ($\tau_0$ is the normal faithful tracial state on the hyperfinite II$_1$ factor $R$ and Tr denotes the operator trace). We say that an element $a\in M\cap L^p(M,\tau)$ is {\it p-approximable} if there exists a sequence $(a_n)_{n=1}^{\infty}$ of elements in $M_0\cap L^p(M,\tau)$, where $M_{0}:=\bigcup_{k=1}^{\infty}M_k$, such that $\sup_n\|a_n\|_{\infty}\le \|a\|_{\infty}$ and $\disp \lim_{n\to \infty}\|a-a_n\|_p=0$ hold. Note that $M_0$ is a $*$-strongly dense unital subalgebra of $M$. 
\begin{proposition}\label{prop: p-approx}
Any $a\in M_{\rm{sa}}\cap L^p(M,\tau)$ is $p$-approximable. 
\end{proposition}
The following lemma may be known. However, because we could not find a reference, we add a proof.
\begin{lemma}\label{lem: continuity of Lprep}
Let $M$ be a semifinite von Neumann algebra with a normal faithful semifinite trace $\tau$, and let $1\le p<\infty$. Let $(x_i)_{i\in I}$ be a bounded net in $M$ converging strongly to $x\in M$. Then for every $a\in L^p(M,\tau)$, $\displaystyle \lim_{i\to \infty}\|(x_i-x)a\|_p=0$ holds.  
\end{lemma}

\begin{proof}
Since $L^p(M)$ is spanned by self-adjoint elements, we may assume that $a$ is self-adjoint. We may also assume that $x_i,x$ are contractions.\\ \\
\textbf{Step 1.} Consider the case that $a$ has finite spectrum (thus $a\in L^q(M,\tau)$ for every $1\le q\le \infty$).\\
Suppose first that $p\ge 2$. Then we use the following inequality
\begin{equation}
\|x\|_p\le \|x\|_{\infty}^{1-\frac{2}{p}}\|x\|_2^{\frac{2}{p}},\,\,x\in M\cap L^2(M,\tau)(\subset L^p(M,\tau)).\label{eq: p and 2 norm}
\end{equation}
By (\ref{eq: p and 2 norm}) and the fact that $M$ is represented on $L^2(M,\tau)$ by left multiplications so that the strong convergence of $(x_i)_{i\in I}$ is the pointwise 2-norm convergence in $L^2(M,\tau)$, we have 
\eqa{
\|(x_i-x)a\|_p&\le \|(x_i-x)a\|_{\infty}^{1-\frac{2}{p}}\|(x_i-x)a\|_2^{\frac{2}{p}}\\
&\le 2^{1-\frac{2}{p}}\|a\|_{\infty}^{1-\frac{2}{p}}\|(x_i-x)a\|_2^{\frac{2}{p}}\\
&\stackrel{i\to \infty}{\to}0.
}
Next, suppose $1\le p<2$.  Since $a$ is self-adjoint and has finite spectrum, it is a real linear combination of $\tau$-finite projections. To show that $\|(x_i-x)a\|_p$ tends to 0, by triangle inequality we may assume that $a$ itself is a $\tau$-finite projection. Let $2\le r<\infty$ be such that 
$\tfrac{1}{p}=\tfrac{1}{2}+\tfrac{1}{r}$. By $a=a^2$, we have 
by the generalized non-commutative H\"older's inequality(\cite[Corollaire 3]{Dix53}. Cf.\cite[Lemma 5.9]{PiXu03}), 
\eqa{
\|(x_i-x)a\|_p&=\|(x_i-x)a\cdot a\|_p\le \|(x_i-x)a\|_2\|a\|_r\\
&\stackrel{n\to \infty}{\to}0. 
}
\textbf{Step 2.} Consider the general case. Let $\varepsilon>0$. 
By Lemma \ref{lem: approx by finspec}, there exists $a_0\in M_{\rm{sa}}\cap L^p(M,\tau)$ with finite spectrum such that $\|a-a_0\|_p<\frac{\varepsilon}{4}$ holds. By Step 1, $\disp \lim_{i\to \infty}\|(x_i-x)a_0\|_p=0$ holds. Thus there exists $i_0\in I$ such that for every $i\ge i_0$, $\|(x_i-x)a_0\|_p<\frac{\varepsilon}{2}$ holds. 
Then for every $i\ge i_0$, 
\eqa{
\|(x_i-x)a\|_p&\le \|(x_i-x)(a-a_0)\|_p+\|(x_i-x)a_0\|_p\\
&<\|x_i-x\|_{\infty}\|a-a_0\|_p+\frac{\varepsilon}{2}\\
&<\varepsilon.
}
Since $\varepsilon$ is arbitrary, the claim follows. 
\end{proof}
\begin{proof}[Proof of Proposition \ref{prop: p-approx}]
Let $\tilde{M}$ be the set of all $a\in M\cap L^p(M,\tau)$ which are $p$-approximable. 
\\ \\
\textbf{Step 1.} We first show that if $a\in \tilde{M}$ is a self-adjoint element with finite spectrum and if $u\in \mathcal{U}(M)$, then $uau^*\in \tilde{M}$. Indeed, by $a\in \tilde{M}$, there exists a sequence $(a_n)_{n=1}^{\infty}$ in $M_0\cap L^p(M,\tau)$ with $\sup_n\|a_n\|_{\infty}\le \|a\|_{\infty}$ such that $\disp \lim_{n\to \infty}\|a_n-a\|_p=0$. There exists a sequence $(u_n)_{n=1}^{\infty}$ in $\mathcal{U}(M_0)$ converging $*$-strongly to $u$. Then $(u_na_nu_n^*)_{n=1}^{\infty}$ is a sequence in $M_0\cap L^p(M,\tau)$ with $\sup_n\|u_na_nu_n^*\|_{\infty}\le \|a\|_{\infty}$, and (we use $\|x\|_p=\|x^*\|_p$ and $\|axb\|_p\le \|a\|_{\infty}\|x\|_p\|b\|_{\infty}$ for $a,b\in M$ and $x\in M\cap L^p(M,\tau)$)
\eqa{
\|u_na_nu_n^*-uau^*\|_p&\le \|(u_n-u)a_nu_n^*\|_p+\|u(a_n-a)u_n^*\|_p+\|ua(u_n^*-u^*)\|_p\\
&\le \|(u_n-u)a_n\|_p+\|a_n-a\|_p+\|(u_n-u)a^*u^*\|_p\\
&\le \|(u_n-u)(a_n-a)\|_p+\|(u_n-u)a\|_p+\|a_n-a\|_p+\|(u_n-u)a^*\|_p\\
&\le 2\|a_n-a\|_p+2\|(u_n-u)a\|_p+\|a_n-a\|_p.
}
The last term tends to 0 as $n\to \infty$ because  $\displaystyle \lim_{n\to \infty}\|(u_n-u)a\|_p=0$ thanks to Lemma \ref{lem: continuity of Lprep}. 
Thus $\displaystyle \lim_{n\to \infty}\|u_na_nu_n^*-uau^*\|_p=0$, which  shows that $uau^*\in \tilde{M}$ as we wanted.\\ \\
\textbf{Step 2.} Next, we show that $\tilde{M}$ contains all elements in $M_{\rm{sa}}\cap L^p(M,\tau)$ which have finite spectrum. To show this, let $a\in M_{\rm{sa}}\cap L^p(M,\tau)$ be an element with finite spectrum and let $\varepsilon>0$. There exist nonzero $\tau$-finite projections $p_1,\dots,p_m\in M$ and $\lambda_1,\dots,\lambda_m\in \R$ such that $a=\sum_{i=1}^m\lambda_ip_i$. For each $i=1,\dots,m$, write $\tau(p_i)=d_i+s_i$, where $d_i\in \mathbb{Z}_{\ge 0}$ and $0\le s_i<1$. 
Take large enough $n\in \N$ so that $n>\sum_{i=1}^m(d_i+1)$ and for each $i=1,\dots,m$, there exists $k_i\in \{0,\dots,2^n-1\}$ such that if we set $t_i=d_i+\frac{k_i}{2^n}$, then $t_i\leq\tau(p_i)$ and the following inequality holds: 
\[\sum_{i=1}^m|\lambda_i|^p(\tau(p_i)-t_i)<(\tfrac{\varepsilon}{2})^p.\]

Let $e_1,\dots,e_{2^n}$ be pairwise orthogonal minimal projections in $R_n\cong M_{2^n}(\C)$ with $\sum_{i=1}^{2^n}e_i=1$. For each $n\in \N$ let $r_n$ be the rank one projection onto $\C\xi_n$. 
Set 
\[q_i:=1_{R_n}\otimes \sum_{j=d_i'+1}^{d_i'+d_i}r_j+\sum_{\ell=1}^{k_i}e_{\ell}\otimes r_{d_i'+d_i+1},\,i=1,\dots,m,\]
where $d_1':=0$ and $d_i':=\sum_{j=1}^{i-1}(d_j+1)\,(i=2,\dots,m)$. Note that some terms in the above sum might be disregarded when $d_i=0$ (in this case the first sum is disregarded) or $k_i=0$ (in this case the second sum is disregarded). 
Note that $q_i\in R_n\otimes f_n\mathbb{B}(\Hil)f_n\subset M_n\,(1\le i\le m)$. 
Then $\{q_i\}_{i=1}^m$ are mutually orthogonal projections in $M_0$ and $\tau(q_i)=d_i+\tfrac{k_i}{2^n}=t_i\le \tau(p_i)\,(1\le i\le m)$ holds. Let $1\le i\le m$. Because $M$ is a type II$_{\infty}$ factor, there exists a projection $p_i'\le p_i$ such that $p_i'\sim q_i$ in $M$ holds (here, $\sim$ stands for the Murray-von Neumann equivalence of projections in $M$). Let $u_i\in M$ be a partial isometry satisfying 
\[u_i^*u_i=q_i,\,\,u_iu_i^*=p_i'.\]
Also, $1-\sum_{i=1}^mq_i\sim 1-\sum_{i=1}^mp_i'$ (both being infinite projections in $M$), we may find a partial isometry $u_0\in M$ satisfying 
\[u_0^*u_0=1-\sum_{i=1}^mq_i,\,\,\,u_0u_0^*= 1-\sum_{i=1}^mp_i'.\]
Then $u=\sum_{i=0}^{m}u_i$ is a unitary in $M$ satisfying $uq_iu^*=p_i'\,(1\le i\le m)$. This shows that $ua_0u^*=a'$, where $a_0\in M_0\cap L^p(M,\tau)$ and $a'\in M\cap L^p(M,\tau)$ are self-adjoint elements defined by
\[a_0:=\sum_{i=1}^m\lambda_iq_i,\,\,a':=\sum_{i=1}^m\lambda_ip_i'.\]
Moreover, $\|a-a'\|_p=\left (\sum_{i=1}^m|\lambda_i|^p(\tau(p_i)-t_i)\right )^{\frac{1}{p}}<\tfrac{\varepsilon}{2}$, and by $a_0\in M_0\cap L^p(M,\tau)$, and Step 1, $a'\in \tilde{M}$ holds. Thus, there exists $a''\in M_0\cap L^p(M,\tau)$ with $\|a''\|_{\infty}\le \|a'\|_{\infty}\le \|a\|_{\infty}$ such that $\|a'-a''\|_p<\tfrac{\varepsilon}{2}$ holds. Therefore we have $\|a-a''\|_p<\varepsilon$. Since $\varepsilon$ is arbitrary, we obtain $a\in \tilde{M}$.    
\\ \\
\textbf{Step 3.} Let $a\in M_{\rm sa}\cap L^p(M,\tau)$ and $\varepsilon>0$. Then by Lemma \ref{lem: approx by finspec}, there exists $a_0\in M_{\rm sa}\cap L^p(M,\tau)$ with finite spectrum such that $\|a_0\|_{\infty}\le \|a\|_{\infty}$ and $\|a-a_0\|_p<\tfrac{\varepsilon}{2}$ hold. Since $a_0\in \tilde{M}$ by Step 2, there exists $a_1\in M_0\cap L^p(M,\tau)$ such that $\|a_1\|_{\infty}\le \|a_0\|_{\infty}$ and $\|a_0-a_1\|_p<\tfrac{\varepsilon}{2}$ hold. Thus $\|a_1\|_{\infty}\le \|a\|_{\infty}$ and $\|a-a_1\|_p<\varepsilon$. Since $\varepsilon$ is arbitrary, we obtain $a\in \tilde{M}$. This concludes the proof. 
\end{proof}
\begin{proof}[Proof of Theorem \ref{thm: U_p(M) amenable iff M amenable}]
(ii)$\Rightarrow$(i) is clear. 
(i)$\Rightarrow$(iii) Fix $1\le p<\infty$ such that $\mathcal{U}_p(M,\tau)$ is amenable. 
Let $T\in \mathbb{B}(L^2(M,\tau))$ be arbitrary.
We set $K:=\overline{\textrm{co}}\{uTu^* \mid\ u\in \mathcal{U}(M)\}$, where the closure is with respect to the weak operator topology.
Then $K$ is a compact convex subset, and $\mathcal{U}_p(M,\tau)$ acts on it by conjugation. Clearly the action is continuous and affine. 
Since $\mathcal{U}_p(M,\tau)$ is amenable, we can find an element $y$ in $K$ so that $uyu^*=y$ holds for any $u\in\mathcal{U}_p(M,\tau)$.
For any finite projection $p$ in $M$, the operator $1-2p$ is in $\mathcal{U}_p(M,\tau)$.
Thus $y$ commutes with all finite projections in $M$.
Since $M$ is generated by the set of finite projections, $y$ is in $M'$.
Therefore $M$ fulfills the Schwartz' property P, which implies that $M$ is hyperfinite by Theorem \ref{thm: Connes}.\medskip

(iii)$\Rightarrow$(ii) Let $1\le p<\infty$. 
For each $n\in \N$, define $G_n:=\mathcal{U}(M_n)\cap \mathcal{U}_p(M,\tau)$\footnote{
An element in $\mathcal{U}(M_n)$ is of the form $u_0\oplus e^{\ri \theta}1_{R_n}\otimes f_n^{\perp}$, which is in $\mathcal{U}_p(M,\tau)$ if and only if $e^{\ri \theta}=1$ (recall that $f_n^{\perp}$ is infinite-rank). 

} with the $\|\cdot\|_p$-metric with respect to $\tau:=\tau_0\otimes {\rm Tr}$. Note that an element $u\in G_n$ is of the form $u_0\oplus 1\otimes f_n^{\perp}$, where $u_0$ is a unitary in $A_n:=M_{2^n}(\C)\otimes f_n\mathbb{B}(\Hil)f_n$. Thus $u-1=u_0-f_n\in L^p(M,\tau)$. Then $G_1\subset G_2\subset \cdots $ is an increasing sequence of compact subgroups of $\mathcal{U}_p(M,\tau)$. We show that $G_0:=\bigcup_{n=1}^{\infty}G_n$ is $\|\cdot\|_p$-dense in $\mathcal{U}_p(M,\tau)$, which will imply that $\mathcal{U}_p(M,\tau)$ is amenable \cite[Proposition G.2.2 (iii)]{BdHV}.      
Let $u\in \mathcal{U}_p(M,\tau)$ and $\varepsilon>0$. Then there exists $a\in M_{\rm sa}\cap L^p(M,\tau)$ such that $u=e^{\ri a}$. By (the proof of) Proposition \ref{prop: p-approx}, there exists a sequence $(a_n)_{n=1}^{\infty}$ in $M_{0,{\rm sa}}\cap L^p(M,\tau)$ such that $\sup_{n}\|a_n\|_{\infty}<\infty$ and $\disp \lim_{n\to \infty}\|a_n-a\|_p=0$. Then $u_n:=e^{\ri a_n}\in G_0\, (n\in \N)$ and by Lemma \ref{lem: p-continuity of exp}, we have $\disp \lim_{n\to \infty}\|u-u_n\|_p=0$. This shows that $G_0$ is $\|\cdot\|_p$-dense in $\mathcal{U}_p(M,\tau)$. 
\end{proof}

\section{The groups $\E_n(A)$ and $\SL(n,A)$}\label{section:En}
The special linear groups $\SL(n,K)$, $K\in\{\Rea,\Com\}$, are undoubtedly one of the most important Lie groups. In the case $n\geq 3$ they are the prominent examples of groups with Property (T) and in the case $n=2$ they are the prominent examples of groups with the Haagerup property. The groups $\SL(n,R)$, and their relatives elementary groups $\E_n(R)$, have been heavily investigated with regard to Property (T) for much more general rings $R$ (see e.g. \cite{Kas07} and \cite{ErJZ10}). Here we consider their variants which are unbounded topological groups, and in many cases they are Banach-Lie groups.

\subsection{The structure of $\E_n(A)$}
Let $A$ be a unital Banach algebra over the complex or real numbers. Let $n\geq 2$ be a natural number and let $M_n(A)$ denote the unital Banach algebra of all $n\times n$ matrices with coefficients from $A$. By $\GL(n,A)$, we shall denote the group $\Inv(M_n(A))$, i.e. the topological group of all invertible matrices with matrix entries in $A$. This is a Banach-Lie group with $M_n(A)$ as a Banach-Lie algebra (see \cite[Proposition IV.9]{Neeb04}).

For any $i\neq j\leq n$, let $E_{i,j}(a)$ be the elementary matrix which has the units on the diagonal and the element $a\in A$ on the $(i,j)$-th coordinate. By $E_{i,j}(A)$ we denote the closed subgroup $\{E_{i,j}(a)\mid a\in A\}$ of $\GL(n,A)$, and by $\E_n(A)$ we denote the closed subgroup of $\GL(n,A)$ generated by all the subgroups $E_{i,j}(A)$, where $i\neq j\leq n$. We note that we do not know whether the subgroup algebraically generated by the collection of subgroups $\{E_{i,j}(A)\mid i\neq j\leq n\}$ is always closed in $\GL(n,A)$. Therefore $\E_n(A)$ is meant to be the closure of $\langle E_{i,j}(A)\mid i\neq j\leq n\rangle$. The normal subgroup structure of these groups was studied in \cite{Va86}.

Notice that $\E_n(A)$ is connected as its dense subgroup consisting of products of elementary matrices is connected. In case when $A$ is commutative, we denote by $\SL(n,A)$ the subgroup of $\GL(n,A)$ (and supergroup of $\E_n(A)$) of matrices of determinant $1$, the unit of $A$. If $A$ is separable, $\E_n(A)$, resp. $\SL(n,A)$ are Polish groups.

Additionally, if $A$ is an abelian unital Banach algebra (real or complex), then $\SL(n,A)$ is a Banach-Lie group with Banach-Lie algebra $\mathfrak{sl}(n,A)$, where the latter is the set of all trace-less matrices from $M_n(A)$ (with the Lie bracket $[X,Y]=XY-YX$). This is clear when $A$ is a unital abelian $C^*$-algebra. Since if in this case $X$ is the Gelfand spectrum of $A$, i.e. $A=C(X)$, then $\SL(n,A)$ is canonically isomorphic to the group $C(X,\SL(n,\Com))$, resp. $C(X,\SL(n,\Rea))$ if $A=C(X,\Rea)$ is the real algebra. The latter is a Banach-Lie group with Banach-Lie algebra $C(X,\mathfrak{sl}(n,\Com))$, resp. $C(X,\mathfrak{sl}(n,\Rea))$, which are again canonically isomorphic to $\mathfrak{sl}(n,A)$; see \cite[Section I.2]{Neeb05}.

For general $A$, we can use the fact that $\SL(n,A)$ is an `algebraic subgroup' of $\GL(n,A)$, thus a Banach-Lie subgroup. We refer to \cite[Theorem 1]{HK77}.\\

This has the effect that we can say something also about $\E_n(A)$, for abelian algebras $A$. We record it in the following lemma.
\begin{lemma}\label{lem:EnInSln}
Let $A$ be an abelian unital Banach algebra over the real or complex numbers. Let $n\geq 2$. Then $\E_n(A)$ is the connected component of the identity in $\SL(n,A)$. In particular, it is a Banach-Lie group.
\end{lemma}
\begin{proof}
Clearly, $\E_n(A)$ is a closed subgroup of $\SL(n,A)$. Although we do not a priori know that $\E_n(A)$ is a Banach-Lie group, since it is a closed subgroup of a Banach-Lie group, $\LA:=\{X\in\mathfrak{sl}(n,A)\mid \forall t\geq 0\; (\exp(tX)\in \E_n(A)\}$ is a well-defined Banach-Lie subalgebra of $\mathfrak{sl}(n,A)$ (see \cite[Corollary IV.3]{Neeb04}). For $i\neq j\leq n$ and $a\in A$ denote by $e_{i,j}(a)$ the matrix $E_{i,j}(a)-\mathrm{Id}\in\mathfrak{sl}(n,A)$, and by $f_{i,j}(a)\in \mathfrak{sl}(n,A)$ the matrix containing $a$ at the $(i,i)$-th entry, $-a$ at the $(j,j)$-th entry, and zeros elsewhere. Matrices of the form $e_{i,j}(a)$ and $f_{i,j}(a)$ span the Banach-Lie algebra $\mathfrak{sl}(n,A)$. The verification is the same as in the finite-dimensional case.

Since $\E_n(A)$ contains all elementary matrices $E_{i,j}(a)$, for $i\neq j\leq n$, $a\in A$, $\LA$ contains all matrices $e_{i,j}(a)$. Since $f_{i,j}(a)=[e_{i,j}(a),e_{j,i}(1)]$, we get that $\LA$ contains the basis of $\mathfrak{sl}(n,A)$, and thus $\LA=\mathfrak{sl}(n,A)$. By the Lie theory, the connected component of the identity of $\SL(n,A)$ is equal to $\{\exp(X_1)\cdots\exp(X_n)\mid X_1,\ldots,X_n\in\mathfrak{sl}(n,A)=\LA\}$. The latter set, denoted by $G$, is a dense subgroup of $\E_n(A)$. Indeed, by definition of $\LA$, $G\subseteq \E_n(A)$, and since $\E_n(A)$ is by definition the closure of the subgroup generated by matrices $E_{i,j}(a)$, which is contained in $G$, we get that $G$ is dense in $\E_n(A)$. Since $G$ is closed, these groups are equal.
\end{proof}

For $A$ non-commutative, the situation is more delicate and we do not know whether in general $\E_n(A)$ is a Banach-Lie group. We do know it under certain $K$-theoretic assumption on $A$.

For the rest of this subsection, we shall assume a mild familiarity with the $K_0$ group of Banach algebras. We refer to \cite{RLLbook} for any unexplained notion. Specifically, we want to recall that for any Banach algebra $A$ and any tracial linear operator $\tau: A\rightarrow E$, where $E$ is a Banach space, i.e. a linear map satisfying $\tau(xy-yx)=0$, we have a unique associated group homomorphism $\tau_*: K_0(A)\rightarrow E$; see \cite[Section 3.3.1]{RLLbook} for details. We shall also work with the \emph{de la Harpe-Skandalis determinant}. Its definition will be recalled later, we refer the reader to \cite{dlHS84} and \cite{dlH13} for more information.\\

Fix now $n\geq 2$ and $A$. Let $\LA:=\{X\in M_n(A)\mid \forall t\geq 0\; (\exp(tX)\in \E_n(A))\}$. It is a closed Lie subalgebra of $M_n(A)$.

Let $\tau:A\rightarrow A/[A,A]=: E$ be the universal tracial operator, i.e. the projection from $A$ onto $E:=A/[A,A]$, where $[A,A]$ is the closed linear span of commutators. Let $\tau_n: M_n(A)\rightarrow E$ be $\tau\circ \mathrm{Tr}$, i.e. $\tau_n((x_{i,j})_{i,j\leq n})=\tau(\sum_{i\leq n} x_{ii})$. Set $L:=\{X\in M_n(A)\mid \tau_n(X)=0\}$.
\begin{proposition}
 $L$ is a Banach-Lie algebra and we have $L\subseteq \LA$. If $K_0(A)$ is trivial, we have equality and $\E_n(A)$ is a Banach-Lie group equal to the (connected component of the identity of the) Banach-Lie subgroup $G_{HS}$ of $\mathrm{GL}(n,A)$ of elements with vanishing de la Harpe-Skandalis determinant. If $K_0(A)$ is finitely generated, then $\E_n(A)$ is still a Banach-Lie group, subgroup of $G_{HS}$.
\end{proposition}
\begin{proof}
Since $\tau_n$ is a continuous linear map and for every $X,Y\in M_n(A)$ (not necessarily from $L$), $\tau_n([X,Y])=0$, it follows that $L$ is a closed Lie subalgebra of $M_n(A)$.\\ \\
\textbf{Step 1.} We claim that $L$ is spanned by the following elements:
\begin{itemize}
    \item $e_{i,j}(a)$, for $i\neq j\leq n$ and $a\in A$, which contains $a$ at the $(i,j)$-th entry and zeros elsewhere;
    \item $f_{i,j}(a)$, for $i\neq j\leq n$ and $a\in A$, which contains $a$ at the $(i,i)$-th entry, $-a$ at the $(j,j)$-th entry, and zeros elsewhere;
    \item $g(a)$, where $a\in [A,A]$, and $g(a)$ is a matrix containg $a$ at the $(n,n)$-th entry and having zeros elsewhere.
\end{itemize}
Indeed, let $(x_{i,j})_{i,j\leq n}\in L$. By subtracting elements of the form $e_{i,j}(x_{i,j})$, for $i\neq j\leq n$, we may suppose that $x_{i,j}=0$ if $i\neq j$. Analogously, by subtracting the element of the form $f_{1,2}(x_{1,1})$, we may suppose that $x_{1,1}=0$. Continuing in the same fashion, we are reduced to the case when $x_{i,j}\neq 0$ if and only if $i=j=n$. Then we must have $\tau(x_{n,n})=0$, so $x_{n,n}\in [A,A]$, and we are done.\\ \\
\textbf{Step 2.} We claim that $L\subseteq \LA$. It suffices to show that all elements of the form $e_{i,j}(a)$, $f_{i,j}(a)$, and $g(b)$ are present in $\LA$. It is clear that all elements $e_{i,j}(a)$ belong to $\LA$. Since $f_{i,j}(a)=[e_{i,j}(a),e_{j,i}(1)]$, also all elements of the form $f_{i,j}(a)$ belong to $\LA$. Finally, for any $a,b\in A$, we have $g(ab-ba)=[e_{n,1}(a),e_{1,n}(b)]-f_{1,n}(ba)$, thus $g(ab-ba)\in\LA$. Since $\LA$ is closed, $g(c)\in\LA$ for all $c\in [A,A]$.\\ \\
\textbf{Step 3.} Set $G=\mathrm{GL}(n,A)$, $E=A/[A,A]$, let $\tau:A\rightarrow E$ be the projection, the universal tracial linear operator. Let $\tau_*: K_0(A)\rightarrow E$ be the associated homomorphism. Let $\Delta_\tau: G_0\rightarrow E/\overline{\tau_*[K_0(A)]}$ be the de la Harpe-Skandalis determinant associated to $\tau$. That is, for any $g=\exp(X_1)\cdots\exp(X_n)\in G_0$, for $X_1,\ldots,X_n\in M_n(A)$, we have \[\Delta_\tau(g)=\sum_{i=1}^n P(\tau_n(X_i)),\] where $P:E\rightarrow E/\overline{\tau_*[K_0(A)]}$ is the projection. The definition does not depend on the decomposition $\exp(X_1)\cdots\exp(X_n)$ (see again \cite{dlHS84} and \cite{dlH13} for more details).

We set $G_{HS}=\mathrm{ker}(\Delta_\tau)$. Since $\Delta_\tau$ is a continuous group homomorphisms between Banach-Lie groups, by \cite[Proposition IV.3.4]{Neeb06}, $G_{HS}$ is a Banach-Lie subgroup of the Banach-Lie group $G$. Since for $i\neq j\leq n$ and $a\in A$, we have $\Delta_\tau(E_{i,j}(a))=0$, we have that $\E_n(A)\subseteq G_{HS}$. Let $\tilde{L}$ be the Banach-Lie algebra of $G_{HS}$. It follows from the inclusion $\E_n(A)\subseteq G_{HS}$ that we have also $\LA\subseteq \tilde{L}$.\\ \\
\textbf{Step 4.} Let $F$ be the closed real linear span of $\tau_*[K_0(A)]$ and set $Z:=E/F$. Let $P_Z: E\rightarrow Z$ be the projection. We set $\hat{L}:=\{X\in M_n(A)\mid P_Z\circ\tau_n(X)=0\}$. We claim that $\tilde{L}\subseteq \hat{L}$. If $X\in M_n(A)$ does not belong to $\hat{L}$, then $P_Z\circ\tau_n(X)\neq 0$. So $\Delta_\tau(\exp(X))=P(\tau_n(X))\neq 0$, thus $X\notin\tilde{L}$.\\ \\
\textbf{Step 5a.} Suppose now that $K_0(A)=\{0\}$. Then we claim that $\hat{L}\subseteq L$ which will immediately imply that $\hat{L}=\tilde{L}=L=\LA$. Indeed, let $X\in\hat{L}$. Then since $P$ and $P_Z$ are trivial, $\tau_n(X)=0$. It follows that $X\in L$.

Finally, we claim that in this case $\E_n(A)=G_{HS,0}$, where $G_{HS,0}$ is the connected component of the identity in $G_{HS}$. So in particular, $\E_n(A)$ is a Banach-Lie group with Banach-Lie algebra $L$. We have already shown that $\E_n(A)\subseteq G_{HS}$. Since $\E_n(A)$ is connected, we have actually $\E_n(A)\subseteq G_{HS,0}$. Conversely, any element $g\in G_{HS,0}$ is of the form $\exp(X_1)\cdots\exp(X_n)$, for some $X_1,\ldots,X_n\in \tilde{L}$. Since $\tilde{L}=\LA$, for each $i$, $\exp(X_i)\in \E_n(A)$, thus $g\in \E_n(A)$, and we are done.\\ \\
\textbf{Step 5b.} Suppose that $K_0(A)$ is finitely generated. Since $K_0(A)$, and thus also $\tau_*[K_0(A)]$, are finitely generated, $F$, the closed real linear span of $\tau_*[K_0(A)]$, is a finite dimensional subspace. It follows that $L$ has a finite codimension in $\hat{L}$, and since $L\subseteq\LA\subseteq\tilde{L}\subseteq \hat{L}$, also $\LA$ has a finite codimension in $\tilde{L}$. It follows from \cite[Lemma 1.4]{Omori} that $\E_n(A)$ is a Banach-Lie subgroup of the Banach-Lie group $G_{HS}$.

\end{proof}

Finally we address the question whether the groups $\E_n(A)$ are unbounded. We show this to be the case for all $n\geq 2$ and $A$ regardless whether $\E_n(A)$ is Banach-Lie, or not.

For the rest of this subsection, we choose and fix $(A,\|\cdot\|_A)$, a unital Banach algebra, and $n\geq 2$. Let $(A^n,\|\cdot\|_n)$ be the direct sum of $n$ copies of $A$ with the $\ell^1$-sum norm. We consider two norms on the algebra $M_n(A)$. For $X=(X_{i,j})_{i,j\leq n}$, set \[\|X\|_\infty:=\max_{i,j\leq n} \|X_{i,j}\|_A\] and set \[\|X\|:=\sup_{\{\xi\in A^n\mid \|\xi\|_n\leq 1\}} \|X\xi\|_n.\] The latter is a Banach algebra norm on $M_n(A)$. However, we clearly have the following.
\begin{fact}\label{fact:twonorms}
The norms $\|\cdot\|$ and $\|\cdot\|_\infty$ are equivalent.
\end{fact}
We prove the following more general criterion which could be useful also elsewhere.
\begin{proposition}\label{prop:unboundedEnA}
Let $G$ be a closed subgroup of a connected component of the identity of $\GL(n,A)$. If there exists a sequence $(X_m)_{m=1}^{\infty}\subseteq G$ such that $\|X_m\|_\infty\to\infty$, then $G$ is unbounded.
\end{proposition}
\begin{proof}
Let $G$ and $(X_m)_{m=1}^{\infty}\subseteq G$ be as in the statement. Denote $\GL(n,A)$ by $H$. It suffices to find a continuous length function $l$ on $G$ such that $l(X_m)\to \infty$. Since $H$ is a Banach-Lie group, $\bel{H}$ is a compatible maximal length function on $H$. The restriction of $\bel{H}$ to $G$ is clearly a continuous (even compatible) length function on $G$, so it suffices to show that $\bel{H}(X_m)\to \infty$. For each $m$, find $x^m_1,\ldots,x^m_{k_m}\in M_n(A)$ such that $X_m=\prod_{k\leq k_m} \exp(x^m_k)$ and $\bel{H}(X_m)\geq \sum_{k=1}^{k_m} \|x^m_k\|-1$. Then we have \[\begin{split}\bel{H}(X_m)\geq \sum_{k=1}^{k_m} \|x^m_k\|-1=\log\Big(\exp\big(\sum_{k=1}^{k_m} \|x^m_k\|\big)\Big)-1\geq\\ \log\Big(\|\big(\prod_{k=1}^{k_m} \exp(x^m_k)\big)\|\Big)-1= \log\big(\|X_m\|\big)-1, 
\end{split}\]
where the second inequality follows from
\[\|\left(\prod_{k=1}^{k_m} \exp(x^m_k)\right)\|_A\leq \exp\left(\sum_{k=1}^{k_m} \|x^m_k\|_A\right).\]

Since by Fact~\ref{fact:twonorms}, $\|X_m\|_\infty\to\infty$ implies $\|X_m\|\to\infty$, which in turn implies $\log\big(\|X_m\|\big)-1\to\infty$, the previous inequalities show that $\bel{H}(X_m)\to\infty$. This finishes the proof.
\end{proof}

\begin{corollary}
For every unital Banach algebra $A$ and $n\geq 2$, the group $\E_n(A)$ is unbounded.
\end{corollary}
\begin{proof}
The group $\E_n(A)$ contains the sequence $X_m:=E_{1,2}(m1)$. Since we have $\|X_m\|_\infty\to\infty$, we are done by Proposition~\ref{prop:unboundedEnA}.
\end{proof}
\subsection{Property (T) of groups $\E_n(A)$ and $\SL(n,A)$}
This subsection contains our main results concerning Property (T) for the groups of the type $\E_n(A)$ and $\SL(n,A)$. Partial results in this direction have been obtained, as discussed earlier in Subsection~\ref{subsection:prelim-T}, by Shalom in \cite{Sha99} and Cornulier in \cite{Cor06} using the bounded generation property. Here we show that much more can be said using just the techniques of the proof of Property (T) for $\SL(n,\Com)$, resp. $\SL(n,\Rea)$, and the Mautner phenomenon. In particular, instead of the special case of the ring, or rather algebra, $C(X)$, we consider arbitrary unital Banach algebras, not necessarily abelian or separable - when dealing with groups $\E_n(A)$. At the end of the subsection,  we also discuss Property (T) for $\SL(n,A)$, when $A=C(X)$ for $X$ compact Hausdorff. The main result is the following.
\begin{theorem}\label{thm:propertyT}
Let $A$ be a unital Banach algebra over the real or complex numbers and let $n\geq 3$. Then the group $\E_n(A)$ has Property (T), and therefore also Property (FH). 
\end{theorem}
\begin{proof}
Fix a unital Banach algebra $A$. We shall only prove the theorem for $\E(3,A)$. For the other groups from the statement, the proof is analogous.\medskip

Notice that we can naturally identify $\GL(n,\Com)$, resp. $\SL(n,\Com)$ with a subgroup of $\GL(n,A)$, resp. $\E_n(A)$. Indeed, since $A$ is unital, a matrix $(a_{i,j})_{i,j}\in\GL(n,\Com)$, where each $a_{i,j}\in\Com$ can be identified with a matrix $(a_{i,j}\cdot 1)_{i,j}\in\GL(n,A)$, where $1$ is the unit of $A$. Since $\SL(n,\Com)$ is generated by the groups of elementary matrices $E_{i,j}(\Com)$, it follows it can be identified with a subgroup of $\E_n(A)$.

Let $\pi: \E_3(A)\rightarrow \U(\Hil)$ be a continuous unitary representation that almost has invariant vectors. There is a standard embedding of the semi-direct product $\Com^2\rtimes \SL(2,\Com)$ into $\SL(3,\Com)$ which sends an element $\Big(\begin{pmatrix} a_1\\
a_2
\end{pmatrix}, \begin{pmatrix}b_1 & b_2\\
b_3 & b_4
\end{pmatrix}\Big)\in \Com^2\rtimes \SL(2,\Com)$ to the element $\begin{pmatrix} b_1 & b_2 & a_1\\
b_3 & b_4 & a_2\\
0 & 0 & 1
\end{pmatrix}\in \SL(3,\Com)$, which is further sent, or identified, with the corresponding element $\begin{pmatrix} b_1\cdot 1 & b_2\cdot 1 & a_1\cdot 1\\
b_3\cdot 1 & b_4\cdot 1 & a_2\cdot 1\\
0 & 0 & 1
\end{pmatrix}\in \E_3(A)$. Restricting the representation $\pi$ to this copy of $\Com^2\rtimes \SL(2,\Com)$, which therefore also almost has invariant vectors, and using the result (see e.g. \cite[Corollary 1.4.13]{BdHV}) that the pair $(\Com^2\rtimes \SL(2,\Com),\Com^2)$ has the relative Property (T) (recall the definition from \cite[Definition 1.4.3]{BdHV}), we get that there is an nonzero $\Com^2$-invariant vector $\xi\in\Hil$ (with respect to this embedding of $\Com^2$ into $\E_3(A)$). Notice that $\xi$ is therefore $E_{1,3}(\Com)$-invariant and $E_{2,3}(\Com)$-invariant. Consider the following subgroups of $\E_3(A)$: $A_1:=\{a_1(\lambda)=\begin{pmatrix} \lambda\cdot 1 & 0 & 0\\
0 & 1 & 0\\
0 & 0 & \lambda^{-1}\cdot1\end{pmatrix}\mid \lambda\in\Com\setminus\{0\}\}$ and $A_2:=\{a_2(\lambda)=\begin{pmatrix}  1 & 0 & 0\\
0 & \lambda\cdot 1 & 0\\
0 & 0 & \lambda^{-1}\cdot 1\end{pmatrix}\mid \lambda\in\Com\setminus\{0\}\}$. 

Let us show that $\xi$ is also $A_1$-invariant and $A_2$-invariant. First we do the former. Let $\lambda\in\Com\setminus\{0\}$ be arbitrary. Fix a sequence $(\lambda_i)_i\subseteq\Com\setminus \{0\}$ with $|\lambda_i|\to 0$. We have \[\begin{pmatrix}1 & 0 & \lambda\lambda_i^{-1}\cdot 1\\
0 & 1 & 0\\
0 & 0 & 1\end{pmatrix}
\begin{pmatrix}0 & 0 & -\lambda^{-1}_i\cdot 1\\
0 & 1 & 0\\
\lambda_i\cdot 1 & 0 & 0\end{pmatrix}
\begin{pmatrix}1 & 0 & (\lambda\lambda_i)^{-1}\cdot 1\\
0 & 1 & 0\\
0 & 0 & 1\end{pmatrix}=\]
\[\begin{pmatrix}\lambda\cdot 1 & 0 & 0\\
0 & 1 & 0\\
\lambda_i\cdot 1 & 0 & \lambda^{-1}\cdot 1\end{pmatrix}\to
\begin{pmatrix}\lambda\cdot 1 & 0 & 0\\
0 & 1 & 0\\
0 & 0 & \lambda^{-1}\cdot 1\end{pmatrix}.\]
\medskip

Denoting the matrix $\begin{pmatrix}0 & 0 & -\lambda^{-1}_i\\
0 & 1 & 0\\
\lambda_i & 0 & 0\end{pmatrix}$ by $u_i$ (notice that $u_i\in\SL(3,\Com)$ and therefore $u_i\in \E_3(A)$), by the continuity of $\pi$ we get that \[\|\pi(a_1(\lambda))\xi-\xi\|=\]
\[\lim_i \| \pi(E_{1,3}(\lambda\lambda^{-1}_i\cdot 1))\pi(u_i)\pi(E_{1,3}((\lambda\lambda_i)^{-1}\cdot 1))\xi-\xi\|=\]
\[\lim_i \|\pi(u_i)\xi-\xi\|,\]

where the last equality follows from the fact that for $g,h,f\in \E_3(A)$ with $\|\pi(g^{\pm})\xi-\xi\|=\|\pi(f^{\pm})\xi-\xi\|=0$ we have
\begin{equation}\label{(T)-eq}
    \|\pi(h)\xi-\xi\|\leq \|\pi(g^{-1})\xi-\xi\|+\|\pi(g)\pi(h)\pi(f)\xi-\xi\|+\|\pi(f^{-1})\xi-\xi\|=\|\pi(g)\pi(h)\pi(f)\xi-\xi\|,
\end{equation}
and symetrically also $\|\pi(g)\pi(h)\pi(f)\xi-\xi\|\leq \|\pi(h)\xi-\xi\|$.

Since the sequence $(u_i)_i$ does not depend on $\lambda$, we get \[\lim_i\|\pi(u_i)\xi-\xi\|=\|\pi(a_1(1))\xi-\xi\|=0,\] thus for every $\lambda$ we have $\|\pi(a_1(\lambda))\xi-\xi\|=0$.

An analogous computation using $E_{2,3}(\lambda\lambda^{-1}_i\cdot 1)$, resp. $E_{2,3}((\lambda\lambda_i)^{-1}\cdot 1)$ instead of $E_{1,3}(\lambda\lambda^{-1}_i\cdot 1)$, resp. $E_{1,3}((\lambda\lambda_i)^{-1}\cdot 1)$, and $\begin{pmatrix}1 & 0 & 0\\
0 & 0 & -\lambda^{-1}_i\cdot 1\\
0 & \lambda_i\cdot 1 & 0\end{pmatrix}$ instead of $u_i$ shows that $\pi(a_2(\lambda))\xi=\xi$, for every $\lambda\in\Com\setminus\{0\}$.
\medskip

Now we show that for every $a\in A$, $\pi(E_{1,3}(a))\xi=\pi(E_{2,3}(a))\xi=\pi(E_{3,1}(a))\xi=\pi(E_{3,2}(a))\xi=\xi$. Fix $a\in A$. Since 

\[\lim_i a_1(\lambda_i)E_{1,3}(a)a_1(\lambda^{-1}_i)=\lim_i \begin{pmatrix}1 & 0 & \lambda_i^2 a\\
0 & 1 & 0\\
0 & 0 & 1\end{pmatrix}=\mathrm{Id},\]

\[\lim_i a_2(\lambda_i)E_{2,3}(a)a_2(\lambda^{-1}_i)=\lim_i \begin{pmatrix}1 & 0 & 0\\
0 & 1 & \lambda_i^2 a\\
0 & 0 & 1\end{pmatrix}=\mathrm{Id},\]

\[\lim_i a_2(\lambda^{-1}_i)E_{3,2}(a)a_2(\lambda_i)=\lim_i \begin{pmatrix}1 & 0 & 0\\
0 & 1 & 0\\
0 & \lambda_i^2 a & 1\end{pmatrix}=\mathrm{Id}\]
and
\[\lim_i a_1(\lambda^{-1}_i)E_{3,1}(a)a_1(\lambda_i)=\lim_i \begin{pmatrix}1 & 0 & 0\\
0 & 1 & 0\\
\lambda_i^2 a & 0 & 1\end{pmatrix}=\mathrm{Id},\]

we obtain \[0=\lim_i \|\pi(a_1(\lambda_i)E_{1,3}(a)a_1(\lambda^{-1}_i))\xi-\xi\|=\|\pi(E_{1,3}(a))\xi-\xi\|,\] where for the last equality we again used \eqref{(T)-eq}, and analogously that $\|\pi(E_{3,1}(a))\xi-\xi\|=\|\pi(E_{2,3}(a))\xi-\xi\|=\|\pi(E_{3,2}(a))\xi-\xi\|=0$.
\medskip

Finally, since the groups $E_{2,1}(A)$, resp. $E_{1,2}(A)$ is a subgroup of a group generated by $E_{2,3}(A)\cup E_{3,1}(A)$, resp. of a group generated by $E_{1,3}(A)\cup E_{3,2}(A)$, we conclude that $E_{i,j}(A)\xi=\xi$, for all $i\neq j\leq 3$. Since $E_3(A)$ is generated by $\bigcup_{i\neq j\leq 3} E_{i,j}(A)$, we obtain that $\xi$ is an invariant vector of $\E_3(A)$ under the representation $\pi$.
\end{proof}

Now we consider the groups $\SL(n,A)$, when $A$ is an abelian unital Banach algebra. The following corollary generalizes \cite[Corollary 2]{Cor06}.
\begin{corollary}
Let $A$ be an abelian unital Banach algebra. Then $\SL(n,A)$ has Property (T) if and only if the discrete group $\SL(n,A)/(\SL(n,A))_0$ does.
\end{corollary}
\begin{proof}
By Lemma~\ref{lem:EnInSln}, the connected component of the identity $(\SL(n,A))_0$ is equal to $\E_n(A)$, which by Theorem~\ref{thm:propertyT} has Property (T). So if the quotient group $\SL(n,A)/(\SL(n,A))_0$ has Property (T), then so does $\SL(n,A)$ since (T) is preserved under (certain) group extensions; see \cite[Proposition 1.7.6]{BdHV} (notice that the proposition is stated only for locally compact groups, but works as well for completely metrizable groups as mentioned in \cite[Remark 1.7.9]{BdHV}). Conversely, if $\SL(n,A)$ has Property (T), then so does every quotient of it, in particular $\SL(n,A)/(\SL(n,A))_0$.
\end{proof}

At the end, we specialize to abelian unital $C^*$-algebras. Let $A$ be such an algebra and let $X$ be its Gelfand spectrum, i.e. $A=C(X)$. As mentioned earlier, in this case we can identify $\SL(n,A)$ with $C(X,\SL(n,\Com))$ (analogously for the real case which is left to the reader). Notice that two elements of $\SL(n,A)$ lie in the same connected component if and only if they lie in the same path connected component if and only if the corresponding continuous functions from $X$ are homotopic. It follows that $\SL(n,A)/(\SL(n,A))_0=\SL(n,A)/\E_n(A)=[X,\SL_n(\Com)]=[X,\mathrm{SU}(n)]$. Recall that for topological spaces $Y$ and $Z$, the symbol $[Y,Z]$ denotes the set of homotopy classes of continuous maps from $Y$ to $Z$. For the last equality, recall that topologically, applying the polar decomposition, $\SL(n,\Com)$ is homeomorphic to a direct product of $\mathrm{SU}(n)$ and a contractible space. 

Summarizing the discussion, we obtain the following.
\begin{proposition}\label{prop:TforSLnA}
 Let $A$ be an abelian unital $C^*$-algebra and let $X$ be its Gelfand spectrum. Then $\SL(n,A)$ has Property (T) if and only if $[X,\mathrm{SU}(n)]$ does.
\end{proposition}
The previous proposition can be used to decide Property (T) of $\SL(n,C(X))$ for a substantial class of compact Hausdorff spaces based on their homology/cohomology properties.

First we mention, as already suggested in \cite[Exercise 4.4.10]{BdHV}, that for all $n\geq 3$, $\SL(n,C(S^3))$ does not have Property (T). Indeed, in this case we have \[\SL(n,C(S^3))/\E_n(C(S^3))=[S^3,\mathrm{SU}(n)]=\pi_3(\mathrm{SU}(n))=\Int,\] where we refer to \cite[Appendix A (VII)]{Encyclopedia} for the last equality.

On the other hand, if $X$ is a `nice' compact space, then Property (T) for $\SL(n,C(X))$ is related to Betti numbers of $X$. In the following, in order to illustrate this idea, we consider for simplicity just the case $n=3$. The proof gives a hint how to analogously proceed for higher $n$.
\begin{corollary}
Let $X$ be a compact Hausdorff space of dimension at most $15$ which is homotopically equivalent to a finite CW complex. Suppose that
\begin{itemize}
    \item either the cohomology groups $H^k(X,\Int)$ are finite for all $k\leq \mathrm{dim}(X)$,
    \item or $H^k(X,\Int)=0$, for $k\in \{3,5,6,8-15\}\cap \{1,\ldots,\mathrm{dim}(X)\}$.
\end{itemize}
Then $\SL(n,C(X))$ has Property (T).
\end{corollary}
\begin{proof}
In both cases we apply \cite[Proposition 8.2.4]{Ark} to $[X,Y]=[X,\mathrm{SU}(3)]$ to get that $[X,\mathrm{SU}(3)]$ is finite, thus it has Property (T), and we are done by Proposition~\ref{prop:TforSLnA}.

The condition on $Y=\mathrm{SU}(3)$ is satisfied since $\mathrm{SU}(3)$ is simply connected and $\pi_k(\mathrm{SU}(3))$ is finitely generated for $k\leq 15$ (see again \cite[Appendix A (VII)]{Encyclopedia}).

If $H^k(X,\Int)=0$, for $k\in \{3,5,6,8-15\}$, then the assumption of the whole \cite[Proposition 8.2.4]{Ark} is satisfied since $\pi_2(\mathrm{SU}(3))=\pi_4(\mathrm{SU}(3))=\pi_7(\mathrm{SU}(3))=0$ (see again \cite[Appendix A (VII)]{Encyclopedia}).

If the cohomology groups $H^k(X,\Int)$ are finite, for all $k\leq \mathrm{dim}(X)$, then the conclusion again holds by the remark following the proof of \cite[Proposition 8.2.4]{Ark}.
\end{proof}

We remark that we do not know whether it can ever happen that the discrete group $[X,\mathrm{SU}(n)]$ is infinite and has Property (T), although it can happen that this group is non-abelian. If there were $X$ with $[X,\mathrm{SU}(n)]$ infinite and having Property (T), it would have to be rather wild since, as pointed out to us by V. Pestov, by the Whitehead theorem, if $X$ is of finite category, in particular a finite-dimensional CW complex, then $[X,\mathrm{SU}(n)]$ is nilpotent (see \cite[Theorem 8.4.9]{Ark}).
\subsection{The groups $\E_2(A)$ and $\SL(2,A)$}
While the groups $\SL(n,K)=\E_n(K)$, for $n\geq 3$ and $K\in\{\Rea,\Com\}$, have Property (T), the group $\SL(2,K)=\E_2(K)$ has the opposite Haagerup property. Since the proof of Property (T) passes to the groups $\E_n(A)$, for $n\geq 3$, it is natural to consider the groups $\E_2(A)$ and $\SL(2,A)$ as candidates for groups with the Haagerup property. The aim of this subsection is to show that these groups actually fail the Haagerup property.
We will use the following result. 
\begin{theorem}\label{thm: SUR for Banach}
Let $X$ be a separable Banach space. Then the following four conditions are equivalent. 
\begin{list}{}{}
\item[{\rm (i)}] $X$ is strongly unitarily representable {\rm (SUR)}. Namely, $X$ is isomorphic as a topological group to a closed subgroup of the unitary group on a Hilbert space. 
\item[{\rm (ii)}] There exists a probability space $(\Omega,\mu)$ such that $X$ is linearly isomorphic to a subspace of $L^0(\Omega,\mu)$, where the latter is equipped with the topology of convergence in measure.
\item[{\rm (iii)}] $X$ admits a uniform embedding into a Hilbert space. 
\item[{\rm (iv)}] $X$ admits a coarse embedding into a Hilbert space. 
\end{list}
If moreover $X$ is of the form $C(Y)$, where $Y$ is a compact metric space, then the above four conditions are equivalent to $Y$ being finite. 
\end{theorem}
\begin{proof}
The equivalence (i)$\Leftrightarrow$(ii) is proved by Aharoni--Maurey--Mityagin in Proposition 3.2 and Corollary 3.6 of \cite{AMM85}, and (ii)$\Leftrightarrow$(iii) is proved in Theorem 4.1 of the same paper. That these conditions are equivalent to (iv) is due to Randrianarivony \cite{Ra06} using the work of Johnson--Randrianarivony \cite{JR06}. The last claim is proved in \cite[Theorem 3.3]{AM20} in a recent work of two of the authors.
\end{proof}
\begin{theorem}\label{thm: SL(2,A)}
Let $A$ be a unital infinite-dimensional ${\rm C}^*$-algebra. 
Then the group $\E_2(A)$ does not have the Haagerup property. 
\end{theorem}
For the proof we need the following lemma. 
\begin{lemma}\label{lem: D_+} Let $A$ be as in Theorem \ref{thm: SL(2,A)}. 
Suppose $B$ is a separable infinite-dimensional unital abelian {\rm C}$^*$-subalgebra of $A$ and define the abelian subgroup $D_+ :=\{{\rm diag}(d,d^{-1})\mid d\in \GL(B)_+\}$ of $\E_2(A)$. Then the following statements hold. 
\begin{list}{}{}
    \item[{\rm{(i)}}] For each $a\in D_+$, there exists a unique $c\in B_{\rm{sa}}$ such that $a=\exp (h), h={\rm diag}(c,-c)$. 
    \item[{\rm{(ii)}}] In {\rm (i)}, $\bel{\GL(2,A)}(a)=\|\log a\|=\bel{D_+}(a)$ holds. In particular, the inclusion $D_+\subset \E_2(A)$ is a coarse embedding of topological groups. 
\end{list}
 
\end{lemma}
\begin{proof}
(i) Suppose $a=\text{diag}(d,d^{-1}),\,d\in \GL(B)_+$. Let $\alpha=\inf \sigma(d)>0$. Regard $\log \mid [\alpha,\|d\|]\to \mathbb{R}$. Then $c=\log (d)\in B_{\rm{sa}}$ and $d=\exp (c)$. Such $c$ is unique by functional calculus. Moreover, $a=\exp (h)$ and $h={\rm diag}(\log d, -\log d)$ holds.\medskip

(ii) Notice first that $D_+$ is a Banach-Lie group. Indeed, since $B$ is unital, it is of the form $C(X)$, for some compact Hausdorff space $X$. It follows that $D_+$ is isomorphic to a connected component of the identity of the group $\GL(B_{\rm{sa}})=\GL(C(X,\Rea))$, which is Banach-Lie by \cite[Proposition IV.9]{Neeb04}. The same reference also implies that $\GL(2,A)$ is Banach-Lie.

Since $a=e^{\log a}$, it is clear that $\bel{\GL(2,A)}(a)\le \bel{D_+}(a)\le \|\log a\|$ holds. Suppose that $x_1,\dots, x_n$ are elements in the Lie algebra $M_2(A)$ of $\GL(2,A)$ which satisfy $a=e^{x_1}\cdots e^{x_n}$. 
Then by $\|e^x\|\le e^{\|x\|}\,(x\in M_2(A))$, we have 
\[e^{\|\log a\|}=\|e^{\log a}\|=\|e^{x_1}\cdots e^{x_n}\|\le e^{\|x_1\|}\cdots e^{\|x_n\|}.\]
Therefore $\|\log a\|\le \sum_{k=1}^n\|x_k\|$. Since this holds for arbitrary expression of $a$ as the products of exponentials in $\GL(2,A)$, we obtain $\|\log a\|\le \bel{\GL(2,A)}(a)$. Therefore $\bel{\GL(2,A)}(a)=\bel{D_+}(a)=\|\log a\|$ holds.\\
Since the exponential length is a maximal compatible metric (Proposition \ref{prop:bel} (iii)), it follows that the inclusion $D_+\subset \GL(2,A)$ is a coarse embedding. This further implies that the inclusion $D_+\subset \E_2(A)$ is a coarse embedding. Indeed, this is witnessed by the restriction of $\bel{\GL(2,A)}$ to $\E_2(A)$.
\end{proof}
\begin{proof}[Proof of Theorem \ref{thm: SL(2,A)}]
Assume by contradiction that $\E_2(A)$ has the Haagerup property. We can find a separable infinite-dimensional unital abelian subalgebra $B$ of $A$. Indeed, since $A$ is infinite-dimensional, it contains an infinite-dimensional separable unital $C^*$-subalgebra. We can then take as $B$ its maximal abelian $C^*$-subalgebra; see e.g. \cite[Proposition 4.40]{AM20}.  
By Lemma \ref{lem: D_+}, $D_+=\{{\rm diag}(d,d^{-1})\mid d\in {\rm GL}_+(B)\}$ is a coaserly embedded closed subgroup of $\E_2(A)$. Therefore the restriction of a coarsely proper affine isometric action of $\E_2(A)$ on a Hilbert space to $D_+$ witnesses the Haagerup property of $D_+$. In this case, $D_+$ is coarsely embeddable into a Hilbert space. The map $B_{\rm sa}\ni h\mapsto {\rm diag}(e^h, e^{-h})\in D_+$ gives a topological group isomorphism, where $B_{\rm sa}$ is regarded as an additive Polish group in the norm topology. This implies that the real Banach space $B_{\rm sa}$ coarsely embeds into a Hilbert space. By Theorem~\ref{thm: SUR for Banach}, this implies that the Gelfand spectrum ${\rm Spec}(B)$ of $B$ must be finite. This is a contradiction since $B$ is infinite-dimensional. Therefore the conclusion follows.
\end{proof}

\begin{corollary}
Let $A$ be a separable infinite-dimensional unital abelian $C^*$-algebra. Then $\SL(2,A)$ does not have the Haagerup property.
\end{corollary}
\begin{proof}
By Lemma~\ref{lem:EnInSln}, $\E_2(A)$ is the connected component of the identity of $\SL(2,A)$. Since $A$ is separable, the discrete quotient $\SL(2,A)/\E_2(A)$ is at most countable, and therefore the inclusion $\E_2(A)\hookrightarrow \SL(2,A)$ is a coarse embedding by Corollary~\ref{cor:G0coarseinG}. It follows that if $\SL(2,A)$ admitted a continuous coarsely proper action on a Hilbert space, the restriction of the action to the subgroup $\E_2(A)$ would be coarsely proper, which would contradict Theorem~\ref{thm: SL(2,A)}.
\end{proof}

\noindent{\bf Acknowledgements.} We would like to thank to Gabriel Larotonda for answering our questions on the compatibility of the Finsler metrics on Banach-Lie groups. We are also grateful to Bruno Duchesne, Vladimir Pestov, and Christian Rosendal for their comments on an earlier version of the paper, and especially to Rosendal for sharing his question on the existence of minimal metrics whose answer resulted in a part of Theorem A. 
\bibliographystyle{siam}
\bibliography{references}
\end{document}